\documentclass[a4paper,12pt]{amsart}
\usepackage{fullpage}
\usepackage{amsmath,amssymb,graphicx,psfrag}
\usepackage{color,ifthen}
\usepackage{moreverb}
  
\title{An abstract analysis of\\ optimal goal-oriented adaptivity}
\author{Michael Feischl}
\author{Dirk Praetorius}
\address{Vienna University of Technology, Institute for Analysis and Scientific Computing, Wiedner Hauptstr.\ 8--10, 1040 Vienna, Austria}
\email{\{Michael.Feischl\,,\,Dirk.Praetorius\}@tuwien.ac.at}

\author{Kristoffer George van der Zee}
\address{University of Nottingham, School of Mathematical Sciences, University Park, Nottingham NG7 2RD, UK}
\email{KG.vanderZee@nottingham.ac.uk}

\thanks{{\bf Acknowledgement.} The authors MF and DP acknowledge support through the Austrian Science Fund (FWF) under grant P27005 \emph{Optimal adaptivity for BEM and FEM-BEM coupling} as well as 
through the FWF doctoral school \emph{Dissipation and Dispersion in Nonlinear PDEs}, funded under grant W1245. KVDZ acknowledges support from the Engineering and Physical Sciences Research Council (EPSRC) grant EP/I036427/1.}

\def\A{\mathbb A}
\def\N{\mathbb N}
\def\R{\mathbb R}
\def\T{\mathbb T}

\def\OO{\mathcal O}
\def\MM{\mathcal M}
\def\SS{\mathcal S}
\def\TT{\mathcal T}
\def\PP{\mathcal P}
\def\RR{\mathcal R}
\def\XX{\mathcal X}

\def\slp{\mathcal V}
\def\dlp{\mathcal K}
\def\diam{{\rm diam}}
\def\refine{{\tt refine}}
\def\qed{\hfill$\blacksquare$}
\def\eps{\varepsilon}

\newcommand{\norm}[3][]{#1\|#2#1\|_{#3}}
\newcommand{\enorm}[2][]{#1|\!#1|\!#1|#2#1|\!#1|\!#1|}
\newcommand{\dual}[3][]{#1\langle#2\,,\,#3#1\rangle}
\newcommand{\adual}[3][]{a#1(#2\,,#3#1)}

\newcommand{\set}[3][]{#1\{#2\,:\,#3#1\}}
\newcommand{\dist}[4][]{{\rm d\!l}_{#2}#1(#3\,,#4#1)}

\newcounter{const}
\def\cname#1#2{%
\ifthenelse{\equal{#1}{rel}}{C_{\rm rel}}{%
\ifthenelse{\equal{#1}{min}}{C_{\rm mark}}{%
\ifthenelse{\equal{#1}{min2}}{C_{\rm mark}^\prime}{%
\ifthenelse{\equal{#1}{rlinear}}{C_{\rm lin}}{%
\ifthenelse{\equal{#1}{final1}}{C_{\rm opt}}{%
\ifthenelse{\equal{#1}{final2}}{C_{\rm opt2}}{%
\ifthenelse{\equal{#1}{orth}}{C_{\rm orth}}{%
\ifthenelse{\equal{#1}{reliable}}{C_{\rm rel}}{%
\ifthenelse{\equal{#1}{stable}}{C_{\rm stb}}{%
\ifthenelse{\equal{#1}{reduction}}{C_{\rm red}}{%
\ifthenelse{\equal{#1}{mon}}{C_{\rm mon}}{%
\ifthenelse{\equal{#1}{estred}}{C_{\rm est}}{%
\ifthenelse{\equal{#1}{linear}}{C_{\rm conv}}{%
\ifthenelse{\equal{#1}{nvb}}{C_{\rm mesh}}{%
\ifthenelse{\equal{#2}{_neW_}}{\refstepcounter{const}\label{const:#1}C_{\arabic{const}}}{C_{\ref{const:#1}}}%
}}}}}}}}}}}}}}}
\def\newc#1{\cname{#1}{_neW_}}
\def\c#1{\cname{#1}{_ref_}}

\newcounter{qconst}
\def\qname#1#2{%
\ifthenelse{\equal{#1}{rlinear}}{q_{\rm lin}}{%
\ifthenelse{\equal{#1}{reduction}}{q_{\rm red}}{%
\ifthenelse{\equal{#1}{estred}}{q_{\rm est}}{%
\ifthenelse{\equal{#1}{contraction}}{q_{\rm ctr}}{%
\ifthenelse{\equal{#1}{gamma}}{q_{\rm wht}}{%
\ifthenelse{\equal{#1}{linear}}{q_{\rm conv}}{%
\ifthenelse{\equal{#2}{_neW_}}{\refstepcounter{const}\label{qconst:#1}q_{\arabic{qconst}}}{q_{\ref{qconst:#1}}}%
}}}}}}}
\def\newq#1{\qname{#1}{_neW_}}
\def\q#1{\qname{#1}{_ref_}}

\newtheorem{theorem}{Theorem}
\newtheorem{algorithm}{Algorithm}

\newtheorem{remark}[theorem]{Remark}
\newtheorem{lemma}[theorem]{Lemma}
\newtheorem{proposition}[theorem]{Proposition}
\newtheorem{corollary}[theorem]{Corollary}

\renewcommand{\subsection}[1]{\refstepcounter{subsection}\medskip{\bf\thesubsection.~#1.}}

\date{\today}
\subjclass[2010]{65N30, 65N50, 65Y20, 41A25}
\keywords{adaptivity, goal-oriented algorithm, quantity of interest, convergence, optimal convergence rates, finite element method, boundary element method}

\begin{document}

\begin{abstract}
We provide an abstract framework for optimal goal-oriented adaptivity for finite element methods and boundary element methods in the spirit of~\cite{axioms}. We prove that this framework covers standard discretizations of general second-order 
linear elliptic PDEs and hence generalizes available results~\cite{bet,ms} beyond the 
Poisson equation.
\end{abstract}

\maketitle

\section{Introduction}
\subsection{State of the art \& contributions}
This work considers the simultaneous adaptive control of two error estimators $\eta_{u,\star}$ and $\eta_{z,\star}$ which satisfy certain
abstract axioms from Section~\ref{section:axioms}, below. 
The estimator product $\eta_{u,\star}\eta_{z,\star}$ is designed to control the error in goal-oriented adaptivity and  allows to prove optimal error decay for the goal functional. This is discussed in Section~\ref{section:motivation} and demonstrated in Section~\ref{section:example}--\ref{section:bem}
for various model problems.
We analyze three adaptive mesh-refinement algorithms (Algorithm~\ref{algorithm},~\ref{algorithm:mod},~\ref{algorithm:bet}) which allow optimal convergence rates for the estimator product in the sense
that each possible algebraic convergence $r>0$ will be achieved, i.e.,
\begin{align*}
\eta_{u,\ell}\eta_{z,\ell}\lesssim (\#\TT_\ell-\#\TT_0)^{-r}\quad\text{for all }\ell\in\N,
\end{align*}
without any {\sl a~priori} knowledge of the optimal rate $r_{\rm opt}$.
Here, the triangulations $(\TT_\ell)_{\ell\in\N_0}$ are generated by the respective adaptive algorithm starting from some given initial triangulation $\TT_0$.

While standard adaptivity aims to approximate some unknown exact solution $u$ at optimal rate in the energy norm (see, e.g., \cite{ckns,ffp,stevenson} for adaptive FEM, \cite{part1,part2,fkmp,gantumur} for adaptive BEM, and \cite{axioms} for a recent overview on available results), goal-oriented adaptivity aims to
approximate, at optimal rate, only the functional value $g(u)$ (also called \emph{quantity of interest} in the literature).
Goal-oriented adaptivity is usually more important in practice than standard adaptivity. It has therefore attracted much interest also in the mathematical literature; see, e.g.,~\cite{BanRanBOOK2003,MR1430239,BecRanAN2001,MR1352472,MR2009374,MR1322810,MR1665351} for some prominent contributions.
However, as far as convergence and quasi-optimality of goal-oriented adaptivity is concerned, earlier results are only the two works~\cite{bet,ms} which are concerned with the Poisson model problem and the work~\cite{hp14} which considers general second-order linear elliptic PDEs, but is concerned with convergence only. Moreover, the analytical arguments of~\cite{bet,ms} are tailored to the Poisson equation and do not directly transfer to the more general setting of~\cite{hp14}. The quasi-optimality analysis for goal-oriented adaptivity has also been named as an open problem in the recent work~\cite{bn15}.
In view of this, the contributions and advances of the present work can be summarized as follows:
\begin{itemize}
\item We give an abstract analysis for optimal goal-oriented adaptivity which applies to general (non-symmetric) second-order linear elliptic PDEs in the spirit of~\cite{ffp} which even extends the problem class of~\cite{hp14}.
\item The analysis avoids any (discrete) efficiency estimate and thus allows for simple newest vertex bisection, while~\cite{bet,ms} require local bisec5-refinement in the spirit of~\cite{stevenson}.
\item Unlike~\cite{bet}, our proofs avoid any assumption on the resolution of the given data as, e.g., a saturation assumption~\cite[eq.~(4.4)]{bet}.
\item Unlike~\cite{hp14}, our analysis does not enforce the condition that the initial triangulation $\TT_0$ is sufficiently small, since we do not exploit the regularity of the dual solution to prove some crucial quasi-Galerkin orthogonality.
\item Finally, our analysis does not only cover the finite element method (FEM), but also applies to the boundary element method (BEM).
\end{itemize}
Related recent work includes~\cite{pointabem}, where the goal are point errors in symmetric BEM computations.
Although we shall verify the mentioned estimator axioms only for standard FEM and BEM discretizations, we expect that they can also be verified for discretizations in the frame of isogeometric analysis; see, e.g., \cite{MR3154029} for some goal-oriented adaptive IGAFEM.

\subsection{Goal-oriented adaptivity in the frame of the Lax-Milgram lemma}\label{section:motivation}
The following introduction covers the main application of the abstract theory, we have in mind. Let $\XX$ be a Hilbert space with 
norm $\norm\cdot{\XX}$, and let
$\adual\cdot\cdot:\XX\times\XX\to\R$ be a continuous and elliptic bilinear form on $\XX$.
For a given linear and continuous functional $f\in\XX^*$, let $u\in\XX$ be the unique solution of
\begin{align}\label{eq:primal}
 \adual{u}{v} = f(v)\quad\text{for all }v\in\XX.
\end{align}
Let $g\in\XX^*$ be the so-called \emph{goal functional}, i.e., we aim to approximate $g(u)$ at optimal rate.
To this end,  suppose that associated with
each triangulation $\TT_\star$ of some problem related domain $\Omega\subset \R^d$, there is a finite dimensional subspace $\XX_\star\subseteq\XX$.
Let $U_\star\in\XX_\star$ be the unique Galerkin approximation of $u$ which solves
\begin{align}\label{eq:primal:discrete}
 \adual{U_\star}{V_\star} = f(V_\star)\quad\text{for all }V_\star\in\XX_\star.
\end{align}
Furthermore, let $z\in\XX$ be the unique solution to the so-called dual problem
\begin{align}\label{eq:dual}
 \adual{v}{z} = g(v)\quad\text{for all }v\in\XX.
\end{align}
Then, for any $Z_\star\in\XX_\star$, it follows
\begin{align}\label{eq:goal error}
 |g(u)-g(U_\star)| = |\adual{u-U_\star}{z}| = |\adual{u-U_\star}{z-Z_\star}|
 \lesssim \norm{u-U_\star}\XX\,\norm{z-Z_\star}\XX.
\end{align}
Here and throughout, $\lesssim$ abbreviates $\le$ up to some generic multiplicative factor $C>0$ which is clear from the context, e.g., the hidden constant in~\eqref{eq:goal error} is the continuity bound of $a(\cdot,\cdot)$.
Suppose that we compute the unique Galerkin approximation $Z_\star\in\XX_\star$ of the dual solution $z\in\XX$, i.e.,
\begin{align}\label{eq:dual:discrete}
 \adual{V_\star}{Z_\star} = g(V_\star)\quad\text{for all }V_\star\in\XX_\star,
\end{align}
and that the Galerkin errors on the right-hand side of~\eqref{eq:goal error} can be controlled by computable 
{\sl a~posteriori} error estimators
\begin{subequations}\label{eq:reliable:laxmilgram}
\begin{align}
 \norm{u-U_\star}\XX &\lesssim \eta_{u,\star}:=\Big(\sum_{T\in\TT_\star}\eta_{u,\star}(T)^2\Big)^{1/2},\\
 \norm{z-Z_\star}\XX &\lesssim \eta_{z,\star}:=\Big(\sum_{T\in\TT_\star}\eta_{z,\star}(T)^2\Big)^{1/2}.
\end{align}
\end{subequations}
Under these assumptions, we are altogether led to
\begin{align}\label{eq:ansatz}
 |g(u)-g(U_\star)| \lesssim \eta_{u,\star}\,\eta_{z,\star}.
\end{align}
Overall, we thus aim for some adaptive algorithm which drives the computable upper bound 
on the right-hand side of~\eqref{eq:ansatz} to zero with optimal rate. 

\begin{remark}
Using the residual $R_{u,\star}(v):=f(v) - \adual{U_\star}{v} \in \XX^*$, one can improve
the hidden constants in~\eqref{eq:goal error} by
\begin{align*}
 |g(u)-g(U_\star)| 
 = |\adual{u-U_\star}{z-Z_\star}|
 \le \!\sup_{v\in\XX\backslash\{0\}}\!\!\frac{R_{u,\star}(v)}{\norm{v}{\XX}}\,\norm{z-Z_\star}{\XX}
 =: \norm{R_{u,\star}}{\XX^*}\,\norm{z-Z_\star}{\XX}.
\end{align*}
Some {\sl a~posteriori} error estimators (such as, e.g., estimators based on flux equilibration~\cite{v96}) allow for the reliability
estimate of the residual $\norm{R_{u,\star}}{\XX^*}\le\c{reliable}\,\eta_{u,\star}$
even with known constant $\c{reliable}=1$. The same arguments for the residual $R_{z,\star}(v):=g(v)-a(v,Z_\star)\in\XX^*$ yield $|g(u)-g(U_\star)|\leq \norm{R_{z,\star}}{\XX^*}\norm{u-U_\star}{\XX}$. \qed
\end{remark}

\subsection{Outline}
In Section~\ref{section:abstract}, we propose  three algorithms which are analyzed in the sections below, define the used mesh-refinement strategy, and outline the main result. Moreover, we provide the abstract framework in terms of four axioms for the estimators.
Section~\ref{section:optimal} proves optimal convergence rates for each adaptive algorithm. In Section~\ref{section:example} we apply the abstract theory to conforming goal-oriented FEM for second-order elliptic PDEs. Section~\ref{section:flux} covers goal-oriented FEM for the evaluation of some weighted boundary flux, whereas Section~\ref{section:bem} applies the abstract theory to goal-oriented adaptivity for BEM. The final Section~\ref{section:conclusion} discusses our results and points at extensions and open questions.

\section{Adaptive Algorithms for the Estimator Product}\label{section:abstract}

\noindent
We consider an adaptive algorithm which allows to drive the estimator product
\begin{align}\label{eq:abstract:problem1}
 \eta_{u,\star}\eta_{z,\star}:=\Big(\sum_{T\in\TT_\star}\eta_{u,\star}(T)^2\Big)^{1/2}\Big(\sum_{T\in\TT_\star}\eta_{z,\star}(T)^2\Big)^{1/2}
\end{align}
to zero with optimal rate.
This includes, in particular, the problem class from Section~\ref{section:motivation},
but also covers adaptive BEM for the approximation of point errors; see the recent own work~\cite{pointabem}. We suppose that each admissible triangulation $\TT_\star$ (see Section~\ref{section:mesh} below) allows for the computation of the error estimators $\eta_{u,\star}$ and $\eta_{z,\star}$, where the local contributions are (at least heuristically) linked to the elements $T\in\TT_\star$, cf.~\eqref{eq:abstract:problem1}.
To abbreviate notation, we shall write 
\begin{align*}
\def\UU{\mathcal U}
\eta_{w,\star} := \eta_{w,\star}(\TT_\star),\quad
\eta_{w,\star}(\UU_\star) := \Big(\sum_{T\in\UU_\star}\eta_{w,\star}(T)^2\Big)^{1/2}
\quad\text{for $w\in\{u,z\}$ and all }\UU_\star\subseteq\TT_\star.
\end{align*}

\subsection{Adaptive algorithm}
We consider two adaptive algorithms, which have been
proposed and analyzed in~\cite{ms} (Algorithm~\ref{algorithm}) and~\cite{bet} (Algorithm~\ref{algorithm:bet}) for goal-oriented adaptive FEM
for the Poisson problem, and propose a slight modification of the algorithm from~\cite{ms}
(Algorihm~\ref{algorithm:mod}) which is also related to that of~\cite{hp14}.
Note that all algorithms  differ only in the marking strategy: Algorithms~\ref{algorithm}--\ref{algorithm:mod} employ a separate D\"orfler 
marking in step~(iii)--(iv), whereas Algorithm~\ref{algorithm:bet} employs a combined D\"orfler marking
in step~(iv).

\begin{algorithm}\label{algorithm}
\textsc{Input:}  Initial triangulation $\TT_0$, marking parameter $0<\theta\le1$, and $\newc{min}\ge1$.\\
\textsc{Loop:} For all $\ell = 0,1,2,3,\dots$ do {\rm(i)--(vi)}:
\begin{itemize}
\item[\rm(i)] Compute refinement indicators $\eta_{u,\ell}(T)$ for all $T\in\TT_\ell$.
\item[\rm(ii)] Compute refinement indicators $\eta_{z,\ell}(T)$ for all $T\in\TT_\ell$.
\item[\rm(iii)] Determine a set $\MM_{u,\ell}\subseteq\TT_\ell$ of up to 
the multiplicative factor $\c{min}$ minimal 
cardinality such that
\begin{align}\label{doerfler:u}
 \theta\,\eta_{u,\ell}^2
 \le \eta_{u,\ell}(\MM_{u,\ell})^2.
\end{align}
\item[\rm(iv)] Determine a set $\MM_{z,\ell}\subseteq\TT_\ell$ of up to 
the multiplicative factor $\c{min}$ minimal 
cardinality such that
\begin{align}\label{doerfler:z}
 \theta\,\eta_{z,\ell}^2
 \le \eta_{z,\ell}(\MM_{z,\ell})^2.
\end{align}
\item[\rm(v)] Choose $\MM_\ell\in\{\MM_{u,\ell},\MM_{z,\ell}\}$ to be the
set of minimal cardinality.
\item[\rm(vi)] Let $\TT_{\ell+1}:=\refine(\TT_\ell,\MM_\ell)$ be the coarsest refinement of $\TT_\ell$ such that all marked elements $T\in\MM_\ell$ have been
refined.
\end{itemize}
\textsc{Output:} Sequence of successively refined triangulations $\TT_\ell$ and corresponding error estimators $\eta_{u,\ell},\eta_{z,\ell}$ for all $\ell\in\N_0$.\qed
\end{algorithm}

\begin{remark}
In the frame of Section~\ref{section:motivation}, the computation of $\eta_{u,\ell}$ in~{\rm (i)} and $\eta_{z,\ell}$ in~{\rm (ii)} usually requires to solve the primal~\eqref{eq:primal:discrete} and the dual problem~\eqref{eq:dual:discrete} to obtain the approximations $U_\ell$ resp. $Z_\ell$.\qed
\end{remark}

\begin{remark}
To construct a set $\MM_{u,\ell}$ with minimal cardinality (i.e., $\c{min}=1$) which satisfies the D\"orfler marking~\eqref{doerfler:u}, one needs to sort the refinement indicators. This results in logarithmic-linear costs. For $\c{min}=2$, the algorithmic construction of some set $\MM_{u,\ell}$ which satisfies~\eqref{doerfler:u}, is possible in linear complexity~\cite{stevenson}.
\qed
\end{remark}

Next, we propose a modified version of Algorithm~\ref{algorithm} which allows for more aggressive marking in step~(v), i.e., less adaptive steps. A similar but non-optimal algorithm has been proposed and tested in~\cite{hp14}.

\begin{algorithm}\label{algorithm:mod}
\textsc{Input:}  Initial triangulation $\TT_0$, marking parameter $0<\theta\le1$, and $\c{min},\newc{min2}\ge1$.\\
\textsc{Loop:} For all $\ell = 0,1,2,3,\dots$ do {\rm(i)--(vi)}:
\begin{itemize}
\item[\rm(i)] Compute refinement indicators $\eta_{u,\ell}(T)$ for all $T\in\TT_\ell$.
\item[\rm(ii)] Compute refinement indicators $\eta_{z,\ell}(T)$ for all $T\in\TT_\ell$.
\item[\rm(iii)] Determine a set $\MM_{u,\ell}\subseteq\TT_\ell$ of up to 
the multiplicative factor $\c{min}$ minimal 
cardinality such that
\begin{align}\label{doerfler:u:mod}
 \theta\,\eta_{u,\ell}^2
 \le \eta_{u,\ell}(\MM_{u,\ell})^2.
\end{align}
\item[\rm(iv)] Determine a set $\MM_{z,\ell}\subseteq\TT_\ell$ of up to 
the multiplicative factor $\c{min}$ minimal 
cardinality such that
\begin{align}\label{doerfler:z:mod}
 \theta\,\eta_{z,\ell}^2
 \le \eta_{z,\ell}(\MM_{z,\ell})^2.
\end{align}
\item[\rm(v)] Choose $\widetilde\MM_\ell\in\{\MM_{u,\ell},\MM_{z,\ell}\}$ to be the
set of minimal cardinality and choose $\widetilde\MM_\ell\subseteq \MM_\ell\subseteq \MM_{u,\ell}\cup\MM_{z,\ell}$ such that
$\#\MM_\ell\leq \c{min2} \#\widetilde\MM_\ell$.
\item[\rm(vi)] Let $\TT_{\ell+1}:=\refine(\TT_\ell,\MM_\ell)$ be the coarsest refinement of $\TT_\ell$ such that all marked elements $T\in\MM_\ell$ have been
refined.
\end{itemize}
\textsc{Output:} Sequence of successively refined triangulations $\TT_\ell$ and corresponding error estimators $\eta_{u,\ell},\eta_{z,\ell}$ for all $\ell\in\N_0$.\qed
\end{algorithm}

\begin{remark}
In our numerical experiments below, we choose $\MM_\ell$ as follows: Having picked $\widetilde\MM_\ell$ to be the minimal set amongst $\MM_{u,\ell}$ and $\MM_{z,\ell}$, we enlarge $\widetilde\MM_\ell$ by adding the largest $\#\widetilde\MM_\ell$ elements of the other set, e.g., if $\#\MM_{u,\ell} \le \#\MM_{z,\ell}$, then $\MM_\ell$ consists of $\MM_{u,\ell}$ plus the $\#\MM_{u,\ell}$ largest contributions of $\MM_{z,\ell}$. This yields $\c{min2} = 2$.\qed
\end{remark}

\begin{remark}
In~\cite{hp14}, the authors consider Algorithm~\ref{algorithm:mod}, but define $\MM_\ell := \MM_{u,\ell}\cup\MM_{z,\ell}$ in step~{\rm(v)}. While this also leads to linear convergence in the sense of Theorem~\ref{theorem:linear}, \cite{hp14} only proves suboptimal convergence rates $\min\{s,t\}$ instead of the optimal rate $s+t$ in Theorem~\ref{theorem:optimal:mod}; see~\cite[Section~4]{hp14}. We note that the strategy of~\cite{hp14} leads to linear convergence $\eta_{u,\ell+n}\le Cq^n\eta_{u,\ell}$ 
and $\eta_{z,\ell+n}\le Cq^n\eta_{z,\ell}$ for either estimator and all $\ell,n\in\N_0$, where $C>0$ and $0<q<1$ are independent constants, while the optimal strategies considered in this work only enforce linear convergence $\eta_{u,\ell+n}\eta_{z,\ell+n}\le Cq^n\eta_{u,\ell}\eta_{z,\ell}$ for the estimator product.\qed
\end{remark}

Finally, the following algorithm has been proposed in~\cite{bet}. 

\begin{algorithm}\label{algorithm:bet}
\textsc{Input:}  Initial triangulation $\TT_0$, marking parameter $0<\theta\le1$, and $\newc{min}\ge1$.\\
\textsc{Loop:} For all $\ell = 0,1,2,3,\dots$ do {\rm(i)--(vi)}:
\begin{itemize}
\item[\rm(i)] Compute indicators $\eta_{u,\ell}(T)$ for all $T\in\TT_\ell$.
\item[\rm(ii)] Compute indicators $\eta_{z,\ell}(T)$ for all $T\in\TT_\ell$.
\item[\rm(iii)] Assemble refinement indicators $\rho_\ell(T)^2 := \eta_{u,\ell}(T)^2\eta_{z,\ell}^2 + \eta_{u,\ell}^2\eta_{z,\ell}(T)^2$ for all $T\in\TT_\ell$.
\item[\rm(iv)] Determine a set $\MM_\ell\subseteq\TT_\ell$ of up to 
the multiplicative factor $\c{min}$ minimal 
cardinality such that
\begin{align}\label{eq:doerfler:bet}
 \theta\,\rho_{\ell}^2
 \le \rho_{\ell}(\MM_\ell)^2.
\end{align}
\item[\rm(v)] Let $\TT_{\ell+1}:=\refine(\TT_\ell,\MM_\ell)$ be the coarsest refinement of $\TT_\ell$ such that all marked elements $T\in\MM_\ell$ have been
refined.
\end{itemize}
\textsc{Output:} Sequence of successively refined triangulations $\TT_\ell$ and corresponding error estimators $\eta_{u,\ell},\eta_{z,\ell}$ for all $\ell\in\N_0$.\qed
\end{algorithm}

\begin{remark}
Building on the results
of~\cite{ms}, it is claimed and empirically investigated in~\cite{bet} that the combined 
D\"orfler marking~\eqref{eq:doerfler:bet} requires less adaptive steps to reach a prescribed accuracy.
We note that this has only  been  proved rigorously in~\cite{bet} for the Poisson
problem with polynomial data, while general data have to satisfy a certain saturation assumption for
the related data oscillation terms; see~\cite[eq.~(4.4)]{bet} and~\cite[Theorem~4.1]{bet}.
We note that this assumption also restricts the quasi-optimality analysis of~\cite{bet} which, in the spirit of~\cite{ms}, relies on the (constrained) contraction of the total error.
\qed
\end{remark}

\subsection{Mesh-refinement}\label{section:mesh}
We suppose that the mesh-refinement is a deterministic and fixed strategy, e.g., newest vertex bisection~\cite{stevenson:nvb}. 
Unlike~\cite{bet,ms}, we do not require the interior node property guaranteed by bisec5-refinement of marked elements~\cite{mns}.
For each triangulation $\TT$ 
and marked elements $\MM\subseteq\TT$, we let $\TT':=\refine(\TT,\MM)$ be the coarsest triangulation, where all
elements $T\in\MM$ have been refined. This may, in particular, include the preservation of conformity 
or bounded shape regularity. We write $\TT'\in\refine(\TT)$, if
there exist finitely many triangulations $\TT^{(0)},\dots,\TT^{(n)}$ and sets $\MM^{(j)}\subseteq\TT^{(j)}$ such that $\TT = \TT^{(0)}$, $\TT'=\TT^{(n)}$ and $\TT^{(j)} = \refine(\TT^{(j-1)},\MM^{(j-1)})$ for all $j=1,\dots,n$, where we formally allow $n=0$, i.e., $\TT=\TT^{(0)}\in\refine(\TT)$.
To abbreviate notation, let $\T:=\refine(\TT_0)$, where $\TT_0$ is the given initial triangulation of Algorithms~\ref{algorithm}--\ref{algorithm:bet}.

\subsection{Main result}\label{section:main}
Our main result requires the following abstract approximation
class for the error estimator $\eta_{u,\ell}$ resp.\ $\eta_{z,\ell}$. 
Let $\T_N:=\set{\TT\in\T}{\#\TT-\#\TT_0\le N}$ denote the (finite) set of all refinements of $\TT_0$ which have at most $N$ elements more
than $\TT_0$. 
For $s>0$ and $w\in\{u,z\}$, 
we write $w\in\A_s$ if
\begin{align*}
 \norm{w}{\A_s} 
 := \sup_{N\in\N_0} \Big((N+1)^s\min_{\TT_\star\in\T_N}\eta_{w,\star}\Big) < \infty,
\end{align*}
where $\eta_{w,\star}$ is the error estimator associated with
the optimal triangulation $\TT_\star\in\T_N$. In explicit terms, $\norm{w}{\A_s}<\infty$ means that an algebraic
convergence rate $\OO(N^{-s})$ for the error estimator is possible, if the optimal
triangulations are chosen.

For either algorithm, our
abstract main result is twofold: First, we prove linear convergence (Section~\ref{section:linear}): For each $0<q<1$ there exists
some $n$ such that for all $\ell\in\N$, the reduction of $\eta_{u,\ell}\,\eta_{z,\ell}$ by the factor 
$q$ requires at most $n$ steps of the adaptive loop, i.e., 
$\eta_{u,\ell+n}\,\eta_{z,\ell+n} \le q\,\eta_{u,\ell}\,\eta_{z,\ell}$. Second, we prove optimal
convergence behavior (Section~\ref{section:optA}--\ref{section:optC}): With respect to the number of elements $N\simeq\#\TT_\ell-\#\TT_0$, the product 
$\eta_{u,\ell}\,\eta_{z,\ell}$ decays with order $\OO(N^{-(s+t)})$ for each possible algebraic rate $s+t>0$, 
i.e., $\norm{u}{\A_s}+\norm{z}{\A_t}<\infty$. This means that the adaptive algorithms will 
asymptotically realize each possible algebraic decay.

\begin{remark}
Since our analysis works with the estimator instead of the error, it avoids the use of any (discrete) efficiency bound. Compared to~\cite{bet,ms}, this allows to use simple newest vertex bisection instead of bisec5-refinement for marked elements. Lemma~\ref{lemma:approximationclass} below states that for standard FEM our approximation class coincides with that of~\cite{bet,ckns,ms} which is defined through the so-called total error (i.e., error plus data oscillations).\qed
\end{remark}

\subsection{Axioms of Adaptivity}
\label{section:axioms}
In the following, let $\dist{w}{\cdot}{\cdot}:\,\T\times\T\to\R_{\geq 0}$ denote a distance function on the set of admissible triangulations which satisfies
\begin{align*}
 C_{\rm dist}^{-1}\dist{w}{\TT}{\TT^{\prime\prime}}&\leq \dist{w}{\TT}{\TT^\prime}+\dist{w}{\TT^\prime}{\TT^{\prime\prime}}\quad\text{for all }\TT,\TT^\prime,\TT^{\prime\prime}\in\T,\\
 \dist{w}{\TT}{\TT^\prime}&\leq C_{\rm dist}\dist{w}{\TT^\prime}{\TT}\quad\text{for all }\TT,\TT^\prime\in\T,
\end{align*}
with some uniform constant $C_{\rm dist}>0$; see also Remark~\ref{remark:laxmilgram} below.

The convergence and optimality analysis of Algorithm~\ref{algorithm} is done in the frame 
of the following four \emph{axioms of adaptivity}~\cite{axioms}, where axiom~\eqref{ass:orthogonal}
is slightly relaxed when compared to~\cite{axioms}:

\begin{enumerate}
\renewcommand{\theenumi}{A\arabic{enumi}}
\item\label{ass:stable}
\emph{Stability on non-refined elements}: There exists a constant $\newc{stable}>0$ 
such that for each triangulation $\TT_\ell\in\T$ and all refinements 
$\TT_\star\in\refine(\TT_\ell)$ the corresponding error estimators satisfy
\begin{align*}
 |\eta_{w,\star}(\TT_\ell\cap\TT_\star)
 - \eta_{w,\ell}(\TT_\ell\cap\TT_\star)|
 \le \c{stable}\,\dist{w}{\TT_\ell}{\TT_\star}
\end{align*}
\item\label{ass:reduction}
\emph{Reduction on refined elements}: There exist constants $0<\newq{reduction}<1$ 
and $\newc{reduction}>0$ such that for each triangulation $\TT_\ell\in\T$ and all refinements 
$\TT_\star\in\refine(\TT_\ell)$ the corresponding error estimators satisfy
\begin{align*}
\eta_{w,\star}(\TT_\star\backslash\TT_\ell)^2
\le \q{reduction}\,\eta_{w,\ell}(\TT_\ell\backslash\TT_\star)^2
 + \c{reduction}\,\dist{w}{\TT_\ell}{\TT_\star}^2
\end{align*}
\item\label{ass:orthogonal}\emph{Quasi-orthogonality}: Let $\TT_{\ell_n}$ be the (possibly finite) subsequence of triangulations $\TT_\ell$ generated by Algorithm~\ref{algorithm},~\ref{algorithm:mod}, or~\ref{algorithm:bet} which satisfy the 
D\"orfler marking on the refined elements, i.e.,
\begin{align}\label{eq:doerfler:refined}
 \theta\,\eta_{w,\ell_n}^2 
 \le \eta_{w,\ell_n}(\TT_{\ell_n}\backslash\TT_{\ell_n+1})^2.
\end{align}
Then, for all $\eps>0$, there exists some constant $\newc{orth}(\eps)>0$ such that for all 
$n\le N$, for which $\TT_{\ell_n},\ldots,\TT_{\ell_N}$ are well-defined, it holds 
\begin{align*}
 \sum_{j=n}^N\big(\dist{w}{\TT_{\ell_{j+1}}}{\TT_{\ell_j}}^2 - \eps\,\eta_{w,\ell_j}^2\big) 
 \le \c{orth}(\eps)\,\eta_{w,\ell_n}^2. 
\end{align*}
\item\label{ass:reliable}\emph{Discrete reliability}:
There exists a constant $\newc{reliable}>0$ such that for each triangulation $\TT_\ell\in\T$ and all refinements $\TT_\star\in\refine(\TT_\ell)$, it holds
\begin{align*}
 \dist{w}{\TT_\star}{\TT_\ell}
 \le \c{reliable}\,
 \eta_{w,\ell}(\RR_w(\TT_\ell,\TT_\star)),
\end{align*}
where the set $\RR_w(\TT_\ell,\TT_\star)$ satisfies $\TT_\ell\backslash\TT_\star\subseteq \RR_w(\TT_\ell,\TT_\star)\subseteq\TT_\ell$ and $\#\RR_w(\TT_\ell,\TT_\star)\le\c{reliable}\,\#(\TT_\ell\backslash\TT_\star)$ and hence consists essentially of the refined
elements only.
\end{enumerate}

We recall some elementary observations of~\cite{axioms}.

\begin{lemma}[quasi-monotonicity of estimator~{\cite[Lemma~3.5]{axioms}}]
\label{lemma:monotone}
There exists a constant $\newc{mon}>0$ which depends only on 
stability~\eqref{ass:stable}, reduction~\eqref{ass:reduction}, and discrete reliability~\eqref{ass:reliable}, such that for all $\TT_\ell\in\T$ and all refinements
$\TT_\star\in\refine(\TT_\ell)$, it holds 
$\eta_{w,\star}^2 \le \c{mon}\,\eta_{w,\ell}^2$.\qed
\end{lemma}

\begin{lemma}[optimality of D\"orfler marking~{\cite[Proposition~4.12]{axioms}}]
\label{lemma:doerfler}
Suppose stability~\eqref{ass:stable} and discrete reliability~\eqref{ass:reliable}.
For all $0<\theta<\theta_\star:=(1+\c{stable}^2\c{reliable}^2)^{-1}$, there
exists some $0<\kappa_\star<1$ such that for all $\TT_\ell\in\T$ and all refinements
$\TT_\star\in\refine(\TT_\ell)$, it holds
\begin{align}
 \eta_{w,\star}^2 \le \kappa_\star\,\eta_{w,\ell}^2
 \quad\Longrightarrow\quad
 \theta\,\eta_{w,\ell}^2
 \le \eta_{w,\ell}(\RR_w(\TT_\ell,\TT_\star))^2,
\end{align}
where $\RR_w(\TT_\ell,\TT_\star)$ is the set of refined elements from~\eqref{ass:reliable}.\qed
\end{lemma}

\begin{remark}[validity of quasi-orthogonality for $\eps=0$]\label{lem:qo0}
Suppose that~\eqref{ass:stable}--\eqref{ass:reliable} are valid. Arguing along the lines of~\cite[Proposition~4.11]{axioms}, one then sees that quasi-orthogonality~\eqref{ass:orthogonal} holds even with $\eps=0$ and
$\c{orth}$ depends only on $\c{linear}$, $\q{linear}$ from Proposition~\ref{prop:linear} and $\c{reliable}$.
\end{remark}

\begin{remark}\label{remark:laxmilgram}
{\rm(i)} In the setting of Section~\ref{section:motivation}, let $w\in\{u,z\}$ with $W_\star\in\{U_\star,Z_\star\}$ being the corresponding Galerkin solution for $\TT_\star\in\T$. The abstract distance is then defined by $\dist{w}{\TT_\ell}{\TT_\star}:=a(W_\star-W_\ell,W_\star-W_\ell)^{1/2}\simeq\norm{W_\star-W_\ell}{\XX}$; see Section~\ref{section:example}--\ref{section:bem} below. 

{\rm(ii)} Suppose that the bilinear form $a(\cdot,\cdot)$ is additionally symmetric, and let $\enorm{v}:=a(v,v)^{1/2}$ denote the induced energy norm which is an equivalent norm on $\XX$. Then, nestedness $\XX_n\subseteq\XX_m\subseteq\XX_k$ of the discrete spaces for all $k\ge m\ge n$ implies the Galerkin orthogonality
\begin{align*}
\enorm{W_k-W_m}^2+\enorm{W_m-W_n}^2=\enorm{W_k-W_n}^2\quad\text{for all }k\geq m\geq n.
\end{align*}
This and~\eqref{ass:reliable} imply
\begin{align*}
 \sum_{j=n}^N\dist{w}{\TT_{\ell_{j+1}}}{\TT_{\ell_j}}^2 
 &=  \sum_{j=n}^N \big(\enorm{W_{\ell_{j_{N+1}}}-W_{\ell_j}}^2-\enorm{W_{\ell_{j_{N+1}}}-W_{\ell_{j+1}}}^2\big)\\
 &\leq \enorm{W_{\ell_{j_{N+1}}}-W_{\ell_n}}^2\simeq \norm{W_{\ell_{j_{N+1}}}-W_{\ell_n}}{\XX}^2\lesssim \eta_{w,\ell_n}^2.
\end{align*}
and hence the quasi-orthogonality~\eqref{ass:orthogonal} with $\eps=0$.
\qed
\end{remark}

\subsection{Generalized linear convergence}\label{section:Rlin}
The following estimator reduction is first found in~\cite{ckns} and, e.g., proved along the lines of~\cite[Lemma~4.7]{axioms}. 
Since we need a slightly stronger result than that of~\cite{axioms}, 
which covers arbitrary refinements $\TT_\star\in\refine(\TT_{\ell+1})$ instead
of $\TT_{\ell+1}$ only,
we include the
proof for convenience of the reader.

\begin{lemma}[generalized estimator reduction]\label{lemma:estred}
Let $0<\theta\le1$. Let $\TT_\ell\in\T$ and $\TT_{\ell+1}\in\refine(\TT_\ell)$ be given 
triangulations and suppose that the refined elements satisfy the D\"orfler marking
\begin{align}\label{eq1:estred}
 \theta\,\eta_{w,\ell}^2 
 \le \eta_{w,\ell}(\TT_\ell\backslash\TT_{\ell+1})^2.
\end{align}
Then, there exist constants $0<\newq{estred}<1$ and $\newc{estred}>0$ which depend
only on assumption~\eqref{ass:stable}--\eqref{ass:reduction} and $\theta$ such that 
for all $\TT_\star \in \refine(\TT_{\ell+1})$, it holds 
\begin{align}\label{eq2:estred}
 \eta_{w,\star}^2 \le \q{estred}\,\eta_{w,\ell}^2 + \c{estred}\,\dist{w}{\TT_\star}{\TT_\ell}^2.
\end{align}
\end{lemma}

\begin{proof}
For each $\delta>0$, the Young inequality $(a+b)^2 \le (1+\delta)a^2 + (1+\delta^{-1})b^2$
and stability~\eqref{ass:stable} yield
\begin{align*}
 \eta_{w,\star}(\TT_\star\cap\TT_\ell)^2
 &\le (1+\delta)\,\eta_{w,\ell}(\TT_\star\cap\TT_\ell)^2
 + \c{stable}^2(1+\delta^{-1})\,\dist{w}{\TT_\star}{\TT_\ell}^2
 \\&
 = (1+\delta)\,\eta_{w,\ell}^2 - (1+\delta)\,\eta_{w,\ell}(\TT_\ell\backslash\TT_\star)^2
 + \c{stable}^2(1+\delta^{-1})\,\dist{w}{\TT_\star}{\TT_\ell}^2.
\end{align*}
Together with reduction~\eqref{ass:reduction}, we obtain
\begin{align*}
 \eta_{w,\star}^2 
 &= \eta_{w,\star}(\TT_\star\cap\TT_\ell)^2
 + \eta_{w,\star}(\TT_\star\backslash\TT_\ell)^2
 \\&
 \le (1+\delta)\,\eta_{w,\ell}^2 - (1+\delta-\q{reduction})\,
 \eta_{w,\ell}(\TT_\ell\backslash\TT_\star)^2
 + \big(\c{stable}^2(1+\delta^{-1})+\c{reduction}\big)\,\dist{w}{\TT_\star}{\TT_\ell}^2.
\end{align*}
With $\TT_\ell\backslash\TT_\star\supseteq\TT_\ell\backslash\TT_{\ell+1}$ and
the D\"orfler marking~\eqref{eq1:estred}, we see
\begin{align*}
 \eta_{w,\ell}(\TT_\ell\backslash\TT_\star)^2
 \ge \eta_{w,\ell}(\TT_\ell\backslash\TT_{\ell+1})^2
 \ge \theta\,\eta_{w,\ell}^2.
\end{align*}
Combining the last two estimates, we prove~\eqref{eq2:estred}
with $\c{estred} = \c{stable}^2(1+\delta^{-1})+\c{reduction}$ and
$\q{estred} = (1 + \delta) - (1+\delta-\q{reduction})\theta
\to 1-(1-\q{reduction})\theta<1$ as $\delta\to0$. For sufficiently small 
$\delta>0$, it thus holds $0<\q{estred}<1$.
\end{proof}

The following result generalizes~\cite[Proposition~4.10]{axioms} to the present setting. We note that~\eqref{ass:reliable} enters only through the quasi-monotonicity of the estimator (Lemma~\ref{lemma:monotone}).

\begin{proposition}[generalized linear convergence]\label{prop:linear}
Let $\TT_\ell$ be a sequence
of successively refined triangulations, i.e., $\TT_\ell \in \refine(\TT_{\ell-1})$
for all $\ell\in\N$. 
Let $0<\theta\le1$. Then, there are constants $0<\newq{linear}<1$ and $\newc{linear}>0$ 
which depend only on~\eqref{ass:stable}--\eqref{ass:reliable}
and $\theta$, such that the following holds: Let 
$\ell,n\in\N_0$ and suppose that there are at least $k\le n$ indices 
$\ell \le \ell_1<\ell_2<\dots<\ell_k<\ell+n$ such that the D\"orfer marking~\eqref{eq1:estred}
is satisfied on the refined elements, i.e.,
\begin{align}\label{eq1:linear}
 \theta\,\eta_{w,j_m}^2 
 \le \eta_{w,\ell_j}(\TT_{\ell_j}\backslash\TT_{\ell_j+1})^2
 \quad\text{for all }j=1,\dots k.
\end{align}
Then, the error estimator satisfies
\begin{align}\label{eq2:linear}
 \eta_{w,\ell+n}^2 \le \c{linear}\,\q{linear}^k\,\eta_{w,\ell}^2.
\end{align}
\end{proposition}

\begin{proof}
To abbreviate notation, set $j_0:=\ell$. Note that $\TT_{\ell_{k+1}}\in\refine(\TT_{\ell_k+1})$. Therefore, the estimator reduction~\eqref{eq2:estred} shows for all $\eps>0$ and all $0\leq j\leq k$
\begin{align*}
 \sum_{i=k-j}^{k}\eta_{w,\ell_{i+1}}^2&\leq \sum_{i=k-j}^{k}\Big(\q{estred}\eta_{w,\ell_i}^2 + \c{estred}\dist{w}{\TT_{\ell_{i+1}}}{\TT_{\ell_i}}^2\Big)\\
 &= \sum_{i=k-j}^{k}\Big((\q{estred}+\c{estred}\eps)\eta_{w,\ell_i}^2 + \c{estred}\big(\dist{w}{\TT_{\ell_{i+1}}}{\TT_{\ell_i}}^2-\eps\eta_{w,\ell_i}^2\big)\Big).
\end{align*}
Choose $\eps < (1-\q{estred})\c{estred}^{-1}$ so that $\kappa:= 1-(\q{estred}+\c{estred}\eps)>0$. Use~\eqref{ass:orthogonal} to obtain for all $0\leq j\leq k$
\begin{align}\label{eq:rlinhelp}
\begin{split}
 \kappa  \sum_{i=k-j}^{k}\eta_{w,\ell_{i+1}}^2&\leq \eta_{w,\ell_{k-j}}^2 + \c{estred}\sum_{i=k-j}^{k}\big(\dist{w}{\TT_{\ell_{i+1}}}{\TT_{\ell_i}}^2-\eps\eta_{w,\ell_i}^2\big)\\&\leq  (1+ \c{estred}\c{orth}(\eps))\eta_{w,\ell_{k-j}}^2.
 \end{split}
\end{align}
With $C:=(1+ \c{estred}\c{orth}(\eps))/\kappa>1$, mathematical induction below shows 
\begin{align}\label{eq:rlinind}
\eta_{w,\ell_k}^2\leq (1-C^{-1})^j \sum_{i=k-j}^{k}\eta_{w,\ell_i}^2\quad\text{for all }0\leq j\leq k.
\end{align}
To see~\eqref{eq:rlinind}, note that the case $j=0$ holds with equality.
Suppose that~\eqref{eq:rlinind} holds for $j<k$. This induction hypothesis and~\eqref{eq:rlinhelp} show
\begin{align*}
 \eta_{w,\ell_k}^2\leq (1-C^{-1})^j\sum_{i=k-j}^j\eta_{w,\ell_i}^2&= (1-C^{-1})^j\big((\sum_{i=k-(j+1)}^k\eta_{w,\ell_i}^2) - \eta_{w,\ell_{k-(j+1)}}^2\big)\\&\overset{\eqref{eq:rlinhelp}}{\leq} (1-C^{-1})^{j+1} \sum_{i=k-(j+1)}^k\eta_{w,\ell_i}^2,
\end{align*}
which proves validity of the induction step. Hence, the assertion~\eqref{eq:rlinind} holds for all $j\leq k$. By use of~\eqref{eq:rlinind} for $j=k-1$ and~\eqref{eq:rlinhelp} for $j=k$, we obtain
\begin{align*}
 \c{mon}^{-1}\eta_{w,\ell+n}^2\leq \eta_{w,\ell_k}^2
 &\overset{\eqref{eq:rlinind}}{\leq} (1-C^{-1})^{k-1} \sum_{i=1}^{k}\eta_{w,\ell_i}^2
 \le (1-C^{-1})^{k-1} \sum_{i=0}^{k}\eta_{w,\ell_{i+1}}^2
 \\&\overset{\eqref{eq:rlinhelp}}{\leq} 
 (1-C^{-1})^{k-1}C\,\eta_{w,\ell_0}^2 
 = (1-C^{-1})^{k}C/(1-C^{-1})\,\eta_{w,\ell}^2.
\end{align*}
This concludes the proof with $\c{linear}=C\c{mon}/(1-C^{-1})$ and $\q{linear}=(1-C^{-1})$.
\end{proof}

\section{Optimal Convergence of Adaptive Algorithms}\label{section:optimal}

\noindent
Throughout this section, we suppose that the error estimators $\eta_{u,\ell}$ 
and $\eta_{z,\ell}$ satisfy the respective assumptions~\eqref{ass:stable}--\eqref{ass:reliable} 
of Section~\ref{section:axioms}. Without loss of generality, we suppose 
that $\eta_{u,\ell}$ and $\eta_{z,\ell}$ satisfy the 
axioms~\eqref{ass:stable}--\eqref{ass:reliable} with the same constants.

\begin{remark}
The axioms~\eqref{ass:stable}--\eqref{ass:reliable} are designed to cover weighted-residual error estimators in the frame of FEM and BEM. However, as is shown in~\cite[Section~8]{axioms} for optimal adaptivity for the energy error, it is sufficient that  for $w\in\{u,z\}$ the error estimator $\eta_{w,\ell}$ used in the adaptive algorithm is locally equivalent to some error estimator $\widetilde\eta_{w,\ell}$ which satisfies~\eqref{ass:stable}--\eqref{ass:reliable}, i.e.,
\begin{align*}
 \eta_{\ell,w}(T) \lesssim \widetilde\eta_{\ell,w}(\omega_\ell(T))
 \quad\text{and}\quad
 \widetilde\eta_{\ell,w}(T) \lesssim \eta_{\ell,w}(\omega_\ell(T))
 \quad\text{for all }T\in\TT_\ell,
\end{align*}
where $\omega_\ell(T)$ denotes a patch of $T$. Then, the convergence (Theorem~\ref{theorem:linear}) as well as optimality results (Theorem~\ref{theorem:optimal}, \ref{theorem:optimal:mod}, and~\ref{theorem:optimal:bet}) remain valid. We leave the details to the reader, but note that such arguments cover averaging-based error estimators, hierarchical error estimators, as well as estimators based on equilibrated fluxes; see~\cite{axioms,MR2776915}.\qed
\end{remark}

\subsection{Linear convergence of Algorithms~\ref{algorithm},~\ref{algorithm:mod}, and~\ref{algorithm:bet}}
\label{section:linear}
The following result gives a criterion for linear convergence which is satisfied for either adaptive algorithm (Algorithm~\ref{algorithm},~\ref{algorithm:mod}, and~\ref{algorithm:bet}). As a consequence, linear convergence is independent of $\c{min}$, and we may formally also choose $\c{min}=\infty$.
Finally, linear convergence~\eqref{eq:prop:linear} does only rely on
discrete reliability~\eqref{ass:reliable} to ensure quasi-monotonicity of the estimator (Lemma~\ref{lemma:monotone}). In the frame of the Lax-Milgram lemma from Section~\ref{section:motivation}, the quasi-monotonicity already follows from classical reliability~\eqref{eq:reliable:laxmilgram}; see~\cite[Lemma~3.6]{axioms}.

\begin{proposition}\label{prop:linearconv}
Let the sequence of meshes $(\TT_\ell)_{\ell\in\N_0}\subset \T$ satisfy $\TT_{\ell+1}\in\refine(\TT_\ell)$ for all $\ell\in\N_0$.
 For all $\ell\in\N_0$, let the set $\MM_\ell:=\TT_\ell\setminus\TT_{\ell+1}$ satisfy either~\eqref{doerfler:u} or~\eqref{doerfler:z}. For all $0<\theta\le1$, there are constants $0<\newq{rlinear}<1$ and $\newc{rlinear}>0$ which depend only on~\eqref{ass:stable}--\eqref{ass:reliable} and $\theta$, such that there holds linear convergence in the sense of
 \begin{align}\label{eq:prop:linear}
  \eta_{u,\ell+n}\eta_{z,\ell+n} \le \c{rlinear}\q{rlinear}^n\eta_{u,\ell}\eta_{z,\ell}
 \quad\text{for all }\ell,n\in\N_0. 
 \end{align}
\end{proposition}
\begin{proof}
The assumptions on $(\TT_\ell)_{\ell\in\N_0}$ imply that for $n$ successive meshes
$\TT_j,\,j=\ell,\dots,\ell+n$, 
$\TT_j\backslash\TT_{j+1}$ satisfies $k$-times the D\"orfler marking for
$\eta_{u,\ell}$ and $(n-k)$-times the D\"orfler marking for
$\eta_{z,\ell}$. According to Proposition~\ref{prop:linearconv}, this implies
\begin{align*}
 \eta_{u,\ell+n}^2 \le \c{linear}\,\q{linear}^k\,\eta_{u,\ell}^2
 \quad\text{as well as}\quad
 \eta_{z,\ell+n}^2 \le \c{linear}\,\q{linear}^{n-k}\,\eta_{z,\ell}^2.
\end{align*}
Altogether, this proves
\begin{align*}
 \eta_{u,\ell+n}^2 \, \eta_{z,\ell+n}^2
 \le \c{linear}^2\,\q{linear}^k\,\eta_{u,\ell}^2 \, \eta_{z,\ell}^2.
\end{align*}
This concludes~\eqref{eq:prop:linear} with $\q{rlinear} = \q{linear}^{1/2}$
and $\c{rlinear} = \c{linear}$.
\end{proof}

\begin{theorem}\label{theorem:linear}
For all $0<\theta\le1$, there are constants $0<\newq{rlinear}<1$ and $\newc{rlinear}>0$ which depend only on~\eqref{ass:stable}--\eqref{ass:reliable} and $\theta$ such that Algorithms~\ref{algorithm},~\ref{algorithm:mod}, and~\ref{algorithm:bet} are linear convergent in the sense of~\eqref{eq:prop:linear} 
\end{theorem}
\begin{proof}[Proof for Algorithm~\ref{algorithm}]
In each step of Algorithm~\ref{algorithm}, the set $\MM_\ell$ satisfies either
the D\"orfler marking~\eqref{doerfler:u} for $\eta_{u,\ell}$ 
or~\eqref{doerfler:z} for $\eta_{z,\ell}$. Since $\MM_\ell\subseteq\TT_\ell\backslash\TT_{\ell+1}$, the assumptions of Proposition~\ref{prop:linearconv} are satisfied. This concludes the proof for Algorithm~\ref{algorithm}.
\end{proof}
\begin{proof}[Proof for Algorithm~\ref{algorithm:mod}]
Analogously, the set $\widetilde\MM_\ell$ satisfies in each step of Algorithm~\ref{algorithm:mod} either~\eqref{doerfler:u:mod} or~\eqref{doerfler:z:mod}. Since $\widetilde\MM_\ell\subseteq \MM_\ell\subseteq \TT_\ell\setminus\TT_{\ell+1}$,
Proposition~\ref{prop:linearconv} applies and concludes the proof for Algorithm~\ref{algorithm:mod}.
\end{proof}
\begin{proof}[Proof for Algorithm~\ref{algorithm:bet}]
 Note that $\rho_\ell^2 = 2\,\eta_{u,\ell}^2\eta_{z,\ell}^2$. Therefore,~\eqref{eq:doerfler:bet} 
becomes
\begin{align*}
 2\theta\,\eta_{u,\ell}^2\eta_{z,\ell}^2
 \le \eta_{u,\ell}(\MM_\ell)^2\,\eta_{z,\ell}^2
 + \eta_{u,\ell}^2\,\eta_{z,\ell}(\MM_\ell)^2.
\end{align*}
In particular, this shows that 
\begin{align*}
 \theta\,\eta_{u,\ell}^2 \le \eta_{u,\ell}(\MM_\ell)^2
 \quad\text{or}\quad
 \theta\,\eta_{z,\ell}^2 \le \eta_{z,\ell}(\MM_\ell)^2.
 \end{align*}
Since $\MM_\ell\subseteq \TT_\ell\setminus\TT_{\ell+1}$, Proposition~\ref{prop:linearconv} applies and concludes the proof.
\end{proof}

\subsection{Fine properties of mesh-refinement\label{section:finemesh}}
The following Theorem~\ref{theorem:optimal} states optimal convergence behavior. Unlike linear convergence, the 
proof of optimal convergence rates is more strongly tailored to the mesh-refinement used. First, we suppose that each refined element has at least two sons, i.e.,
\begin{align}\label{eq:mesh-sons}
\#(\TT\backslash\TT') + \#\TT \le \#\TT'
\quad\text{for all }\TT\in\T\text{ and }\TT'\in\refine(\TT).
\end{align}
Second, we require the mesh-closure estimate
\begin{align}\label{eq:mesh-closure}
 \#\TT_\ell-\#\TT_0 \le \c{nvb}\,\sum_{j=0}^{\ell-1}\#\MM_j
 \quad\text{for all }\ell\in\N,
\end{align}
where $\newc{nvb}>0$ depends only on $\TT_0$. This has first been proved for 2D newest
vertex bisection in~\cite{bdd} and has later been generalized to arbitrary dimension 
$d\ge2$ in~\cite{stevenson:nvb}. While both works require an additional admissibility assumption 
on $\TT_0$, this has at least been proved unnecessary for 2D in~\cite{nvb}. Finally,
it has been proved in~\cite{ckns,stevenson} that newest vertex bisection ensures the
overlay estimate, i.e., for all triangulations $\TT,\TT'\in\T$ there exists a common
refinement $\TT\oplus\TT'\in\refine(\TT)\cap\refine(\TT')$ which satisfies 
\begin{align}\label{eq:mesh-overlay}
 \#(\TT\oplus\TT') \le \#\TT + \#\TT' - \#\TT_0.
\end{align}
Although not used explicitly, we note that for newest vertex bisection, the triangulation $\TT\oplus\TT'$  is, in fact, the overlay of
$\TT$ and $\TT'$. For 1D bisection (e.g., for 2D BEM computations in Section~\ref{section:bem}), the algorithm from~\cite{eps65} satisfies~\eqref{eq:mesh-sons}--\eqref{eq:mesh-overlay} and guarantees that the local mesh-ratio is uniformly bounded. For meshes with first-order hanging nodes, \eqref{eq:mesh-sons}--\eqref{eq:mesh-overlay} are analyzed in~\cite{MR2670003}, while T-spline meshes for isogeometric analysis are considered in~\cite{peterseim}.

\subsection{Optimal convergence rates for Algorithm~\ref{algorithm}}\label{section:optA}
Our proofs of the following theorems (Theorem~\ref{theorem:optimal}, \ref{theorem:optimal:mod}, \ref{theorem:optimal:bet}) follow the ideas of~\cite{ms} as worked out in~\cite{pointabem}. We include it here for the sake of completeness and a self-contained presentation.

\begin{theorem}\label{theorem:optimal}
Suppose that the mesh-refinement satisfies \eqref{eq:mesh-sons} as well as the mesh-closure estimate~\eqref{eq:mesh-closure}
and the overlay estimate~\eqref{eq:mesh-overlay}.
Let $0<\theta< \theta_\star:=(1+\c{stable}^2\c{reliable}^2)^{-1}$ be sufficiently small.
Then, there exists a constant $\newc{final1}>0$ which depends only on $\theta$, $\c{nvb}$, and~\eqref{ass:stable}--\eqref{ass:reliable}, such that
for all $s,t>0$ the assumption
$(u,z)\in\A_s\times\A_t$ implies $\text{for all }\ell\in\N_0$
\begin{align}\label{eq:thm:optimal}
 C_{\rm goal}^{-1}|g-g_\ell|\leq\eta_{u,\ell}\eta_{z,\ell} \le \frac{\c{final1}^{1+s+t}}{(1-\q{rlinear}^{1/(s+t)})^{s+t}}\,\norm{u}{\A_s}\norm{z}{\A_t}\,(\#\TT_\ell-\#\TT_0)^{-(s+t)}
\end{align}
i.e., Algorithm~\ref{algorithm} guarantees that the estimator product decays asymptotically with any possible algebraic rate.
\end{theorem}

\color{black}
\begin{corollary}\label{corollary:optimal}
 Assume that the estimators both have finite optimal convergence rate, i.e.,
 \begin{align*}
  s_{\rm max}:=\sup\set{s>0}{\norm{u}{\A_s}<\infty}<\infty
 \quad\text{and}\quad
  t_{\rm max}:=\sup\set{t>0}{\norm{z}{\A_t}<\infty}<\infty.
 \end{align*}
Then,~\eqref{eq:thm:optimal} implies for all $0<s<s_{\rm max}$ and $0<t<t_{\rm max}$ and 
all $\ell\in\N$
\begin{align*}
 \eta_{u,\ell}&\lesssim (\#\TT_\ell-\#\TT_0)^{-s}
 \quad\text{as well as}\quad
  \eta_{z,\ell}\lesssim (\#\TT_\ell-\#\TT_0)^{-t},
\end{align*}
where the hidden constants additionally depend on $s_{\rm max}-s>0$
resp.\ $t_{\rm max}-t>0$. \qed
\end{corollary}

The heart of the proof of Theorem~\ref{theorem:optimal} is the following lemma.

\begin{lemma}\label{step1:optimal}
For any $0<\theta<\theta_\star:=(1+\c{stable}^2\c{reliable}^2)^{-1}$ and $\ell\in\N_0$,
there exist constants $\newc{stevenson3},\newc{stevenson2}>0$ and some refinement $\TT_\star\in\refine(\TT_\ell)$ such that the sets $\RR_u(\TT_\ell,\TT_\star)$ and
$\RR_z(\TT_\ell,\TT_\star)$ from the discrete reliability~\eqref{ass:reliable} satisfy
for all $s,t>0$ with $(u,z)\in\A_s\times\A_t$
\begin{align}\label{eq:new1}
 \max \{\#\RR_u(\TT_\ell,\TT_\star)\,,\,\#\RR_z(\TT_\ell,\TT_\star)\}
\le \c{stevenson3}\,(\c{stevenson2}\norm{u}{\A_s}\norm{z}{\A_t})^{1/(s+t)}\,
(\eta_{u,\ell}\eta_{z,\ell})^{-1/(s+t)}.
\end{align}
Moreover, either  $\RR_u(\TT_\ell,\TT_\star)$ or $\RR_z(\TT_\ell,\TT_\star)$ satisfies the D\"orfler marking, i.e., it holds
\begin{align}\label{eq:new2}
 \theta\eta_{u,\ell}^2 \le \eta_{u,\ell}\big(\RR_u(\TT_\ell,\TT_\star)\big)^2
 \quad\text{or}\quad
 \theta\eta_{z,\ell}^2 \le \eta_{z,\ell}\big(\RR_z(\TT_\ell,\TT_\star)\big)^2.
\end{align}
The constants $C_1,C_2$ depend only on $\theta$, \eqref{ass:stable},\eqref{ass:reduction}, and~\eqref{ass:reliable}.
\end{lemma}

\begin{proof}
Adopt the notation of Lemma~\ref{lemma:doerfler}.
For $\eps := \c{mon}^{-1}\kappa_\star\,\eta_{u,\ell}\eta_{z,\ell}$, the quasi-monotonicity of the estimators (Lemma~\ref{lemma:monotone}) yields
 $\eps \le \kappa_\star\,\eta_{u,0}\eta_{z,0} < \norm{u}{\A_s}\norm{z}{\A_t} < \infty$.
Choose the minimal $N\in\N_0$ such that $\norm{u}{\A_s}\norm{z}{\A_t}
\le\eps\,(N+1)^{s+t}$. 
Then, Lemma~\ref{lemma:monotone}, the definition of the 
approximation classes, and the choice of $N$ give
\begin{align*}
 \eta_{u,\star}\eta_{z,\star}
 \le \c{mon} \eta_{u,\eps_1}\eta_{z,\eps_2}
 \le \c{mon}(N+1)^{-(s+t)}\norm{u}{\A_s}\norm{z}{\A_t}
 \le \c{mon}\eps  = \kappa_\star\,\eta_{u,\ell}\eta_{z,\ell}.
\end{align*}
This implies $\eta_{u,\star}^2 \le \kappa_\star\,\eta_{u,\ell}^2$ or
$\eta_{z,\star}^2 \le \kappa_\star\,\eta_{z,\ell}^2$, and Lemma~\ref{lemma:doerfler}
hence proves~\eqref{eq:new2}. It remains to derive~\eqref{eq:new1}. First, note that 
\begin{align}\label{eq:new:step1}
\max\{\#\RR_u(\TT_\ell,\TT_\star)\,,\,\#\RR_z(\TT_\ell,\TT_\star)\}
\stackrel{\eqref{ass:reliable}}\le \c{reliable}\,\#(\TT_\ell\backslash\TT_\star)
\stackrel{\eqref{eq:mesh-sons}}\le \c{reliable}(\#\TT_\star-\#\TT_\ell),
\end{align}
since refined elements are refined into at least two sons~\eqref{eq:mesh-sons}.
Second, minimality of $N$ yields
\begin{align*}
 N < (\norm{u}{\A_s}\norm{z}{\A_t})^{1/(s+t)}\eps^{-1/(s+t)}
 = C \,(\eta_{u,\ell}\eta_{z,\ell})^{-1/(s+t)}
\end{align*}
with $C := (\norm{u}{\A_s}\norm{z}{\A_t})^{1/(s+t)}(\c{mon}^{-1}\kappa_\star)^{-1/(s+t)} = (\c{mon}\kappa_\star^{-1}\,\norm{u}{\A_s}\norm{z}{\A_t})^{1/(s+t)}$.
Choose $\TT_{\eps_1},\TT_{\eps_2}\in\T_N$ with 
$\eta_{u,\eps_1}=\min_{\TT_\star\in\T_N}\eta_{u,\star}$ and
$\eta_{z,\eps_2}=\min_{\TT_\star\in\T_N}\eta_{z,\star}$.
Define $\TT_\eps:=\TT_{\eps_1}\oplus\TT_{\eps_2}$
and $\TT_\star:=\TT_\eps\oplus\TT_\ell$. The overlay estimate~\eqref{eq:mesh-overlay} yields
\begin{align}\label{eq:new:step2}
\#\TT_\star - \#\TT_\ell
\stackrel{\eqref{eq:mesh-overlay}}\le\#\TT_\eps - \#\TT_0
\stackrel{\eqref{eq:mesh-overlay}}\le\#\TT_{\eps_1} + \#\TT_{\eps_2} - 2\,\#\TT_0
\le 2N < 2C\, (\eta_{u,\ell}\eta_{z,\ell})^{-1/(s+t)}.
\end{align}
Combining~\eqref{eq:new:step1}--\eqref{eq:new:step2}, we conclude~\eqref{eq:new1}
with $\c{stevenson3} = 2\c{rel}$ and $\c{stevenson2} = \c{mon}/\kappa_\star$.
\end{proof}

\begin{proof}[Proof of Theorem~\ref{theorem:optimal}]
According to~\eqref{eq:new2} of Lemma~\ref{step1:optimal} and the marking strategy in Algorithm~\ref{algorithm}, it holds for all $j\in\N_0$
\begin{align}\label{eq:minmark}
\begin{split}
\#\MM_j = \min\{\#\MM_{u,j}\,,\,\#\MM_{z,j}\}
\le \c{min}\,\max\{\#\RR_u(\TT_j,\TT_\star)\,,\,\#\RR_z(\TT_j,\TT_\star)\}.
\end{split}
\end{align}
With the mesh-closure estimate~\eqref{eq:mesh-closure} and estimate~\eqref{eq:new1} of Lemma~\ref{step1:optimal},
we obtain
\begin{align*}
\#\TT_\ell-\#\TT_0
\stackrel{\eqref{eq:mesh-closure}}\le\c{nvb} \sum_{j=0}^{\ell-1}\#\MM_j
\stackrel{\eqref{eq:new1}}\le \c{nvb}\c{min}\c{stevenson3}\,(\c{stevenson2}\norm{u}{\A_s}\norm{z}{\A_t})^{1/(s+t)} \sum_{j=0}^{\ell-1}(\eta_{u,j}\eta_{z,j})^{-1/(s+t)}.
\end{align*}
Linear convergence~\eqref{eq:prop:linear} implies
\begin{align*}
 \eta_{u,\ell}\eta_{z,\ell}
 \le\c{rlinear}\, \q{rlinear}^{\ell-j}\eta_{u,j}\eta_{z,j}
 \quad\text{for all }0\le j\le\ell
\end{align*}
and hence
\begin{align*}
 (\eta_{u,j}\eta_{z,j})^{-1/(s+t)}
 \le \c{rlinear}^{1/(s+t)} \q{rlinear}^{(\ell-j)/(s+t)}
 (\eta_{u,\ell}\eta_{z,\ell})^{-1/(s+t)}.
\end{align*}
With $0<q:=\q{rlinear}^{1/(s+t)}<1$, the geometric series applies
and yields
\begin{align*}
 \sum_{j=0}^{\ell-1}(\eta_{u,j}\eta_{z,j})^{-1/(s+t)}
 \le \c{rlinear}^{1/(s+t)}  (\eta_{u,\ell}\eta_{z,\ell})^{-1/(s+t)}\,
 \sum_{j=0}^{\ell-1}q^{\ell-j}
 \le \frac{\c{rlinear}^{1/(s+t)}}{1-\q{rlinear}^{1/(s+t)}}\, (\eta_{u,\ell}\eta_{z,\ell})^{-1/(s+t)}. 
\end{align*}
Combining this with the first estimate, we obtain
\begin{align*}
 \#\TT_\ell-\#\TT_0
 \le \frac{\c{nvb}\c{min}\c{stevenson3}}{1-\q{rlinear}^{1/(s+t)}}\,(\c{rlinear}\c{stevenson2}\,\norm{u}{\A_s}\norm{z}{\A_t})^{1/(s+t)}
 \,(\eta_{u,\ell}\eta_{z,\ell})^{-1/(s+t)}.
\end{align*}
Rearranging this estimate, we conclude~\eqref{eq:thm:optimal}
with $\c{final1} =\max\{ \c{rlinear}\c{stevenson2},\c{nvb}\c{min}\c{stevenson3}\}$.
\end{proof}
\color{black}

\subsection{Optimal convergence rates for Algorithm~\ref{algorithm:mod}}
\begin{theorem}\label{theorem:optimal:mod}
Let $\theta_\star:=(1+\c{stable}\c{reliable})^{-1}$. 
For any $0<\theta<\theta_\star$, Algorithm~\ref{algorithm:mod} guarantees 
optimal algebraic convergence rates in the sense of Theorem~\ref{theorem:optimal}
and Corollary~\ref{corollary:optimal}.
\end{theorem}

\begin{proof}
We note that Lemma~\ref{step1:optimal} is not affected by the marking strategy and hence remains valid. To conclude the proof, we only need to show that estimate~\eqref{eq:minmark} also remains true.
Since $\MM_j$ in Algorithm~\ref{algorithm:mod} is a set of minimal cardinality up to the factor $\c{min2}\c{min}$ which satisfies either~\eqref{doerfler:u} or~\eqref{doerfler:z}, estimate~\eqref{eq:minmark} holds
with different constants, i.e.,
\begin{align*}
 \#\MM_j\leq \c{min2}\c{min}\max\{\#\RR_u(\TT_j,\TT_\star)\,,\,\#\RR_z(\TT_j,\TT_\star)\}.
\end{align*}
Therefore, the claim follows with
$\c{final1} =\max\{ \c{rlinear}\c{stevenson2},\c{nvb}\c{min2}\c{min}\c{stevenson3}\}$.
\end{proof}

\subsection{Optimal convergence rates for Algorithm~\ref{algorithm:bet}}\label{section:optC}

\begin{theorem}\label{theorem:optimal:bet}
Let $\theta_\star:=(1+\c{stable}\c{reliable})^{-1}$. 
For any $0<\theta<\theta_\star/2$, Algorithm~\ref{algorithm:bet} guarantees 
optimal algebraic convergence rates in the sense of Theorem~\ref{theorem:optimal}
and Corollary~\ref{corollary:optimal}.
\end{theorem}

\begin{proof}
We only need to show that~\eqref{eq:minmark} remains valid.
Note that $0<2\theta<\theta_\star$. Therefore, estimate~\eqref{eq:new2} of Lemma~\ref{step1:optimal} yields
\begin{align*}
 2\theta\,\eta_{u,j}^2 \le \eta_{u,j}\big(\RR_u(\TT_j,\TT_\star)\big)^2
 \quad\text{or}\quad
 2\theta\,\eta_{z,j}^2 \le \eta_{z,j}\big(\RR_z(\TT_j,\TT_\star)\big)^2.
\end{align*}
Either for $\RR_j := \RR_u(\TT_j,\TT_\star)$ or for 
$\RR_j := \RR_z(\TT_j,\TT_\star)$ 
this implies
\begin{align*}
 \theta\,\rho_j^2 
 = 2\theta\,\eta_{u,j}^2\eta_{z,j}^2
 \le \eta_{u,j}(\RR_j)^2\,\eta_{z,j}^2
 + \eta_{u,j}^2\,\eta_{z,j}(\RR_j)^2
 = \rho_j(\RR_j)^2.
\end{align*}
According to the marking strategy in Algorithm~\ref{algorithm:bet}, we obtain
\begin{align*}
 \#\MM_j 
 &\le \c{min}\#\RR_j
 \le \c{min}\, \max\{\#\RR_u(\TT_j,\TT_\star)\,,\,\#\RR_z(\TT_j,\TT_\star)\}
\end{align*}
which is~\eqref{eq:minmark}. 
Therefore, the claim follows with
$\c{final1} =\max\{ \c{rlinear}\c{stevenson2},\c{nvb}\c{min}\c{stevenson3}\}$.
\end{proof}
\color{black}

\section{Goal-Oriented Adaptive FEM for Second-Order Linear Elliptic PDEs}\label{section:example}

\noindent
In this section, we extend the ideas from~\cite{ffp} and prove that our abstract frame of 
convergence and optimality of goal-oriented AFEM applies, in particular, to general
second-order linear elliptic PDEs.
\def\bform#1#2{a(#1,#2)}%
\def\operator#1{\mathcal{#1}}%
\def\matrix#1{\boldsymbol{#1}}%

\subsection{Model problem}\label{afem:model}
On the bounded Lipschitz domain $\Omega\subset \R^d$ and for given $f_1,g_1\in L^2(\Omega)$ and $\boldsymbol{f}_2,\boldsymbol{g}_2\in L^2(\Omega)^d$, we aim to compute
\begin{align*}
g(u):=\int_\Omega g_1u - \boldsymbol{g}_2\cdot\nabla u\,dx,
\end{align*}
where $u\in H^1_0(\Omega)$ is the weak solution to
\begin{align}\label{ex:nonsymm}
 \operator{L}u
 := -\text{div}(\matrix{A}\nabla u) + \matrix{b}\cdot\nabla u + c u 
 = f_1+{\rm div}\,\boldsymbol{f}_2\quad\text{in }\Omega
\quad\text{and}\quad
 u=0\quad\text{on }\Gamma:=\partial\Omega.
\end{align}
For all $x\in\Omega$, $\matrix{A}(x)\in \R^{d\times d}_{\rm sym}$ is a symmetric matrix with $\matrix{A}\in W^{1,\infty}(\Omega; \R^{d\times d}_{\rm sym})$. Moreover, $\matrix{b}(x)\in\R^d$ is a vector with $\matrix{b}\in W^{1,\infty}(\Omega;\R^d)$, and $c(x)\in \R$ is a scalar with $c\in L^\infty(\Omega)$.
To formulate the residual error estimators in~\eqref{ex:nonsymm:estimator:primal}--\eqref{ex:nonsymm:estimator:dual} below, we additionally require that ${\rm div}\,\boldsymbol{f}_2,{\rm div}\,\boldsymbol{g}_2$ exist in $L^2(\Omega)$ elementwise on the initial mesh $\TT_0$ and that the edge jumps satisfy $[\boldsymbol{f}_2\cdot n],[\boldsymbol{g}_2\cdot n]\in L^2(\partial T)$ for all $T\in\TT_0$. (These assumptions are for instance satisfied if $\boldsymbol{f}_2,\boldsymbol{g}_2$ are $\TT_0$-piecewise constant.)
Note that $\operator{L}$ is non-symmetric as
\begin{align}\label{eq:opnonsymm}
\operator{L}w\neq \operator{L}^Tw=-\text{div}\matrix{A}\nabla w -\matrix{b}\cdot\nabla w + (c-\text{div}\matrix{b}) w.
\end{align}
We suppose that the induced bilinear form 
\begin{align*}
 \bform{u}{v}:=\dual{\operator{L}u}{v}
 = \int_\Omega \matrix{A}\nabla u \cdot \nabla v + \matrix{b}\cdot\nabla u v + cu v\,dx
 \quad\text{for }u,v\in\XX:=H^1_0(\Omega)
\end{align*}
is continuous and $H^1_0(\Omega)$-elliptic and hence fits in the frame of Section~\ref{section:motivation}. The right-hand side of~\eqref{eq:primal} reads 
$f(v):=\int_\Omega f_1v-\boldsymbol{f}_2\cdot\nabla v\,dx$.

\subsection{Discretization}\label{afem:discretization}
For a given regular triangulation $\TT_\star$ of $\Omega$
and a polynomial degree $p\ge1$, define 
$\PP^p(\TT_\star):=\set{V\in L^2(\Omega)}{V|_T\text{ is polynomial of degree }\leq p\text{ for all }T\in\TT_\star}.$
We consider $\XX_\star := \SS^p_0(\TT_\star) := \PP^p(\TT_\star)\cap H^1_0(\Omega)$
and let $U_\star,Z_\star\in\XX_\star$ be the unique FEM solutions of~\eqref{eq:primal:discrete} resp.~\eqref{eq:dual:discrete}, i.e.,
\begin{subequations} \label{eq:afem:discrete}
\begin{align}
 &U_\star\in\SS^p_0(\TT_\star)
 \quad\text{such that}\quad
 a(U_\star,V_\star) = f(V_\star)\quad\text{for all }V_\star\in \SS^p_0(\TT_\star),\\
 &Z_\star\in\SS^p_0(\TT_\star)
 \quad\text{such that}\quad
 a(V_\star,Z_\star) = g(V_\star)\quad\text{for all }V_\star\in \SS^p_0(\TT_\star).
\end{align}
\end{subequations}

\subsection{Residual error estimator}\label{afem:estimator}
For $T\in\TT_\star$, let
$h_T:=|T|^{1/d}$ and $\operator{L}|_T V:=-\text{div}|_T \matrix{A}(\nabla V )+ \matrix{b}\cdot \nabla V + c V$. Then, the residual error-estimator for the discrete primal problem~\eqref{eq:primal:discrete} reads
\begin{align}\label{ex:nonsymm:estimator:primal}
\eta_{u,\star}(T)^2:=h_T^2\norm{\operator{L}|_T U_\star - f_1-{\rm div}\,\boldsymbol{f}_2}{L^2(T)}^2 
+ h_T\norm{[(\matrix{A}\nabla U_\star+\boldsymbol{f}_2) \cdot n]}{L^2(\partial T\cap\Omega)}^2.
\end{align}
The residual error-estimator for the discrete dual problem~\eqref{eq:dual:discrete} reads
\begin{align}\label{ex:nonsymm:estimator:dual}
\eta_{z,\star}(T)^2:=h_T^2\norm{\operator{L}^T|_T Z_\star - g_1-{\rm div}\,\boldsymbol{g}_2}{L^2(T)}^2 
+ h_T\norm{[(\matrix{A}\nabla Z_\star+\boldsymbol{g}_2)\cdot n]}{L^2(\partial T\cap\Omega)}^2,
\end{align}
where $\operator{L}^T|_T V:=-\text{div}|_T \matrix{A}(\nabla V )- \matrix{b}\cdot \nabla V + (c-\text{div}\matrix{b}) V$. 

The error estimators satisfy reliability~\eqref{eq:reliable:laxmilgram}; see, e.g.,~\cite{ao00,v96}.
The abstract analysis of Section~\ref{section:motivation} thus results in
\begin{align}\label{eq:afem:estimator}
 |g(u)-g(U_\star)| \lesssim \eta_{u,\star}\eta_{z,\star},
\end{align}
and we aim for optimal convergence of the right-hand side.
Moreover, efficiency and the C\'ea lemma prove that the estimator based approximation class $\A_s$ from Section~\ref{section:main} coincides with the approximation class based on the total error used, e.g., in~\cite{bet,ckns,ms}. The following result is proved in~\cite[Lemma~5.1]{ffp} for $\boldsymbol{f}_2=0=\boldsymbol{g}_2$, but holds verbatim in the present case.

\begin{lemma}\label{lemma:approximationclass}
Let $w\in\{u,z\}$. There holds $w\in\A_s$ if and only if
 \begin{align*}
 \sup_{N\in\N_0}\Big( (N+1)^s \min_{\TT_\star\in\T_N}\big(\min_{V_\star\in\XX_\star}\norm{w-V_\star}{\XX} + {\rm osc}_{w,\star}(V_\star)\big)\Big)<\infty,
 \end{align*}
where ${\rm osc}_{w,\star}(V_\star)^2=\sum_{T\in\TT_\star}{\rm osc}_{w,\star}(T,V_\star)^2$ and
\begin{align*}
 {\rm osc}_{u,\star}^2(T,V_\star)&:=h_T^2\norm{(1-\Pi_T^{2p-2})(\operator{L}|_T V_\star-f_1-{\rm div}\,\boldsymbol{f}_2)}{L^2(T)}^2\\&\qquad
 + h_T\norm{(1-\Pi_{\partial T}^{2p-1})[(\matrix{A}\nabla V_\star+\boldsymbol{f}_2)\cdot n]}{L^2(\partial T\cap\Omega)}^2,\\
 {\rm osc}_{z,\star}^2(T,V_\star)&:=h_T^2\norm{(1-\Pi_T^{2p-2})(\operator{L}^T|_T V_\star-g_1-{\rm div}\,\boldsymbol{g}_2)}{L^2(T)}^2\\
 &\quad+ h_T\norm{(1-\Pi_{\partial T}^{2p-1})[(\matrix{A}\nabla V_\star+\boldsymbol{g}_2)\cdot n]}{L^2(\partial T\cap\Omega)}^2.
\end{align*}
Here, $\Pi_T^q:L^2(T)\to \PP^q(T)$ denotes the $L^2$-orthogonal projection onto polynomials of degree $q$ and $\Pi_{\partial T}^q:L^2(\partial T)\to \PP^q(\SS_{\partial T})$ denotes the $L^2$-orthogonal projection onto
(discontinuous) piecewise polynomials of degree $q$ on the faces of $T$.\hfill\qed
\end{lemma}

\subsection{Verification of axioms}
With newest vertex bisection from~\cite{stevenson:nvb} as mesh-refinement strategy, 
the assumptions of Section~\ref{section:finemesh} are satisfied. It remains to verify 
the axioms~\eqref{ass:stable}--\eqref{ass:reliable},
where $\dist{w}{\TT_\ell}{\TT_\star}:=a(W_\ell-W_\star,W_\ell-W_\star)^{1/2}\simeq \norm{W_\ell-W_\star}{H^1(\Omega)}$ and $W_\ell$ resp. $W_\star$ are the corresponding FEM approximations of $w\in\{u,z\}$.

\begin{theorem}\label{thm:ex:nonsymm}
Consider the model problem of Section~\ref{afem:model}. Then, 
the conforming discretization~\eqref{eq:afem:discrete} with the
residual error estimators~\eqref{ex:nonsymm:estimator:primal}--\eqref{ex:nonsymm:estimator:dual} 
satisfies stability~\eqref{ass:stable}, reduction~\eqref{ass:reduction} with $\q{reduction}=2^{-1/d}$,
quasi-orthogonality~\eqref{ass:orthogonal}, and discrete reliability~\eqref{ass:reliable} with $\RR_w(\TT_\ell,\TT_\star) = \TT_\ell\backslash\TT_\star$ and $w\in\{u,z\}$. In particular, the Algorithms~\ref{algorithm}--\ref{algorithm:bet} are linearly convergent with optimal rates in the sense of Theorem~\ref{theorem:linear}, \ref{theorem:optimal}, \ref{theorem:optimal:mod}, and~\ref{theorem:optimal:bet} for the upper bound in~\eqref{eq:afem:estimator}.
\end{theorem}

\begin{proof}[Proof of Theorem~\ref{thm:ex:nonsymm}, ~\eqref{ass:stable}--\eqref{ass:reduction} and~\eqref{ass:reliable}]
The work~\cite{ckns} considers some symmetric model problem with 
$\matrix{b}=0$ and $c\geq 0$  as well as $\boldsymbol{f}_2=0=\boldsymbol{g}_2$. Stability~\eqref{ass:stable} and reduction~\eqref{ass:reduction} are essentially part of the proof of~\cite[Corollary~3.4]{ckns}. The discrete reliability~\eqref{ass:reliable} 
is  found in~\cite[Lemma~3.6]{ckns}. Both proofs transfer verbatim to the present situation with a non-symmetric differential operator and general $\boldsymbol{f}_2,\boldsymbol{g}_2$.
\end{proof}

\begin{lemma}\label{lem:ex:nonsymm:conv}
In the setting of Theorem~\ref{thm:ex:nonsymm} and for Algorithm~\ref{algorithm}--\ref{algorithm:bet},
the Galerkin approximations $U_\ell$ and $Z_\ell$ converge in the sense of
 \begin{align}\label{eq:aprioriconv}
  \lim_{\ell\to\infty}\norm{U_\infty-U_\ell}{H^1(\Omega)}=0=\lim_{\ell\to\infty}\norm{Z_\infty-Z_\ell}{H^1(\Omega)},
 \end{align}
 for certain $U_\infty,Z_\infty\in H^1_0(\Omega)$. Moreover, there holds at least $U_\infty=u$ or $Z_\infty=z$.
\end{lemma}

\begin{proof}
Note that adaptive mesh-refinement guarantees nestedness $\XX_\ell\subseteq\XX_\star$ for all $\TT_\ell\in\T$ and $\TT_\star\in\refine(\TT_\ell)$.
As in~\cite[Section~3.6]{axioms} or~\cite[Lemma~6.1]{bv}, the C\'ea lemma thus implies {\sl a~priori} convergence in the sense that
there exist $U_\infty,Z_\infty\in \XX_\infty:=\overline{\bigcup_{\ell\in\N_0}\XX_\ell}\subseteq H^1_0(\Omega)$ such that
\begin{align*}
\lim_{\ell\to\infty}\norm{U_\infty-U_\ell}{H^1(\Omega)}=0=\lim_{\ell\to\infty}\norm{Z_\infty-Z_\ell}{H^1(\Omega)}. 
\end{align*}
For $w\in\{u,z\}$, let $\ell_{w,n}$ denote the subsequences which satisfy 
\begin{align*}
\theta\eta_{w,\ell_{w,n}}^2\leq \eta_{w,\ell_{w,n}}(\MM_{w,\ell_{w,n}})^2 \quad\text{for all }n\in\N.
\end{align*}
There holds $\#\set{\ell_{w,n}}{n\in\N}=\infty$ for at least one $w\in\{u,z\}$. While this is obvious for Algorithm~\ref{algorithm} and Algorithm~\ref{algorithm:mod}, it follows for Algorithm~\ref{algorithm:bet} from the proof of Theorem~\ref{theorem:linear}.
For this particular $w$, the estimator reduction 
from Lemma~\ref{lemma:estred} reads
\begin{align*}
 \eta_{w,\ell_{w,n+1}}^2\leq \q{estred}\eta_{w,\ell_{w,n}}^2 + \c{estred}\dist{w}{\TT_{\ell_{w,n+1}}}{\TT_{\ell_{w,n}}}^2\quad\text{for all }n\in\N.
\end{align*}
The {\sl a~priori} convergence~\eqref{eq:aprioriconv} implies $\dist{w}{\TT_{\ell_{w,n+1}}}{\TT_{\ell_{w,n}}}^2\to 0$ as $n\to\infty$. Elementary calculus thus yields $\lim_{n\to\infty}\eta_{w,\ell_{w,n}}=0$; see, e.g.,~\cite[Corollary~4.8]{axioms} resp.~\cite[Lemma~2.3]{estconv}.
Reliability~\eqref{eq:reliable:laxmilgram} of 
$\eta_{w,\ell}$ concludes $\lim_{n\to\infty}\norm{w-W_{\ell_{w,n}}}{H^1(\Omega)}=0$, i.e. $w=W_\infty$.
\end{proof}

\begin{proof}[Proof of Theorem~\ref{thm:ex:nonsymm}, \eqref{ass:orthogonal}]
Recall the sequence $\TT_{\ell_n}$ from~\eqref{ass:orthogonal}. With the {\sl a~priori} convergence of Lemma~\ref{lem:ex:nonsymm:conv}, the proof of~\cite[Lemma~3.5]{ffp} applies and shows the weak convergence in $H^1_0(\Omega)$ for $W_\infty\in\{U_\infty,Z_\infty\}$
\begin{align*}
 \frac{W_\infty-W_{\ell_n}}{\norm{W_\infty-W_{\ell_n}}{H^1(\Omega)}}\rightharpoonup 0 \quad\text{and}\quad  \frac{W_{\ell_{n+1}}-W_{\ell_n}}{\norm{W_{\ell_{n+1}}-W_{\ell_n}}{H^1(\Omega)}}\rightharpoonup 0 \quad\text{as }\ell\to\infty.
\end{align*}
Define $\dist{w}{\TT_\infty}{\cdot}:=a(W_\infty-(\cdot),W_\infty-(\cdot))^{1/2}$.
With this,~\cite[Proposition~3.6]{ffp} applies for the primal as well as the dual problem and shows that given any $\delta>0$, there exists $j_\delta\in\N$ such that all $j\geq j_\delta$ satisfy
\begin{align}\label{eq:qo}
\begin{split}
\dist{w}{\TT_{\ell_{j+1}}}{\TT_{\ell_j}}^2&\leq \frac{1}{1-\delta}\dist{w}{\TT_\infty}{\TT_{\ell_j}}^2 -\dist{w}{\TT_\infty}{\TT_{\ell_{j+1}}}^2.
\end{split}
\end{align}
The discrete reliability~\eqref{ass:reliable} and the convergence~\eqref{eq:aprioriconv} yield
\begin{align}\label{eq:ex:rel}
 \dist{w}{\TT_\infty}{\TT_{\ell_j}} =\lim_{k\to\infty}\dist{w}{\TT_{\ell_k}}{\TT_{\ell_j}}\leq \c{rel}\eta_{w,\ell_j}.
\end{align}
With~\eqref{eq:qo}--\eqref{eq:ex:rel}, the quasi-monotonicity from Lemma~\ref{lemma:monotone} (since~\eqref{ass:stable},~\eqref{ass:reduction},~\eqref{ass:reliable} have already been verified)  implies for $\delta=1-1/(1+\eps \c{reliable}^{-2})$ and hence $1/(1-\delta)=1+\eps \c{reliable}^{-2}$ that
\begin{align}\label{eq:qohelp}
\nonumber
 \sum_{j=n}^N\big(&\dist{w}{\TT_{\ell_{j+1}}}{\TT_{\ell_j}}^2-\eps \c{reliable}^{-2} \dist{w}{\TT_\infty}{\TT_{\ell_j}}^2\big)\\
\begin{split}
 &\overset{\eqref{eq:qo}}{\leq}\sum_{j=j_\delta}^N\big( (\frac{1}{1-\delta}-\eps \c{reliable}^{-2})\dist{w}{\TT_\infty}{\TT_{\ell_j}}^2 -\dist{w}{\TT_\infty}{\TT_{\ell_{j+1}}}^2\big)+\sum_{j=n}^{j_\delta-1}\dist{w}{\TT_{\ell_{j+1}}}{\TT_{\ell_j}}^2\\
 &\leq\dist{w}{\TT_\infty}{\TT_{\ell_{j_\delta}}}^2 
+\c{reliable}^2\sum_{j=n}^{j_\delta-1}\eta_{w,\ell_j}^2\overset{\eqref{eq:ex:rel}}{\leq} (1+j_\delta)\c{reliable}^2\c{mon}\eta_{w,\ell_n}^2.
\end{split}
\end{align}
Another application of the reliability~\eqref{eq:ex:rel} shows
\begin{align*}
 \sum_{j=n}^N\big(\dist{w}{\TT_{\ell_{j+1}}}{\TT_{\ell_j}}^2-\eps \eta_{w,\ell_j}^2\big)&\overset{\eqref{eq:ex:rel}}{\leq}\sum_{j=n}^N\big(\dist{w}{\TT_{\ell_{j+1}}}{\TT_{\ell_j}}^2-\eps \c{reliable}^{-2} \dist{w}{\TT_\infty}{\TT_{\ell_j}}^2\big)\\
 &\overset{\eqref{eq:qohelp}}{\leq} (1+j_\delta)\c{reliable}^2\c{mon}\eta_{w,\ell_n}^2.
\end{align*}
This proves~\eqref{ass:orthogonal} with $\c{orth}(\eps):= (1+j_\delta)\c{reliable}^2\c{mon}$.
\end{proof}

\begin{figure}
\includegraphics[scale=0.5]{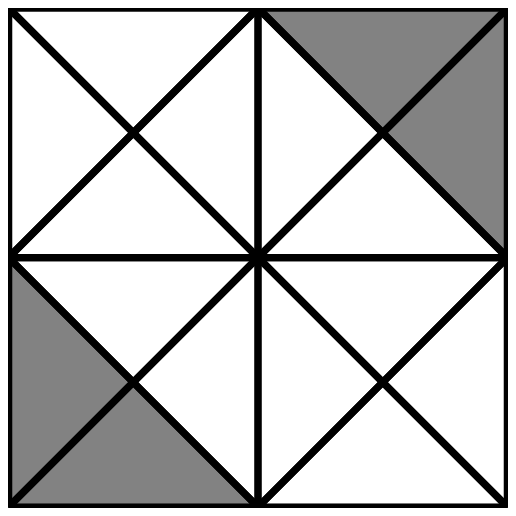}
\includegraphics[scale=0.4]{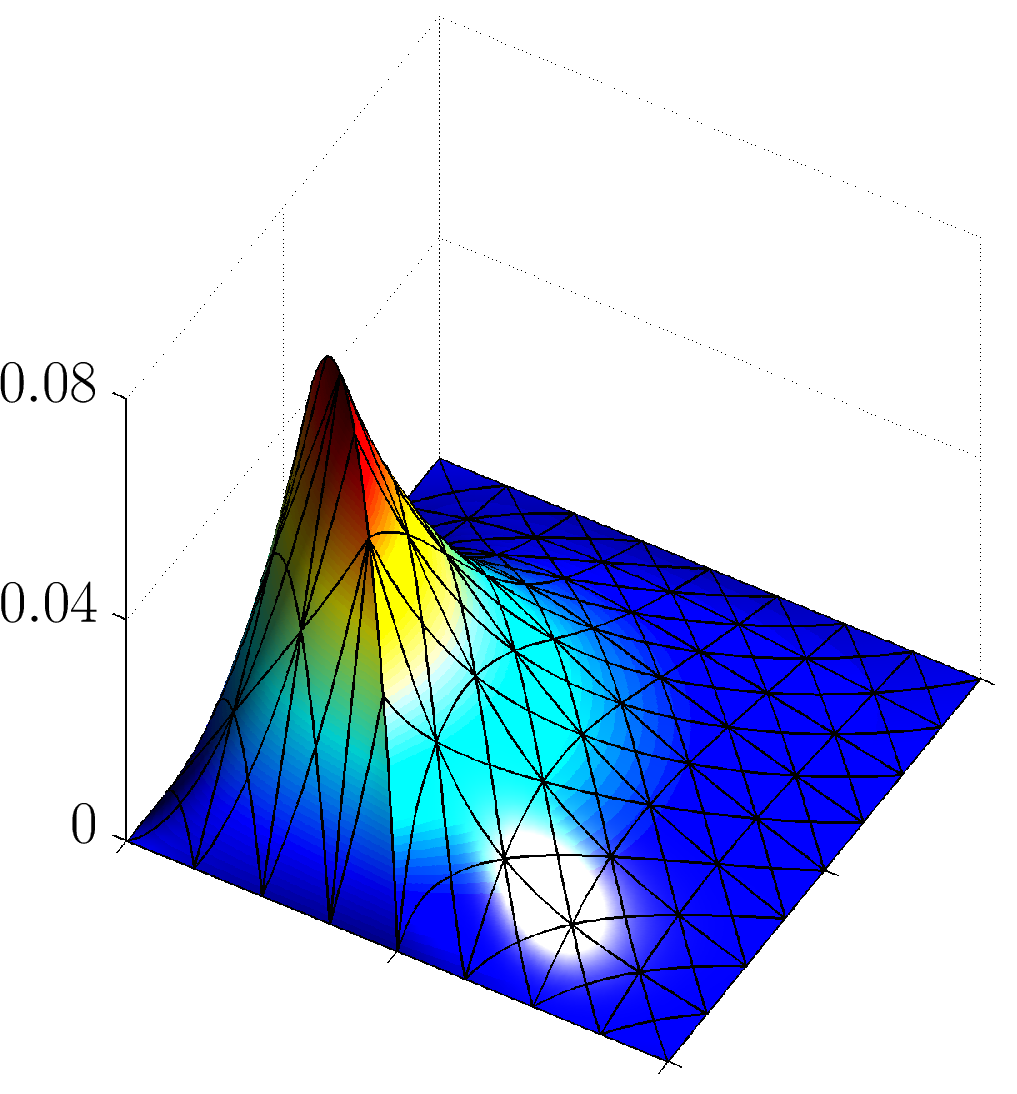}
\includegraphics[scale=0.4]{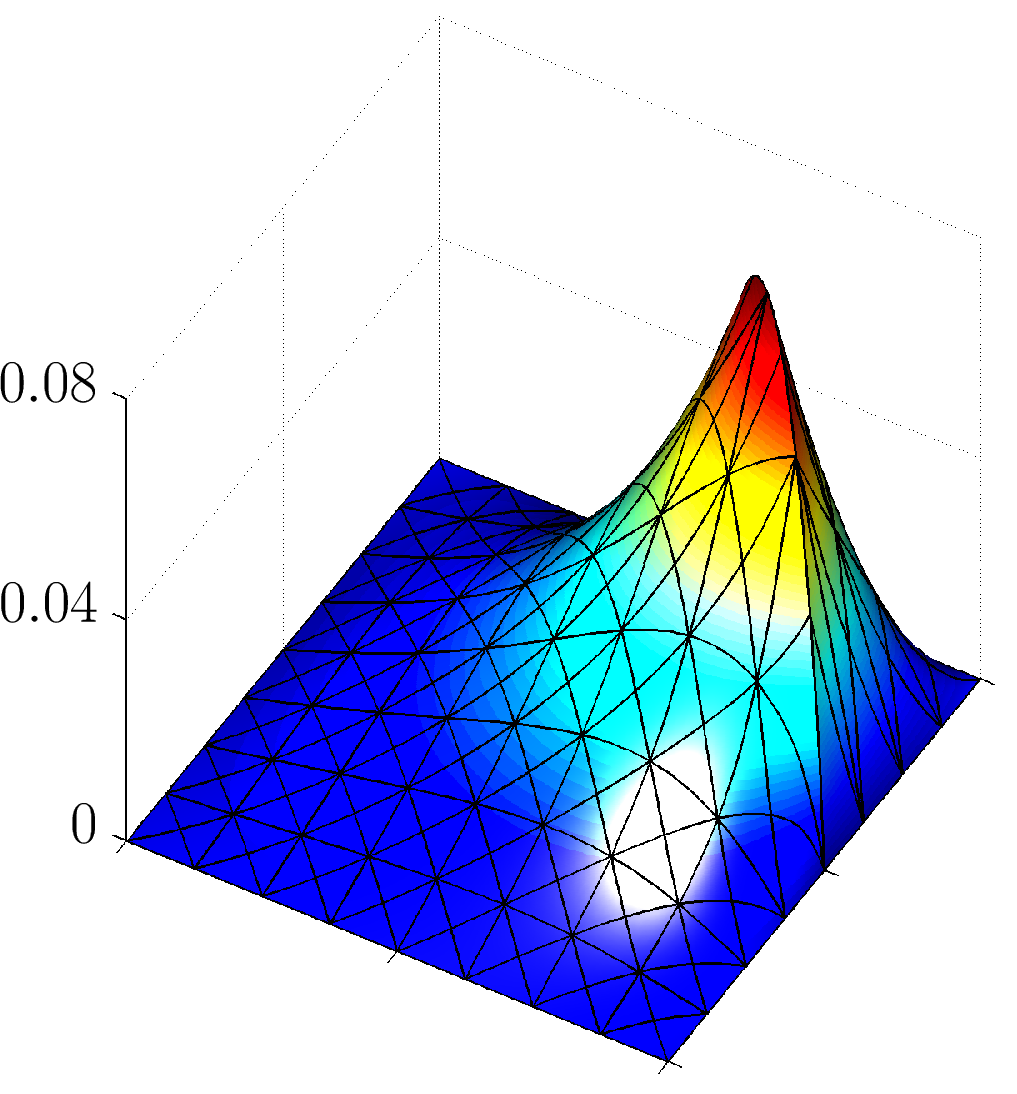}
\caption{Example from Section~\ref{example1:afem}: The initial mesh~$\TT_0$ (left) and the triangles $T_f$ (bottom left) and $T_g$ (top right) indicated in gray. An approximation to the primal solution (middle) and dual solution (right) on a uniform mesh with $256$ elements, where the singularities of both are clearly visible.} 
\label{fig:MStestcase}
\end{figure}

\begin{figure}
\psfrag{0.1}{\scalebox{.5}{$0.1$}}
\psfrag{0.2}{\scalebox{.5}{$0.2$}}
\psfrag{0.3}{\scalebox{.5}{$0.3$}}
\psfrag{0.4}{\scalebox{.5}{$0.4$}}
\psfrag{0.5}{\scalebox{.5}{$0.5$}}
\psfrag{0.6}{\scalebox{.5}{$0.6$}}
\psfrag{0.7}{\scalebox{.5}{$0.7$}}
\psfrag{0.8}{\scalebox{.5}{$0.8$}}
\psfrag{0.9}{\scalebox{.5}{$0.9$}}
\psfrag{1.0}{\scalebox{.5}{$1.0$}}
\psfrag{estu}{\tiny$\eta_u$}
\psfrag{estz}{\tiny$\eta_z$}
\psfrag{estz*estu}{\tiny$\eta_u\eta_z$}
\psfrag{est}[c][c]{\tiny estimators}
\psfrag{error}[c][c]{\tiny error resp. estimators}
\psfrag{err}[c][c]{\tiny error}
\psfrag{Algorithm A}[c][c]{\tiny Algorithm A}
\psfrag{Algorithm B}[c][c]{\tiny Algorithm B}
\psfrag{Algorithm C}[c][c]{\tiny Algorithm C}
\psfrag{o3}{\tiny $\mathcal{O}(N^{-3})$}
\psfrag{o32}{\tiny $\mathcal{O}(N^{-3/2})$}
\psfrag{o1}{\tiny $\mathcal{O}(N^{-1})$}
\psfrag{nE}[c][c]{\tiny number of elements $N=\#\TT_\ell$}
\includegraphics[scale=0.4]{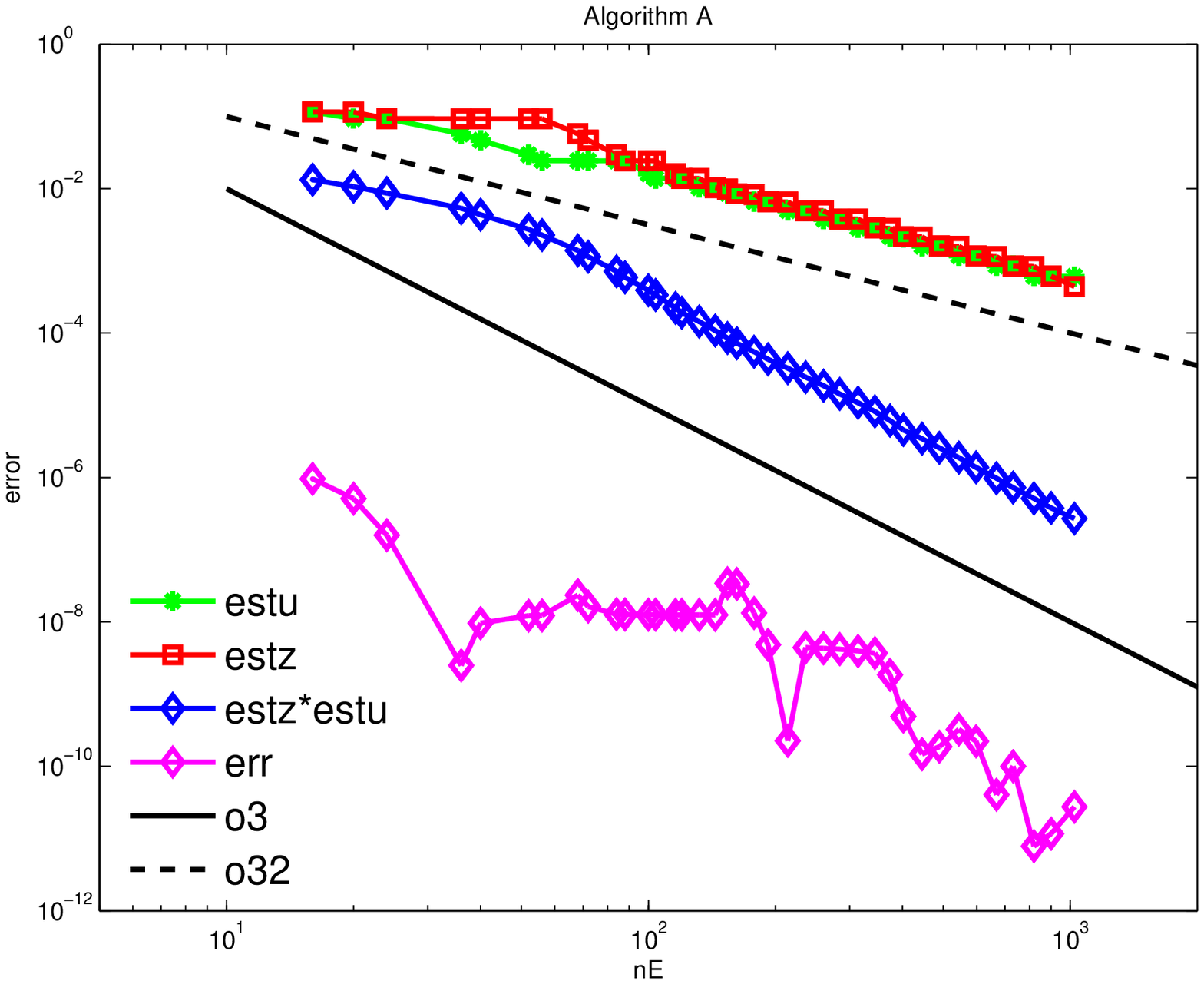}
\qquad
\includegraphics[scale=0.4]{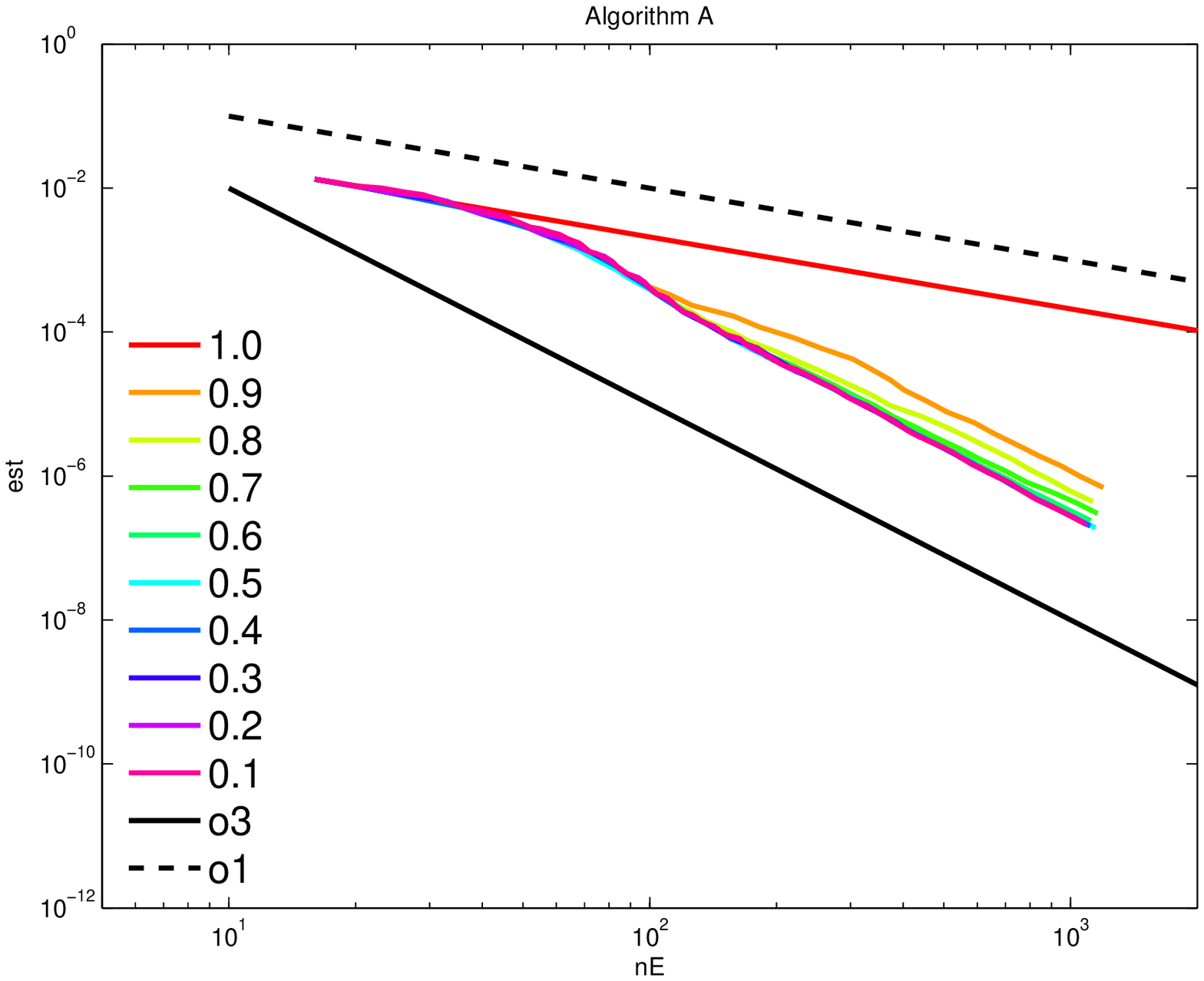}
\\[1ex]
\includegraphics[scale=0.4]{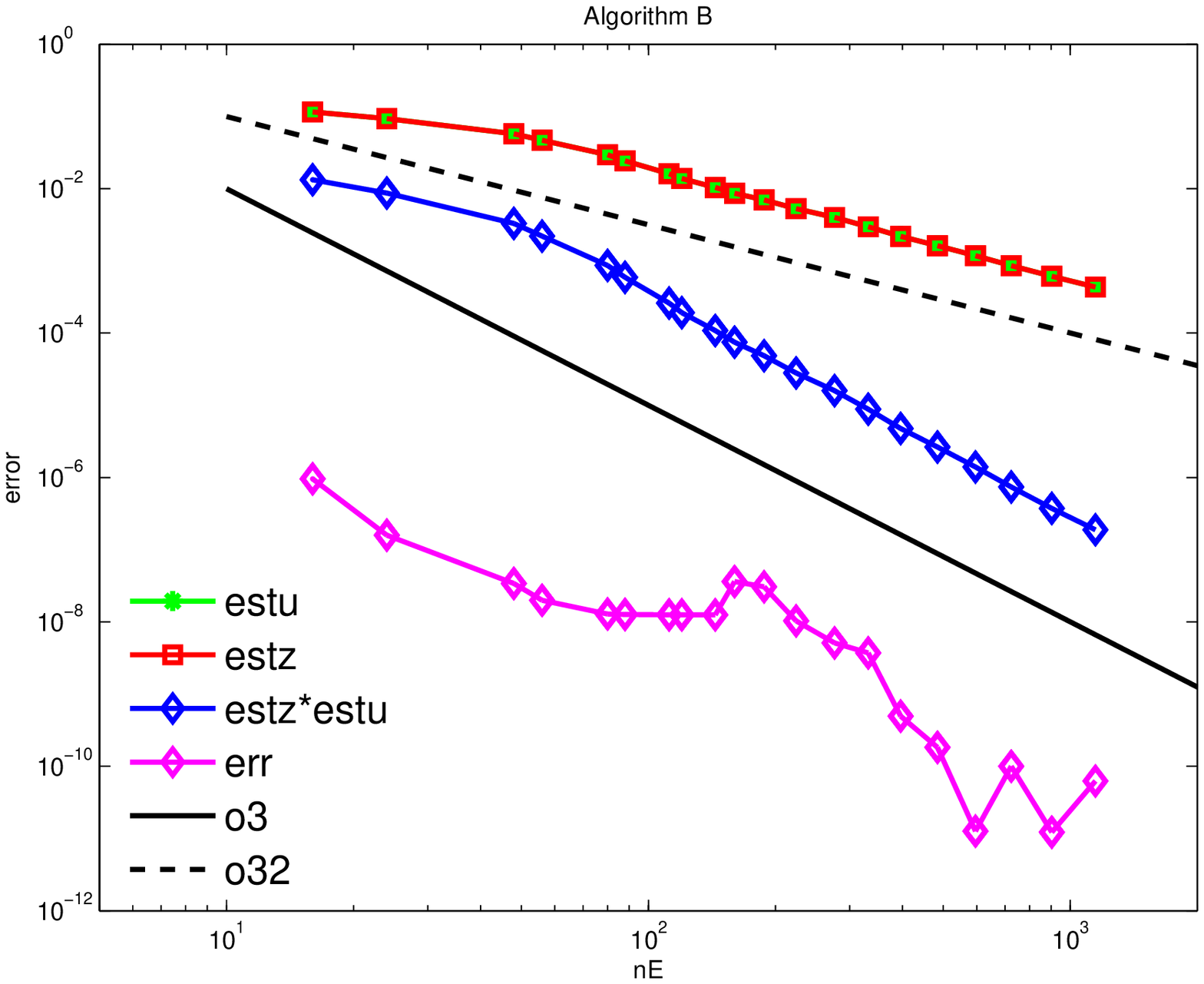}
\qquad
\includegraphics[scale=0.4]{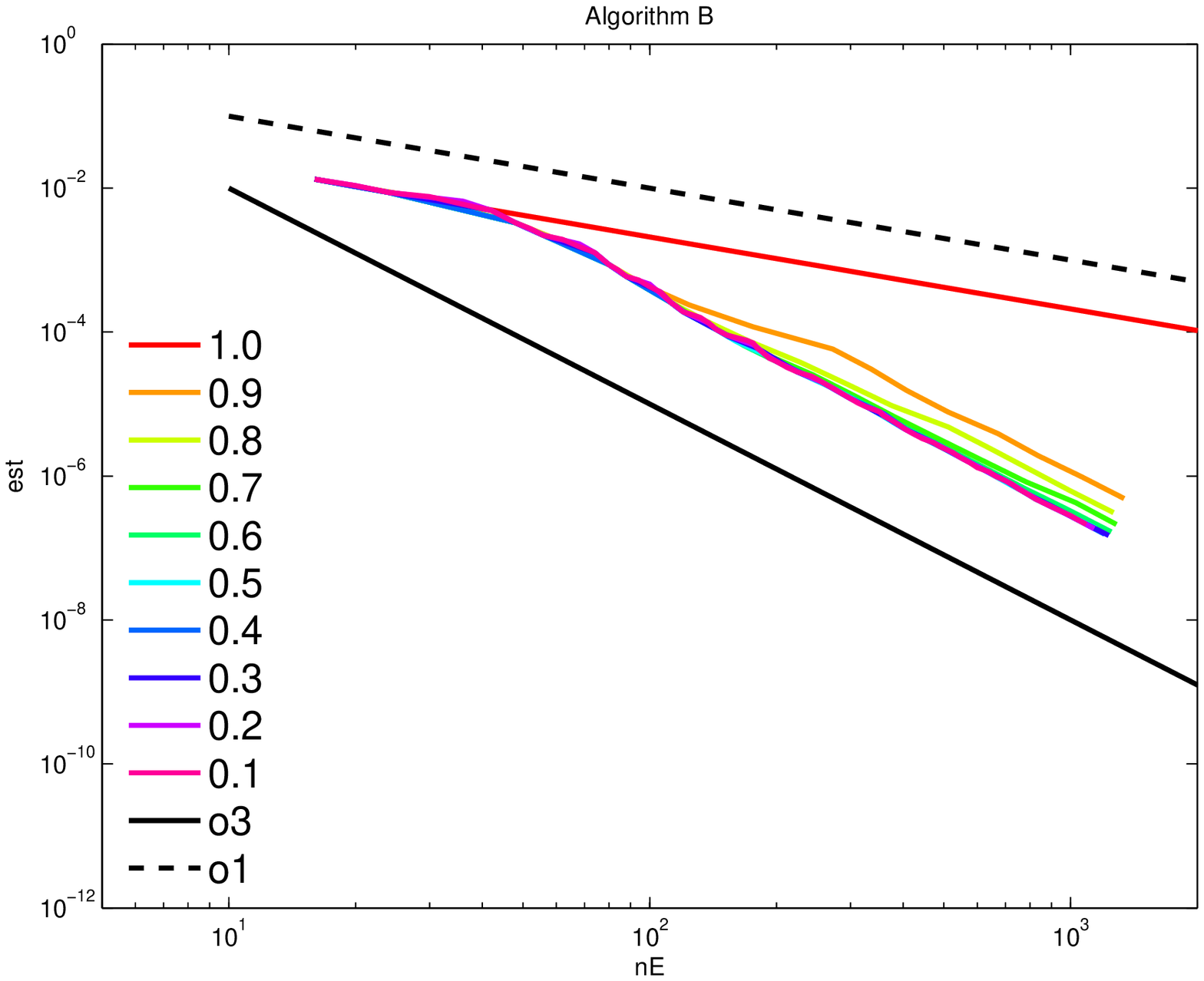}
\\[1ex]
\includegraphics[scale=0.4]{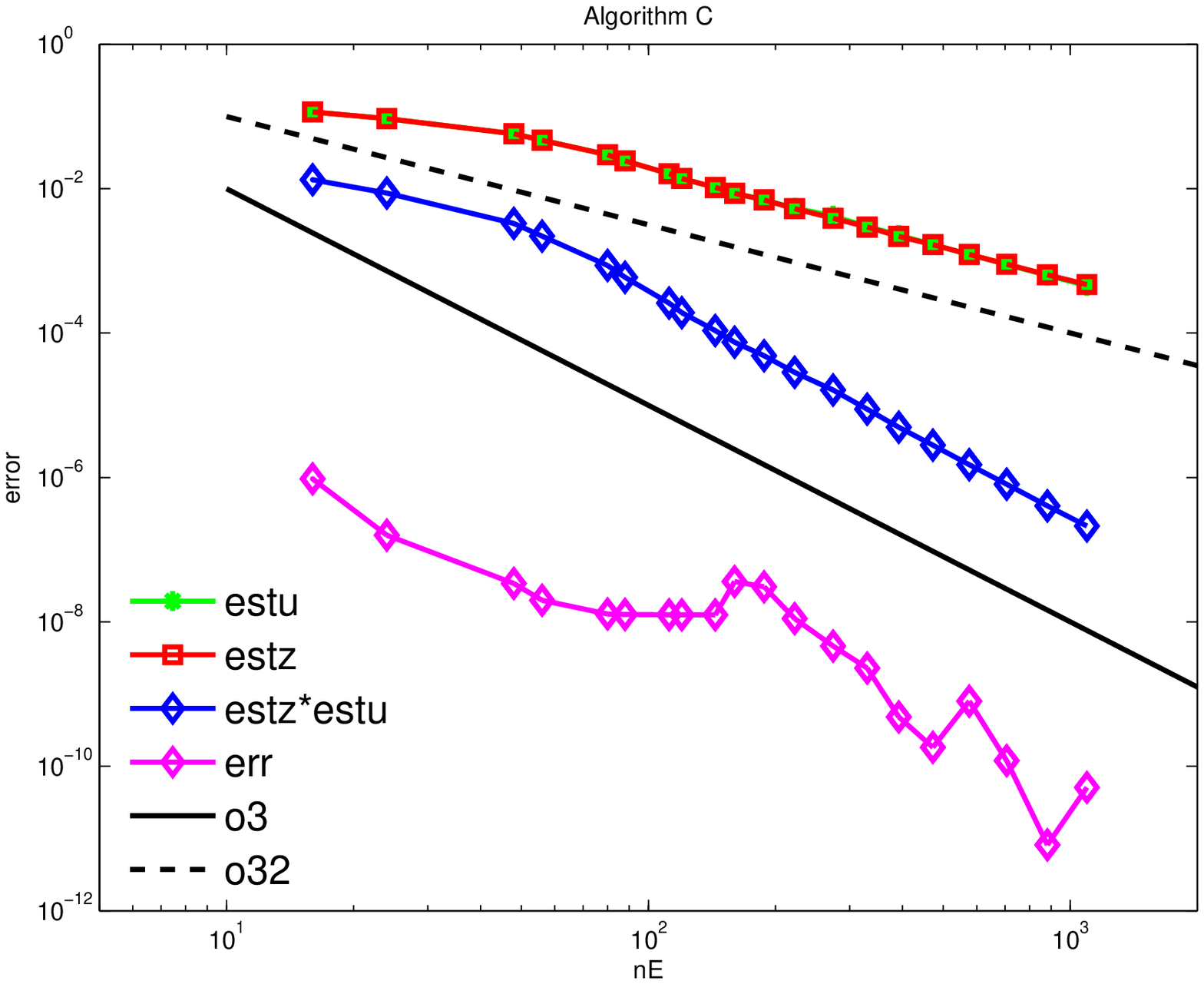}
\qquad
\includegraphics[scale=0.4]{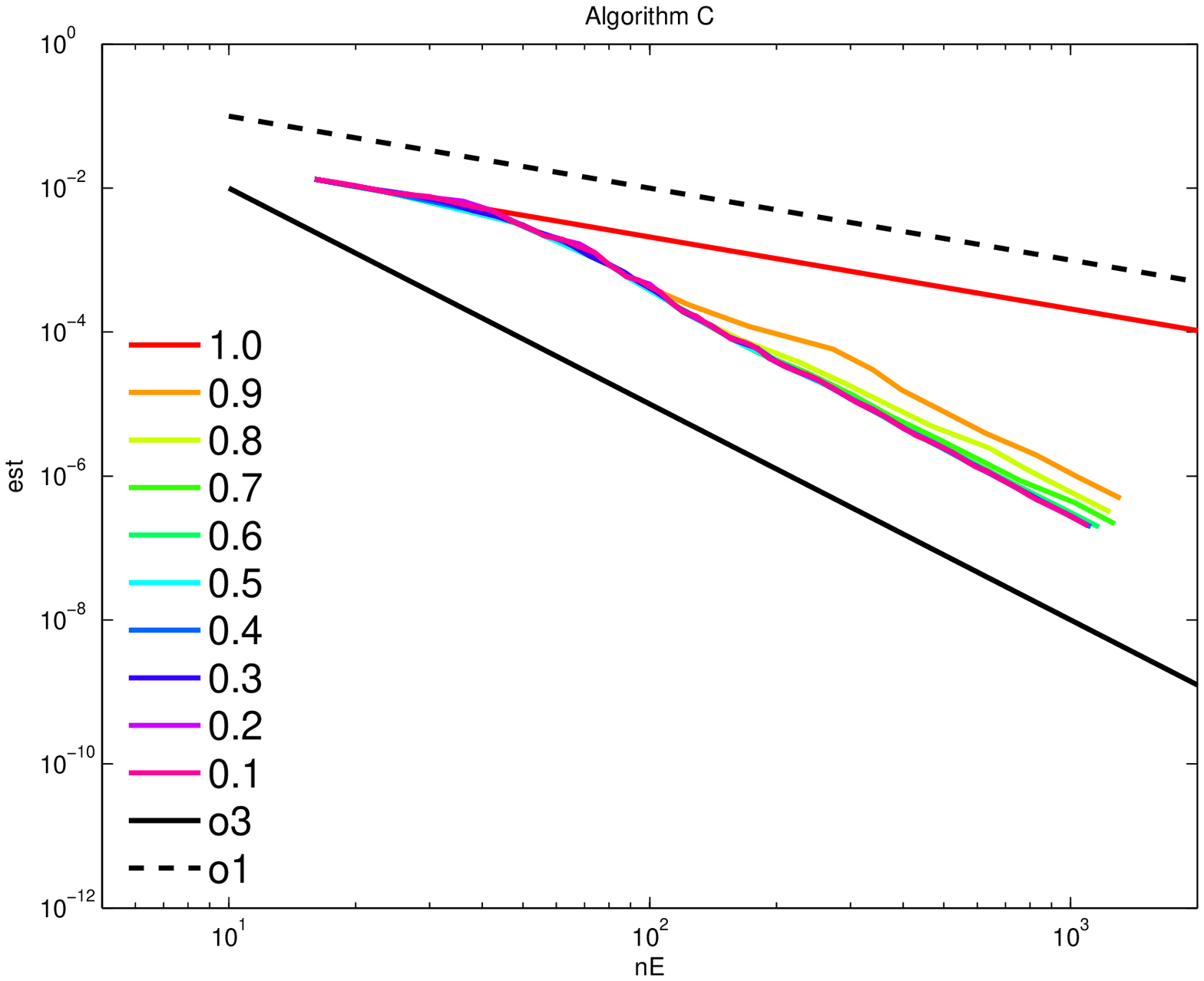}
\caption{Example from Section~\ref{example1:afem}: 
Over the numbers of elements $\#\TT_\ell$, we plot the estimators $\eta_{u,\ell}$ and $\eta_{z,\ell}$, the estimator product $\eta_{u,\ell}\eta_{z,\ell}$, as well as the goal error $|g(u) - g(U_\ell)|$ as output of Algorithm~\ref{algorithm}--\ref{algorithm:bet} with $\theta=0.5$ (left) resp.\ the estimator product for various $\theta\in\{0.1,\dots,0.9\}$ as well as for $\theta=1.0$ which corresponds to uniform refinement. }
 \label{fig:fem1:MSconv}
\end{figure}

\begin{figure}
\psfrag{0.1}{\scalebox{.5}{$0.1$}}
\psfrag{0.2}{\scalebox{.5}{$0.2$}}
\psfrag{0.3}{\scalebox{.5}{$0.3$}}
\psfrag{0.4}{\scalebox{.5}{$0.4$}}
\psfrag{0.5}{\scalebox{.5}{$0.5$}}
\psfrag{0.6}{\scalebox{.5}{$0.6$}}
\psfrag{0.7}{\scalebox{.5}{$0.7$}}
\psfrag{0.8}{\scalebox{.5}{$0.8$}}
\psfrag{0.9}{\scalebox{.5}{$0.9$}}
\psfrag{1.0}{\scalebox{.5}{$1.0$}}
\psfrag{estu}{\tiny$\eta_u$}
\psfrag{estz}{\tiny$\eta_z$}
\psfrag{estz*estu}{\tiny$\eta_u\eta_z$}
\psfrag{est}[c][c]{\tiny estimators}
\psfrag{error}[c][c]{\tiny error resp. estimators}
\psfrag{err}[c][c]{\tiny error}
\psfrag{Algorithm AFEM (P)}[c][c]{\tiny Standard AFEM (primal)}
\psfrag{Algorithm AFEM (D)}[c][c]{\tiny {\centering Standard AFEM (dual)}}
\psfrag{o2}{\tiny $\mathcal{O}(N^{-2})$}
\psfrag{o32}{\tiny $\mathcal{O}(N^{-3/2})$}
\psfrag{o1}{\tiny $\mathcal{O}(N^{-1})$}
\psfrag{nE}[c][c]{\tiny number of elements $N=\#\TT_\ell$}
\includegraphics[scale=0.4]{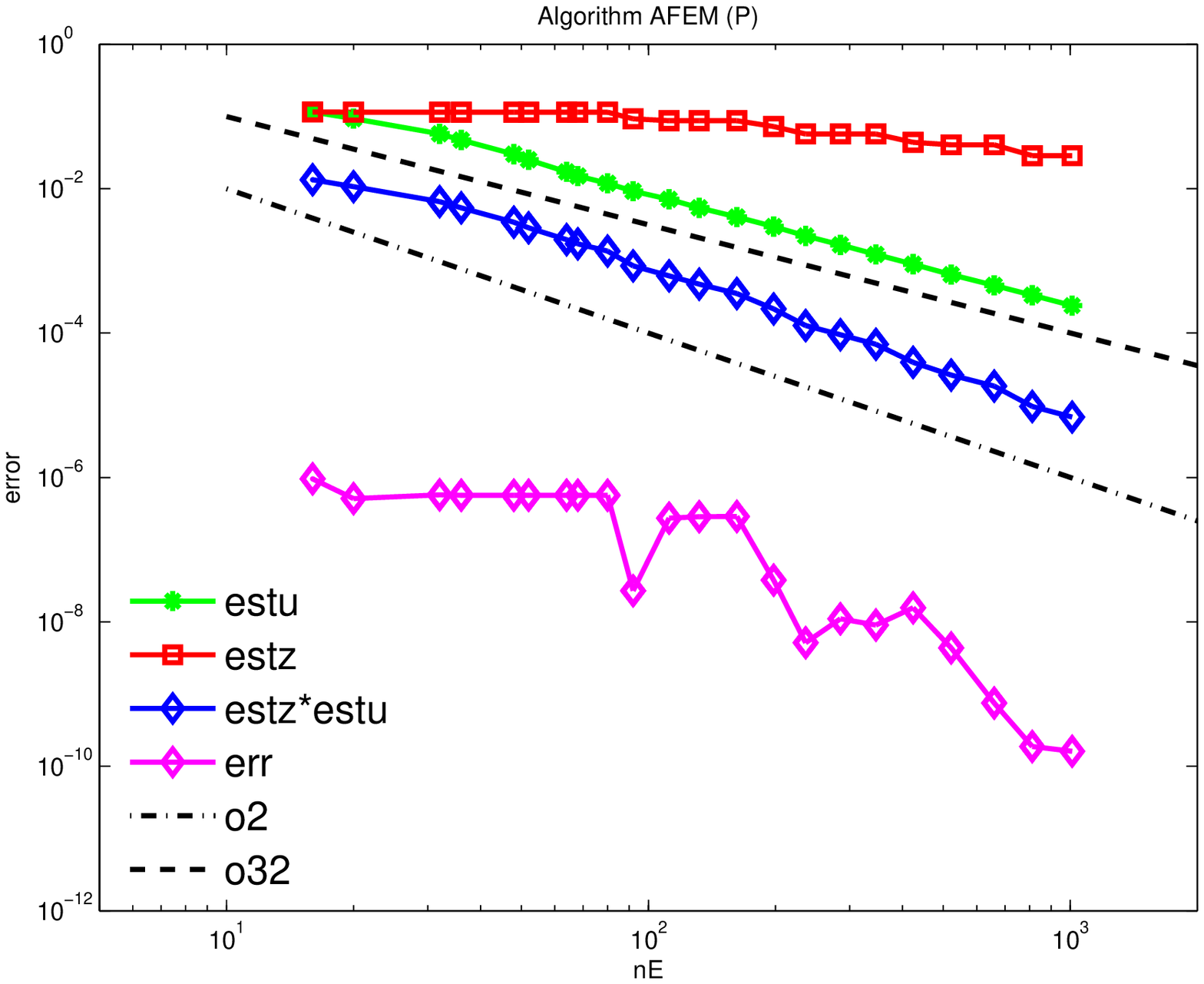}
\qquad
\includegraphics[scale=0.4]{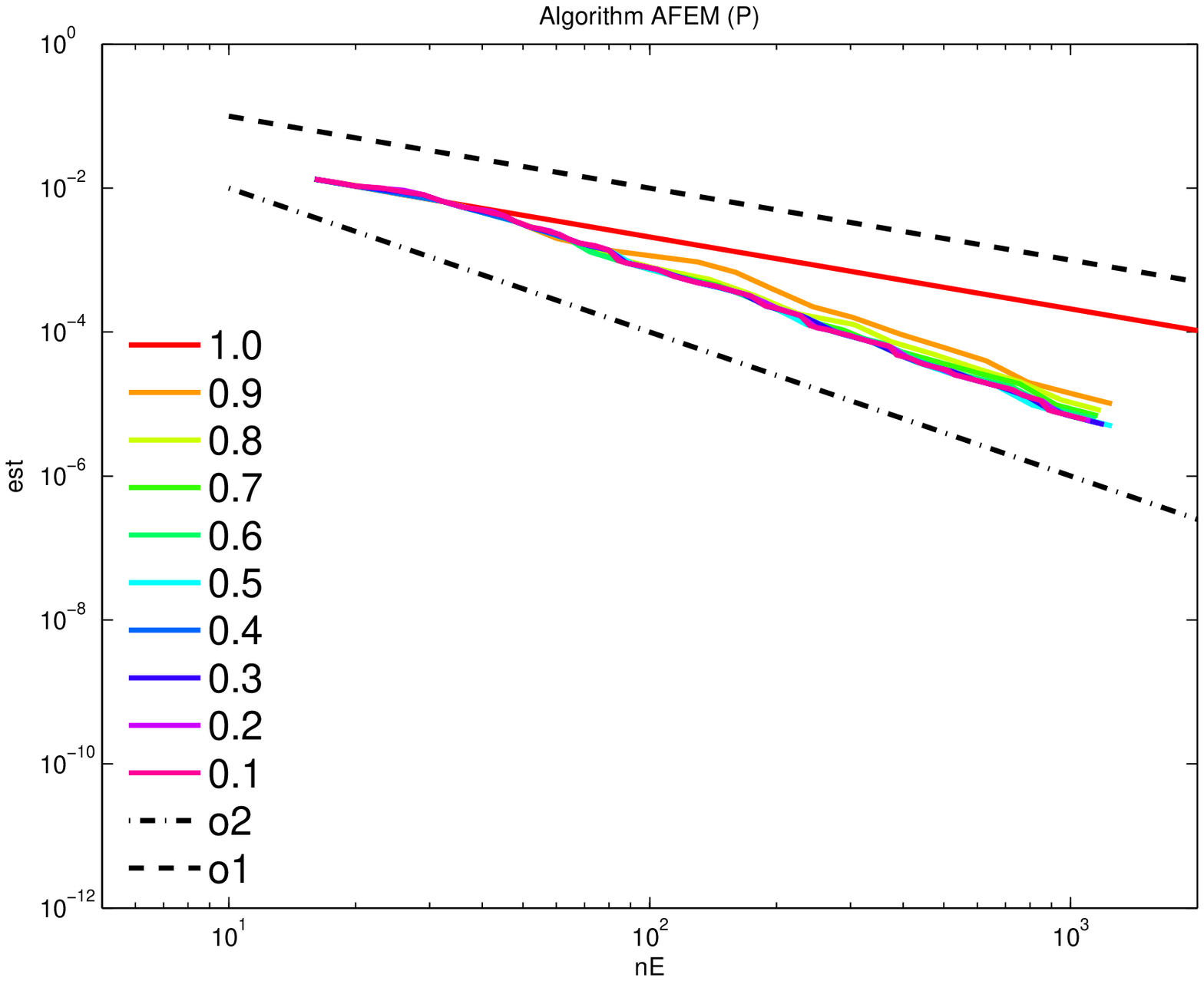}
\\[1ex]
\includegraphics[scale=0.4]{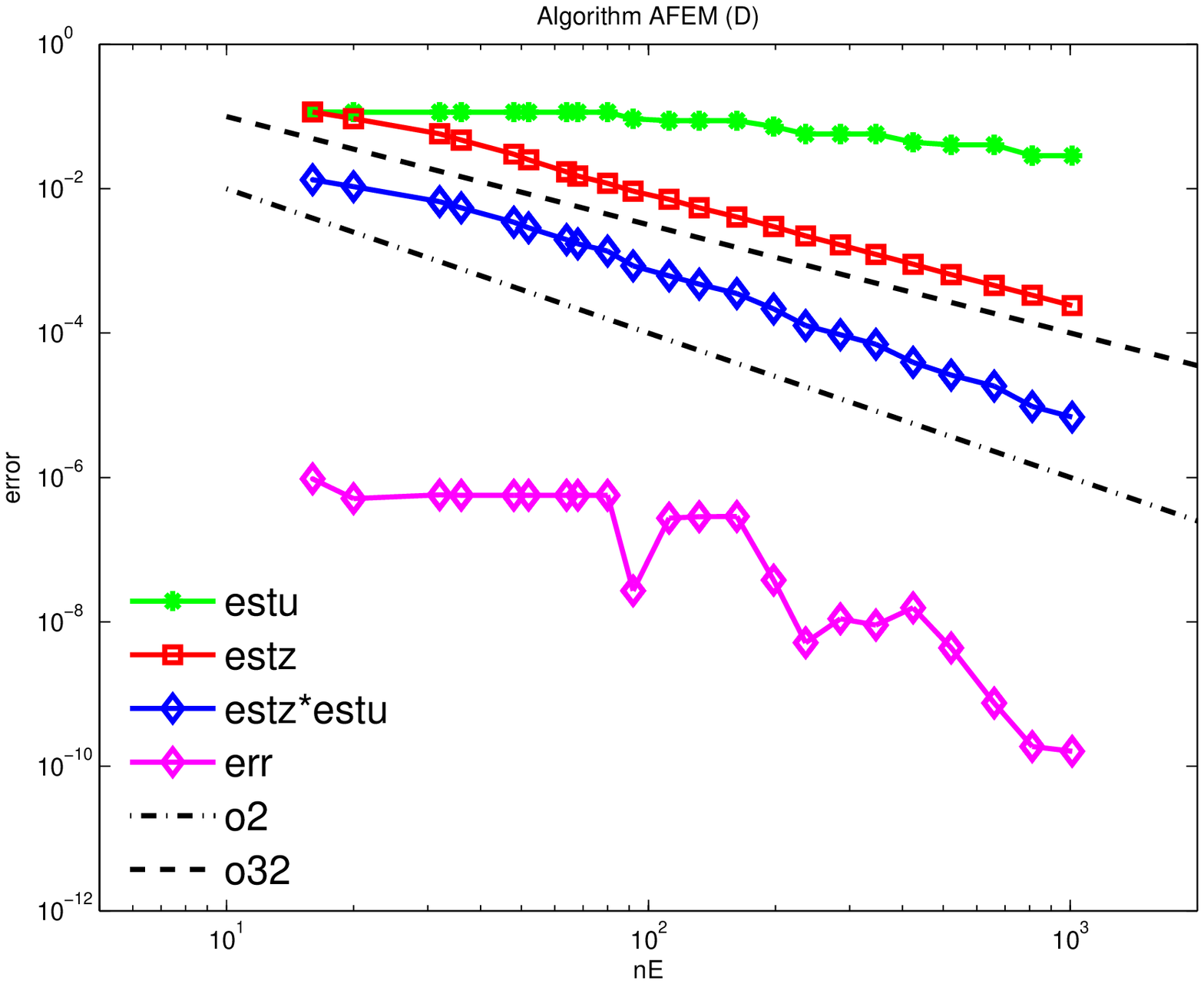}
\qquad
\includegraphics[scale=0.4]{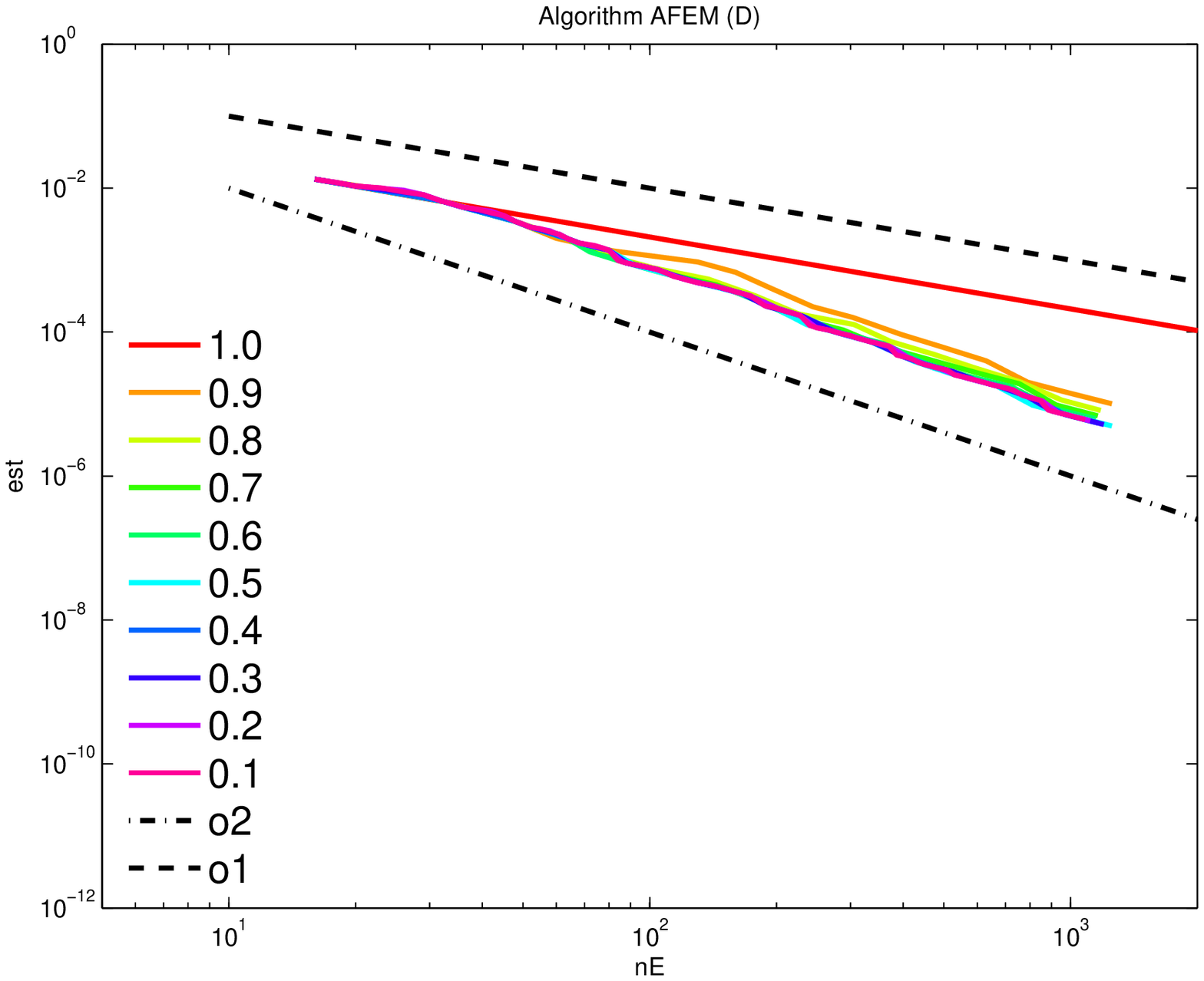}
\caption{Example from Section~\ref{example1:afem}: 
Output of standard (non-goal-oriented) AFEM algorithms with $\theta=0.5$, where adaptive mesh-refinement is steered only by the primal error estimator $\eta_{u,\ell}$ (top) resp.\ the dual error estimator $\eta_{z,\ell}$ (bottom). Over the numbers of elements $\#\TT_\ell$, we plot the estimators $\eta_{u,\ell}$ and $\eta_{z,\ell}$, the estimator product $\eta_{u,\ell}\eta_{z,\ell}$, as well as the goal error $|g(u) - g(U_\ell)|$ (left) resp.\ the estimator product for various $\theta\in\{0.1,\dots,0.9\}$ as well as for $\theta=1.0$ which corresponds to uniform refinement.}
\label{fig:fem1:QconvPD}
\end{figure}

\begin{figure}
{\tiny 
\begin{tabular}{c@{\quad}c@{\quad}c@{\quad}c@{\quad}c}
{\tiny Algorithm A}
&
{\tiny Algorithm B}
&
{\tiny Algorithm C}
&
{\tiny AFEM (primal)}
&
{\tiny AFEM (dual)}
\\
\includegraphics[scale=0.125]{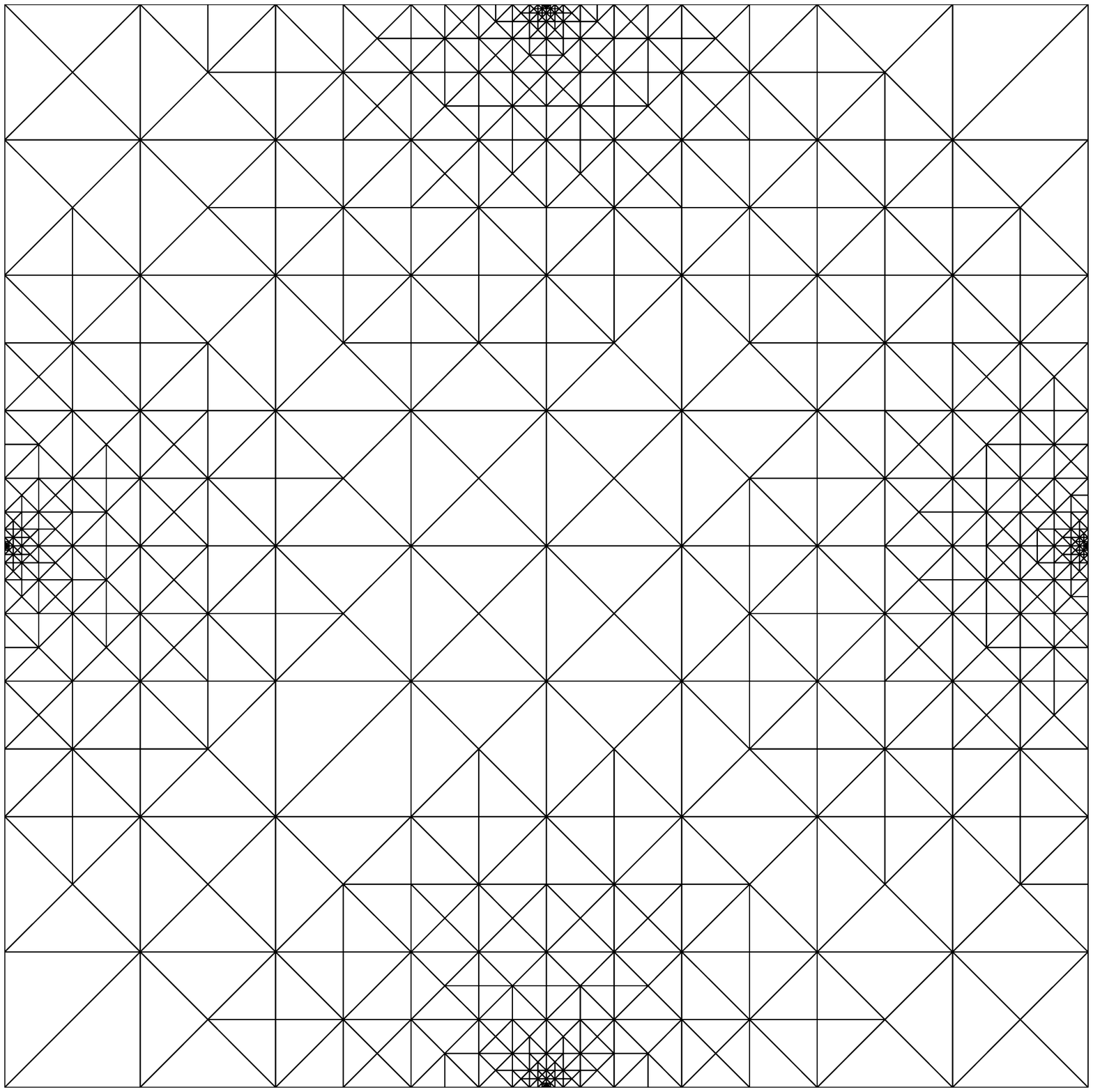}
&
\includegraphics[scale=0.125]{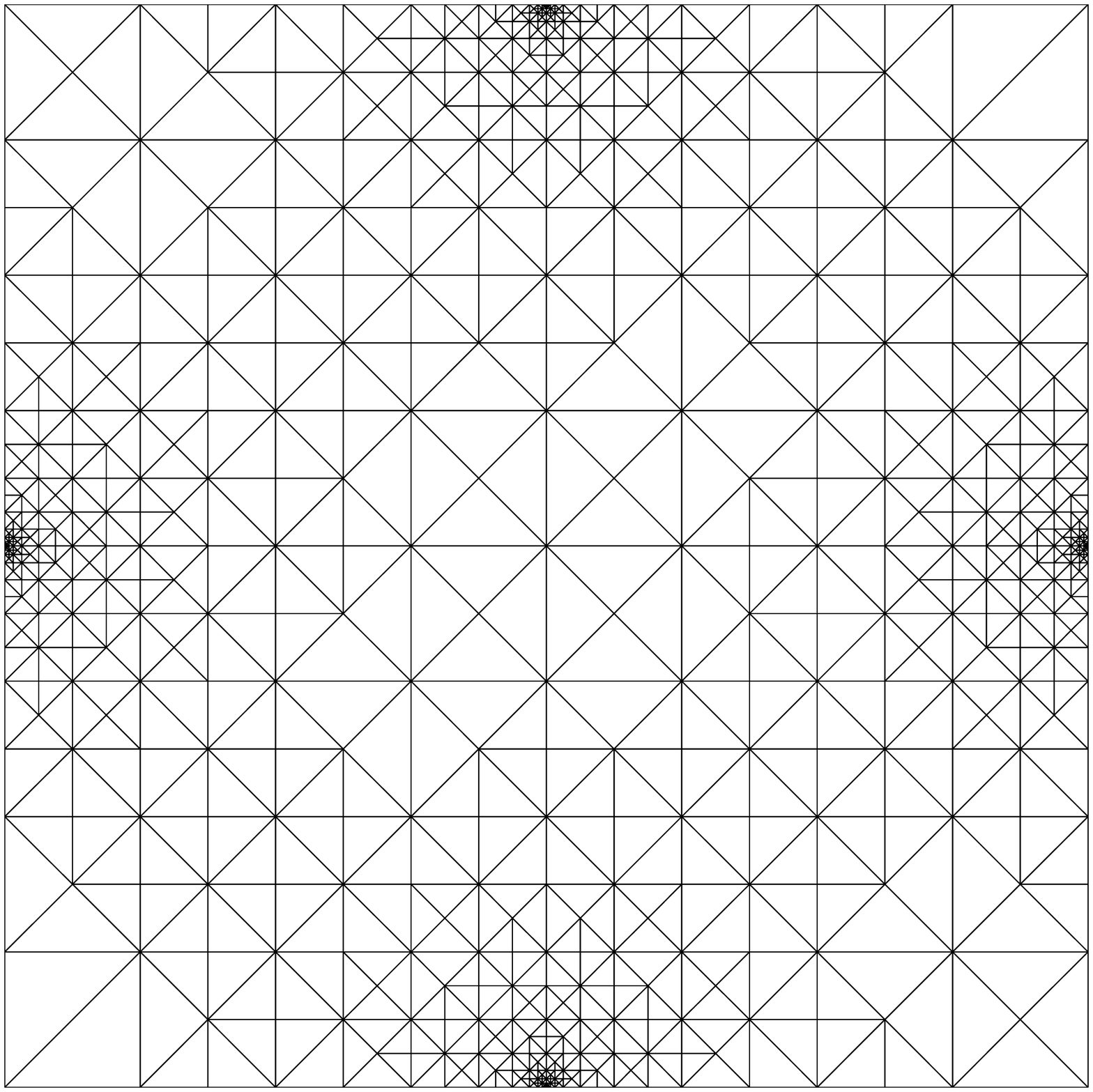}
&
\includegraphics[scale=0.125]{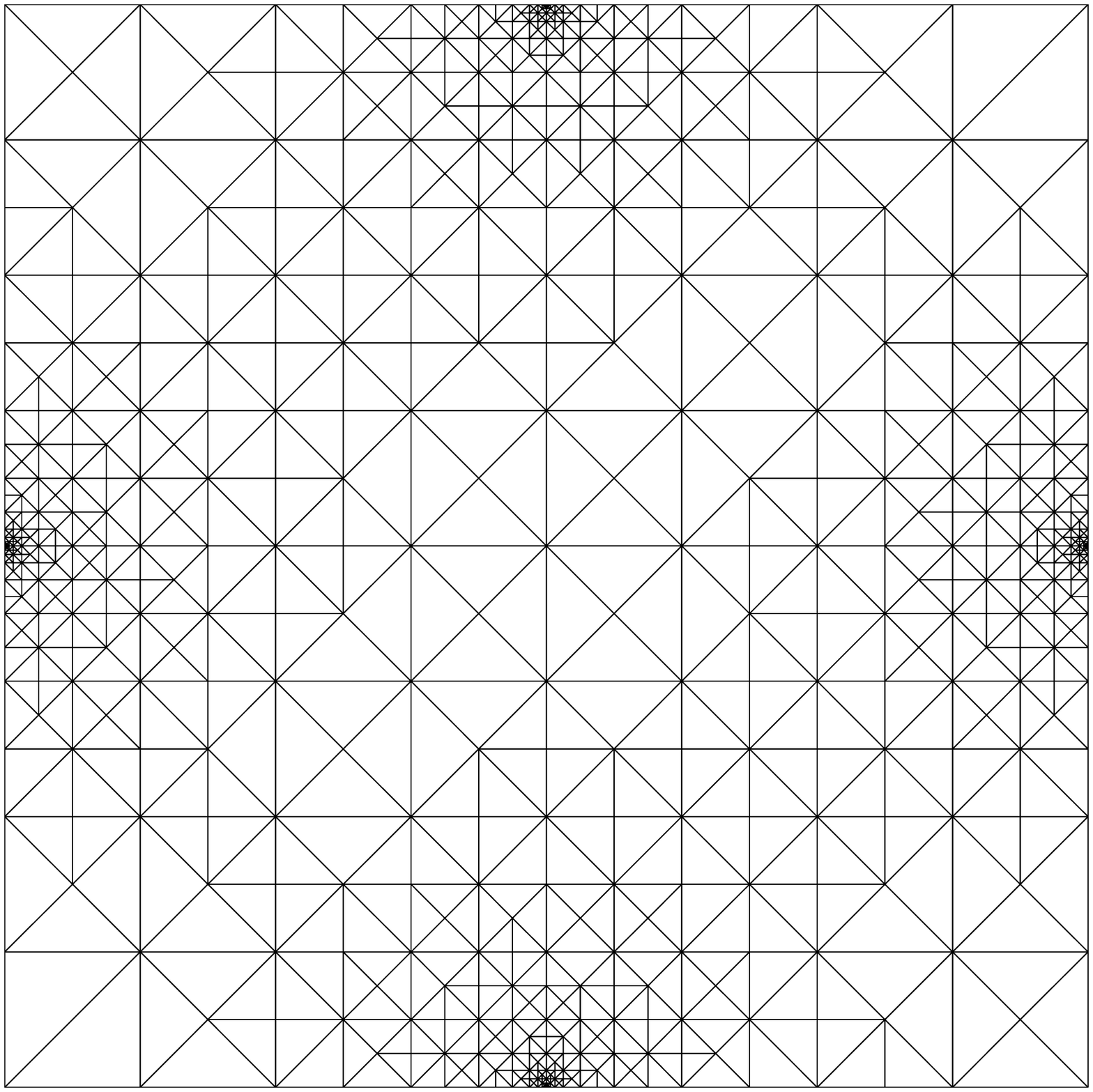}
&
\includegraphics[scale=0.125]{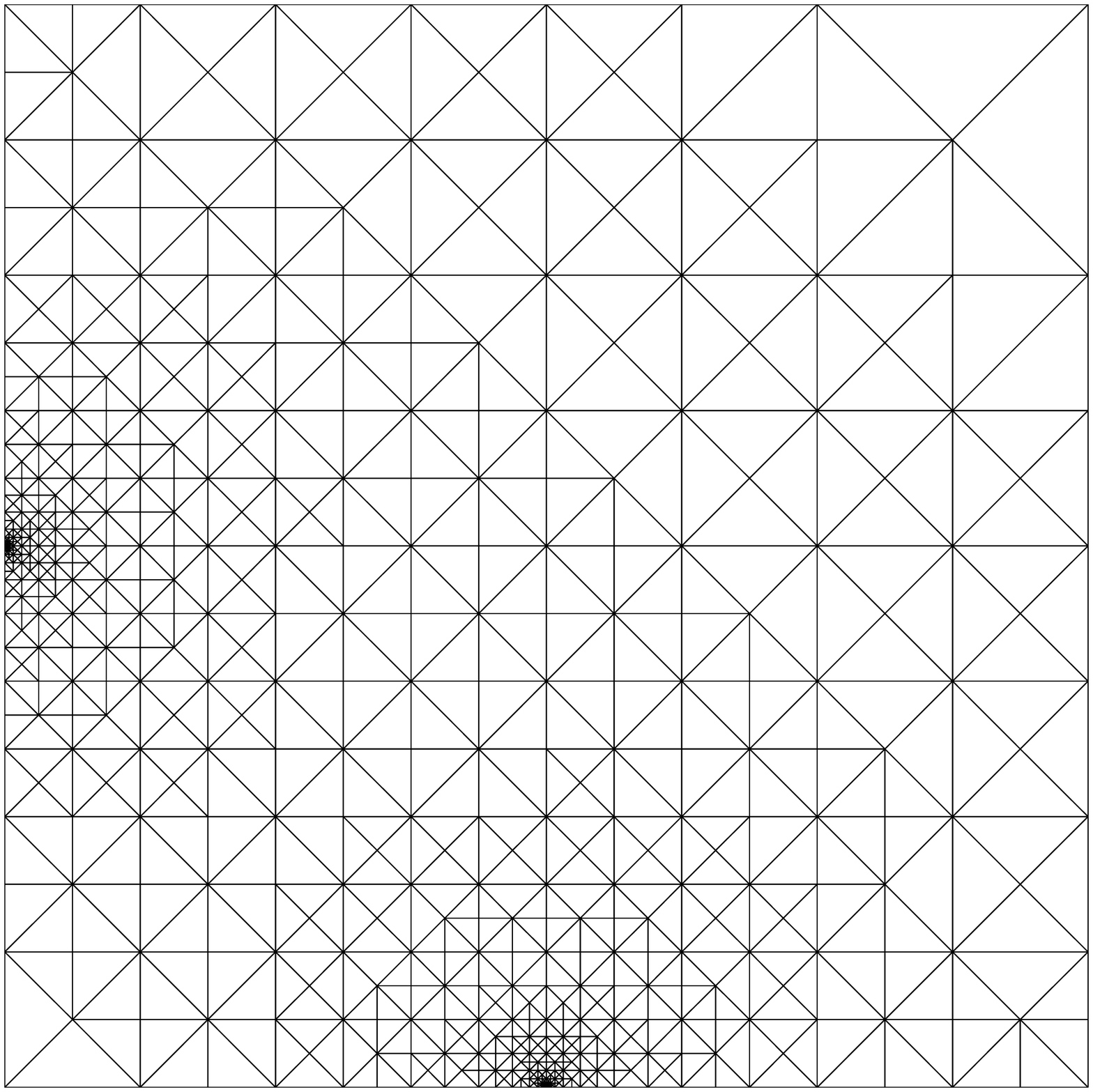}
&
\includegraphics[scale=0.125]{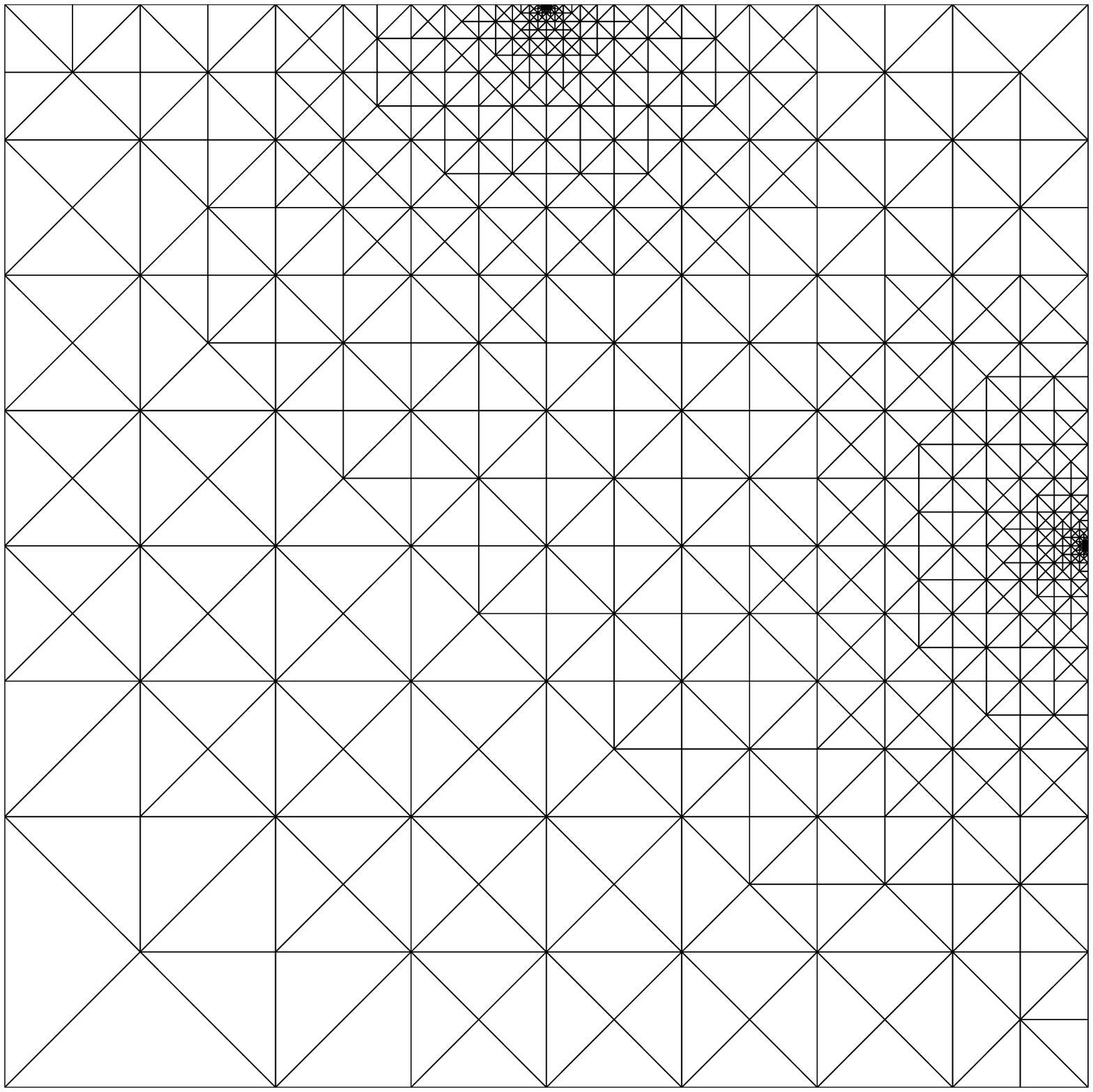}
\\
{\tiny $\#\TT_{38} = 1{,}022$}
&
{\tiny $\#\TT_{20} = 1{,}146$}
&
{\tiny $\#\TT_{20} = 1{,}094$}
&
{\tiny $\#\TT_{22} = 1{,}010$}
&
{\tiny $\#\TT_{22} = 1{,}010$}
\end{tabular}
}
\caption{Example from Section~\ref{example1:afem}: Meshes generated by Algorithm~\ref{algorithm}, \ref{algorithm:mod}, and~\ref{algorithm:bet} as well as standard (non-goal-oriented) AFEM driven by the primal error estimator resp.\ the dual error estimator (from left to right) for $\theta=0.5$.}
\label{fig:fem1:meshes}
\end{figure}

\begin{figure}
\psfrag{A}{\scalebox{.5}{Algorithm A}}
\psfrag{B}{\scalebox{.5}{Algorithm B}}
\psfrag{C}{\scalebox{.5}{Algorithm C}}
\psfrag{P}{\scalebox{.5}{adaptive algorithm for primal problem}}
\psfrag{D}{\scalebox{.5}{adaptive algorithm for dual problem}}
\psfrag{theta}[t]{\tiny{}parameter $\theta$}
\psfrag{ncum}{\tiny{}$N_{\mathrm{cum}}$}
\includegraphics[scale=0.4]{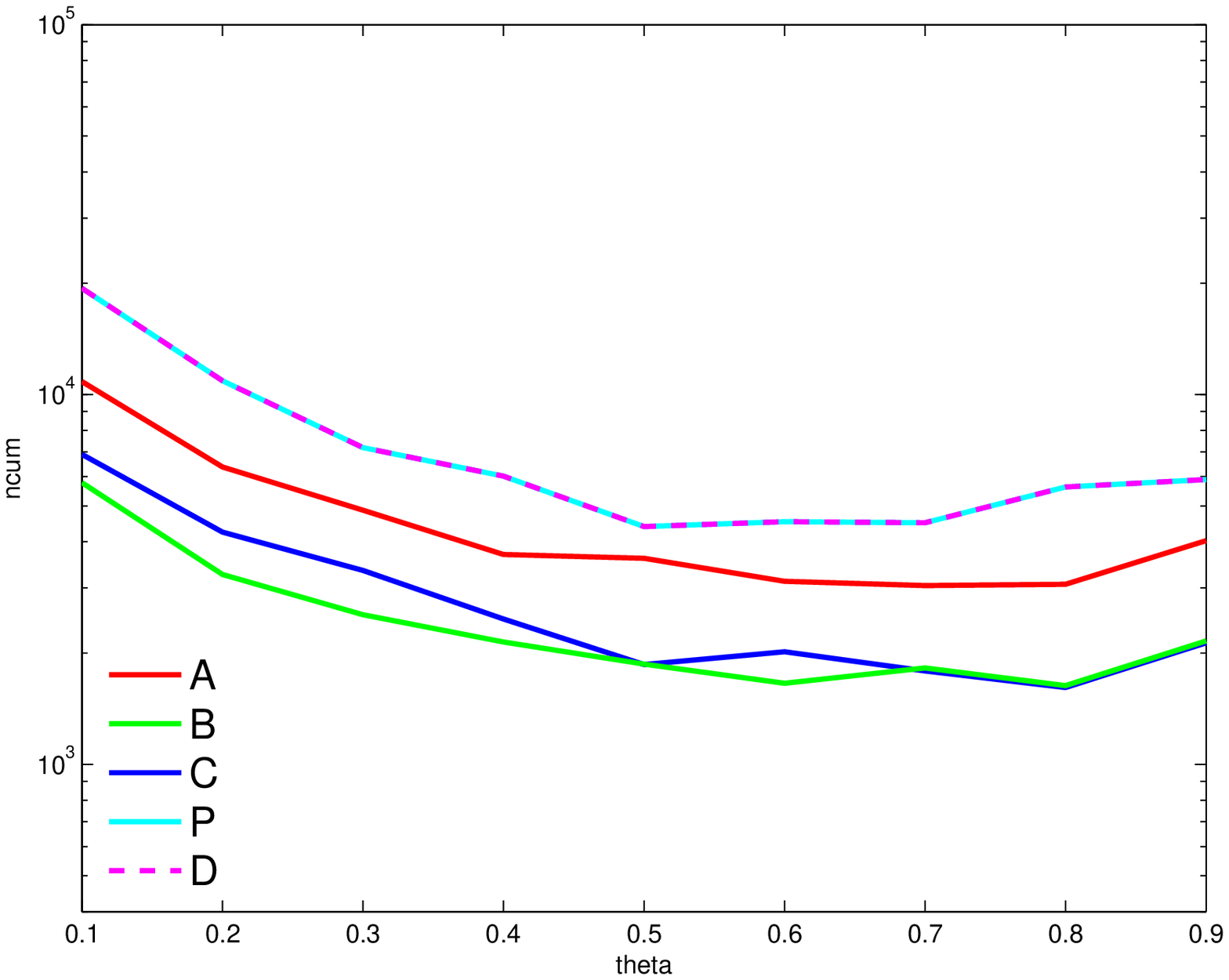}
\qquad
\includegraphics[scale=0.4]{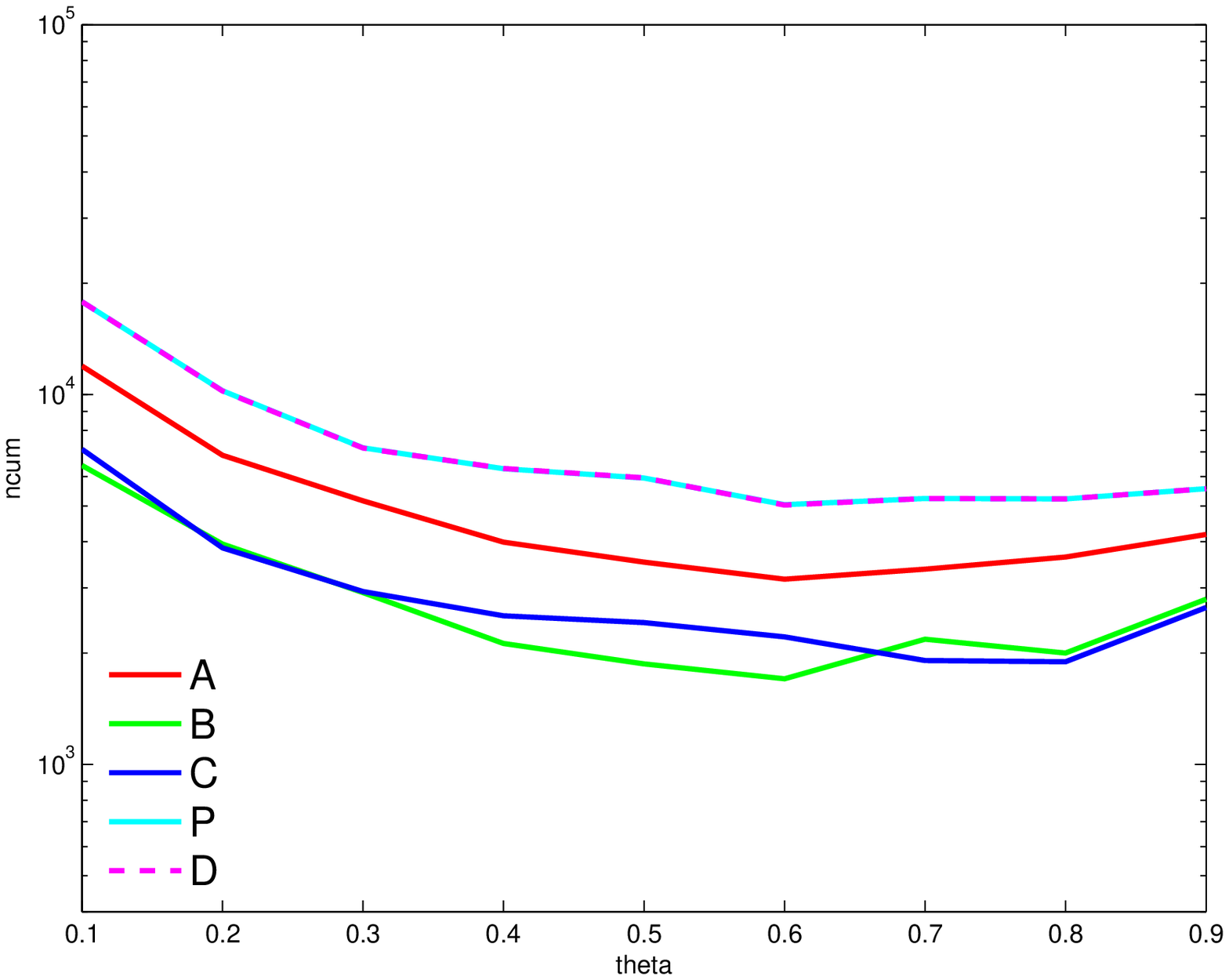}
\caption{Example from Section~\ref{example1:afem}: For Algorithm~\ref{algorithm}, \ref{algorithm:mod}, and~\ref{algorithm:bet} as well as standard (non-goal-oriented) AFEM driven by the primal error estimator resp.\ the dual error estimator, we plot the cumulative number of elements~$N_{\mathrm{cum}} := \sum_{j=0}^\ell \# \TT_j$ necessary to reach a prescribed accuracy~$\eta_{u,\ell} \eta_{z,\ell} \le \mathsf{tol}$ over $\theta \in \{0.1,\dots,0.9\}$ for $p=3$ and $\mathsf{tol} = 10^{-5}$ (left) resp.\ $p=2$ and $\mathsf{tol} = 10^{-4}$ (right).}
\label{fig:fem1:comparison}
\end{figure}

\subsection{Numerical experiment I: Goal oriented FEM for the Poisson equation}
\label{example1:afem}
We consider a numerical example proposed in~\cite[Example 7.3]{ms} for the Laplace operator in 2D, while a nonsymmetric second-order elliptic operator is considered in Section~\ref{section:example:afem2}. The goal of this first experiment is to verify the optimal convergence of Algorithm~\ref{algorithm}--\ref{algorithm:bet} as predicted by theory, and to compare the various algorithms as well as standard AFEM (i.e., non-goal-oriented adaptive FEM, where $\MM_\ell := \MM_{u,\ell}$ resp.\ $\MM_\ell :=\MM_{z,\ell}$ in Algorithm~\ref{algorithm}); see, e.g.,~\cite{axioms,ckns,ffp,stevenson}).

We consider the Poisson model problem (i.e., $\matrix{A} = \matrix{I}$, $\matrix{b} = \matrix{0}$, and $c=0$) on the unit cube $\Omega =(0,1)^2 \subset \mathbb{R}^2$. Unlike~\cite{ms} which considers quadratic elements $p=2$, we apply cubic elements $p=3$ (unless stated otherwise).
The initial mesh $\TT_0$ is shown in Figure~\ref{fig:MStestcase} (left), where also the triangles 
$T_f:={\rm conv}\{(0,0),(\frac{1}{2},0),(0,\frac{1}{2}) \}$ and $T_g:={\rm conv}\{(1,1),(\frac{1}{2},1),(1,\frac{1}{2}) \}$ are visualized.
The right-hand sides of the primal~\eqref{eq:primal} and dual problem~\eqref{eq:dual} are 
\begin{alignat*}{2}
 f(v) = - \int_{T_f} \frac{\partial v}{\partial x_1} \, dx
 \quad\text{resp.}\quad
 g(u) &= - \int_{T_g} \frac{\partial v}{\partial x_1} \, dx .
\end{alignat*}
This corresponds to $f_1 = 0$, $\matrix{f}_2 = (\chi_{T_f} , 0)$, $g_1 = 0$, $\matrix{g}_2 = (\chi_{T_g} , 0)$, where $\chi_{\omega}$ for $\omega\subset\R^2$ denotes the characteristic function, i.e., $\chi_\omega(x)=1$ for $x\in\omega$ and $\chi_\omega(x)=0$ for $x\in\R^2\backslash\omega$. Figure~\ref{fig:MStestcase} also shows some approximations of the primal and dual solution.
The primal solution $u$ has a line singularity along ${\rm conv}\{(\frac{1}{2},0),(0,\frac{1}{2}) \}$, while the dual solution $z$ has a line singularity along ${\rm conv}\{(\frac{1}{2},1),(1,\frac{1}{2}) \}$. At the intersection of the lines with $\partial\Omega$, there are point singularities.

Figure~\ref{fig:fem1:MSconv} (left) shows the typical convergence behavior for the estimators~$\eta_{u,\ell}$ and $\eta_{z,\ell}$, the estimator product~$\eta_{u,\ell} \eta_{z,\ell}$, and the goal error~$|g(u) - g(U_{\ell})|$, where we used Algorithm~\ref{algorithm}--\ref{algorithm:bet} with $\theta = 0.5$. Similar results are obtained for other choices of $0<\theta<1$ (not displayed). The estimator product $\eta_{u,\ell} \eta_{z,\ell}$ shows the optimal convergence rate of $\OO(N^{-3})$ as predicted by theory for~$p=3$ in 2D.

Figure~\ref{fig:fem1:MSconv} (right) 
shows that all Algorithms~\ref{algorithm}--\ref{algorithm:bet} yield the optimal rate of convergence $\OO(N^{-3})$, for a large range of values of $\theta$ including~$\theta = 0.9$. Uniform refinement corresponds to~$\theta = 1.0$ and shows a suboptimal rate of~$\OO(N^{-1})$

Figure~\ref{fig:fem1:QconvPD} shows the numerical results for standard AFEM, which are based on adaptivity for either the primal or the dual problem. In both cases, theory predicts optimal convergence behavior $\OO(N^{-3/2})$ for the related error estimator, at least if the adaptivity parameter $\theta$ is sufficiently small; see, e.g.,~\cite{ckns,ffp,axioms}. For all $\theta\in\{0.1,\dots,0.9\}$, we observe the optimal rate $\OO(N^{-3/2})$ for the error estimator which drives the adaptive process. 
However, for the estimator product $\eta_{u,\star} \eta_{z,\star}$ these strategies result in a suboptimal convergence rate $\OO(N^{-2})$.

In adaptive computations, the overall runtime depends on the entire history of adaptively generated meshes. To better compare the various algorithms, Figure~\ref{fig:fem1:comparison} shows the cumulative number of elements~
\begin{align}\label{eq:Ncum}
 N_{\mathrm{cum}} := \sum_{j=0}^\ell \# \TT_j,
\end{align}
which is necessary to reach a prescribed accuracy of~$\eta_{u,\ell} \eta_{z,\ell} \le \mathsf{tol}$, versus~$\theta\in\{0.1,0.9\}$. The definition of $N_{\mathrm{cum}}$ reflects the total amount of work in the complete adaptive process. Altogether, we compare five adaptive strategies: Besides Algorithm~\ref{algorithm}--\ref{algorithm:bet}, we consider standard AFEM based on the primal error estimator and standard AFEM based on the dual error estimator. For example, for a tolerance $\mathsf{tol} = 10^{-5}$ and~$p=3$, Figure~\ref{fig:fem1:comparison} (left) shows that $N_{\mathrm{cum}}$ is smallest for Algorithm~\ref{algorithm:mod}--\ref{algorithm:bet} for $\theta = 0.8$. Furthermore, we see that the goal-oriented algorithms~\ref{algorithm}--\ref{algorithm:bet} are superior to standard AFEM. Amongst the goal-oriented algorithms, because of having \emph{combined} primal and dual refinement, Algorithm~\ref{algorithm:mod}--\ref{algorithm:bet} are superior to Algorithm~\ref{algorithm}, which only does one-sided refinement per iteration step. Furthermore, Algorithm~\ref{algorithm:mod} is at least competitive and sometimes even superior to Algorithm~\ref{algorithm:bet}. As visible in Figure~\ref{fig:fem1:comparison} (right), for $\mathsf{tol} = 10^{-4}$ and~$p=2$, $N_{\mathrm{cum}}$ is smallest for Algorithm~\ref{algorithm:mod} and~$\theta = 0.6$.

\section{Goal-Oriented Adaptive FEM for Flux Evaluation}\label{section:flux}
\subsection{Model problem}\label{flux:model}
On a bounded Lipschitz domain $\Omega\subseteq \R^d$ with boundary $\Gamma:=\partial\Omega$ and for given $\Lambda$, $f_1\in L^2(\Omega)$,
and $\boldsymbol{f}_2\in L^2(\Omega)^d$, we aim to compute
the weighted boundary flux
\begin{subequations}\label{dpr:flux}
\begin{align}
g(u):=\int_\Gamma (\matrix{A}\nabla u)\cdot n\,\Lambda\,ds,
\end{align}
where $u$ is the solution to~\eqref{ex:nonsymm}.
For smooth $u$, $g(u)$ can be rewritten as
\begin{align}
g(u)= \int_\Omega {\rm div}(\matrix{A}\nabla u) z\,dx+\int_\Omega\matrix{A}\nabla u\cdot \nabla z= \bform{u}{z}-f(z)=:N_z(u)
\end{align}
\end{subequations}
for all $z\in H^1(\Omega)$ with $z|_\Gamma=\Lambda$.
Since the right-hand side is well-defined for $u\in H^1_0(\Omega)$, this is a valid generalization of the flux~\cite[Section~7]{MR2009374}.
Let $z$ be the unique solution of the following inhomogeneous Dirichlet problem:
\begin{align*}
 z\in H^1(\Omega)\text{ with }z|_\Gamma=\Lambda
 \quad\text{such that}\quad\bform{v}{z}=0\quad\text{for all }v\in H^1_0(\Omega).
\end{align*}
Then, it holds
\begin{align*}
 N_z(u)=-f(z).
\end{align*}

\subsection{Discretization}\label{flux:discretization}
For a given regular triangulation $\TT_\star$ of $\Omega$ and a polynomial degree $p\ge1$, 
let $\PP^p(\TT_\star)$ be defined as in Section~\ref{afem:discretization}. Consider
$\SS^p(\TT_\star):=\PP^p(\TT_\star)\cap H^1(\Omega)$ and $\SS^p_0(\TT_\star):=\PP^p(\TT_\star)\cap H^1_0(\Omega)$. Let $U_\star$ be the unique FEM solution of the homogeneous Dirichlet problem
\begin{subequations}\label{eq:flux:discrete}
\begin{align}\label{eq:flux:discrete:primal}
 U_\star\in\SS^p_0(\TT_\star)
 \quad\text{such that}\quad
 a(U_\star,V_\star) = f(V_\star)
 \quad\text{for all }V_\star\in\SS^p_0(\TT_\star).
\end{align}
Suppose that $\Lambda\in \SS^p(\TT_0|_\Gamma):=\set{V_0|_\Gamma}{V_0\in\SS^p(\TT_0)}$ belongs to the discrete trace space on the initial triangulation $\TT_0$. To approximate $N_z(u)$ from~\eqref{dpr:flux}, we let $Z_\star$ be the unique FEM solution of
\begin{align}
 Z_\star\in\SS^p(\TT_\star)\text{ with }Z_\star|_\Gamma=\Lambda
 \quad\text{such that}\quad
 \bform{V_\star}{Z_\star}=0\quad\text{for all }V_\star\in\SS^p_0(\TT_\star)
\end{align}
\end{subequations}
and define
\begin{align}\label{dpr:flux:discrete}
N_{z,\star}(U_\star)=-f(Z_\star).
\end{align}

\begin{lemma}\label{lem:flux}
There holds
 \begin{align*}
   |N_z(u)-N_{z,\star}(U_\star)|\leq C_{\rm flux} \norm{u-U_\star}{H^1(\Omega)}\norm{z-Z_\star}{H^1(\Omega)},
 \end{align*}
 where $U_\star$ denotes the FEM approximation of $u$ from~\eqref{eq:primal:discrete} and $C_{\rm flux}>0$ depends only on $\bform{\cdot}{\cdot}$.
\end{lemma}

\begin{proof}
Since $z-Z_\star\in H^1_0(\Omega)$, there holds
\begin{align*}
|N_z(u)-N_{z,\star}(U_\star)|&=|f(z)-f(Z_\star)|=|f(z-Z_\star)|=|a(u,z-Z_\star)|\\
&=|a(u-U_\star,z-Z_\star)|\lesssim \norm{u-U_\star}{H^1(\Omega)}\norm{z-Z_\star}{H^1(\Omega)},
\end{align*}
where we used the definition of $z$ and $Z_\star$.
\end{proof}

\subsection{Residual error estimator}\label{flux:estimator}
The residual error estimator for the primal problem remains the same as in~\eqref{ex:nonsymm:estimator:primal}, i.e.,
\begin{align}\label{flux:nonsymm:estimator:primal}
\eta_{u,\star}(T)^2:=h_T^2\norm{\operator{L}|_T U_\star - f_1-{\rm div}\,\boldsymbol{f}_2}{L^2(T)}^2 
+ h_T\norm{[(\matrix{A}\nabla U_\star+\boldsymbol{f}_2) \cdot n]}{L^2(\partial T\cap\Omega)}^2.
\end{align}
Since the inhomogeneous boundary data satisfies $\Lambda\in\SS^p(\TT_0|_\Gamma)$ also
the dual estimator $\eta_{z,\star}$ remains the same as in~\eqref{ex:nonsymm:estimator:dual} with $g_1=0$ and $\boldsymbol{g}_2=0$, i.e.,
\begin{align}\label{flux:nonsymm:estimator:dual}
 \eta_{z,\star}(T)^2:=h_T^2\norm{\operator{L}^T|_T Z_\star}{L^2(T)}^2 
 + h_T\norm{[\matrix{A}\nabla Z_\star\cdot n]}{L^2(\partial T\cap\Omega)}^2.
\end{align}
Lemma~\ref{lem:flux} together with the reliability of $\eta_{w,\star}$ for $w\in\{u,z\}$ (see, e.g.,~\cite[Proposition~3]{dirichlet3d} for the inhomogeneous Dirichlet problem for $z$) implies
\begin{align}\label{eq:flux:estimator}
 |N_z(u)-N_{z,\star}(U_\star)|\lesssim \eta_{u,\star}\eta_{z,\star}.
\end{align}
Hence, the problem fits into the abstract framework of Section~\ref{section:abstract}.
We aim for optimal convergence of the right-hand side of~\eqref{eq:flux:estimator}.

\subsection{Verification of axioms}
With newest vertex bisection from~\cite{stevenson:nvb} as mesh-refinement strategy, 
the assumptions of Section~\ref{section:finemesh} are satisfied. It remains to verify 
the axioms~\eqref{ass:stable}--\eqref{ass:reliable},
where $\dist{w}{\TT_\ell}{\TT_\star}:=a(W_\ell-W_\star,W_\ell-W_\star)^{1/2}\simeq \norm{W_\ell-W_\star}{H^1(\Omega)}$.

\begin{theorem}\label{theorem:flux}
Consider the model problem of Section~\ref{flux:model}. Then, the conforming discretization~\eqref{eq:flux:discrete} of Section~\ref{flux:discretization} with the residual error estimators~\eqref{flux:nonsymm:estimator:primal}--\eqref{flux:nonsymm:estimator:dual} from Section~\ref{flux:estimator} satisfies stability~\eqref{ass:stable}, reduction~\eqref{ass:reduction} with $\q{reduction}=2^{-1/d}$,
quasi-orthogonality~\eqref{ass:orthogonal}, and discrete reliability~\eqref{ass:reliable} with $\RR_u(\TT_\ell,\TT_\star) =\RR_z(\TT_\ell,\TT_\star)= \TT_\ell\backslash\TT_\star$.
In particular, the Algorithms~\ref{algorithm}--\ref{algorithm:bet} are linearly convergent with optimal rates in the sense of Theorem~\ref{theorem:linear}, \ref{theorem:optimal}, \ref{theorem:optimal:mod}, and~\ref{theorem:optimal:bet} for the upper bound in~\eqref{eq:flux:estimator}.
\end{theorem}

\begin{proof}
 For the primal problem,~\eqref{ass:stable}--\eqref{ass:reliable} follow as in Theorem~\ref{thm:ex:nonsymm}. For the dual problem, the axioms~\eqref{ass:stable}--\eqref{ass:reduction} follow from Theorem~\ref{thm:ex:nonsymm}
 since the estimator did not change. The discrete reliability~\eqref{ass:reliable} is proved in~\cite{dirichlet3d} for general $W\in H^1(\Gamma)$. In our particular situation, the proof simplifies vastly and shows even $\RR_z(\TT_\ell,\TT_\star)= \TT_\ell\backslash\TT_\star$. To see the quasi-orthogonality~\eqref{ass:orthogonal}, choose a discrete extension $\widehat W\in\SS^1(\TT_0)$ with $\widehat W|_\Gamma=W$.
Consider the solution $Z_\star^0\in\SS^p_0(\TT_\star)$ of
 \begin{align*}
  \bform{V_\star}{Z_\star^0}=-\bform{V_\star}{\widehat W}\quad\text{for all }V_\star\in\SS^p_0(\TT_\star).
 \end{align*}
Then, there holds $Z_\star=Z_\star^0+\widehat W$. Consequently, there holds $\dist{z}{\TT_{\ell_{j+1}}}{\TT_{\ell_j}}\simeq \norm{Z_{\ell_{j+1}}-Z_{\ell_{j}}}{H^1(\Omega)}=\norm{Z_{\ell_{j+1}}^0-Z_{\ell_{j}}^0}{H^1(\Omega)}$.
Since $Z_\star^0$ is the solution to a homogeneous Dirichlet problem, the proof of~\eqref{ass:orthogonal} follows analogously to that of Section~\ref{section:example}.
\end{proof}

\begin{figure}[t]
{
\psfrag{0}{$\!0$}
\psfrag{ 0 }{$0$}
\psfrag{1/6}{\,$\tfrac{1}{6}$}
\psfrag{1/3}{\,$\tfrac{1}{3}$}
\psfrag{1/2}{\,$\tfrac{1}{2}$}
\psfrag{0.5}{\,\,$\tfrac{1}{2}$}
\psfrag{2/3}{\,$\tfrac{2}{3}$}
\psfrag{5/6}{\,$\tfrac{5}{6}$}
\psfrag{ 1 }{$1$}
\psfrag{1}{$\!1$}
\includegraphics[width=0.2\textwidth]{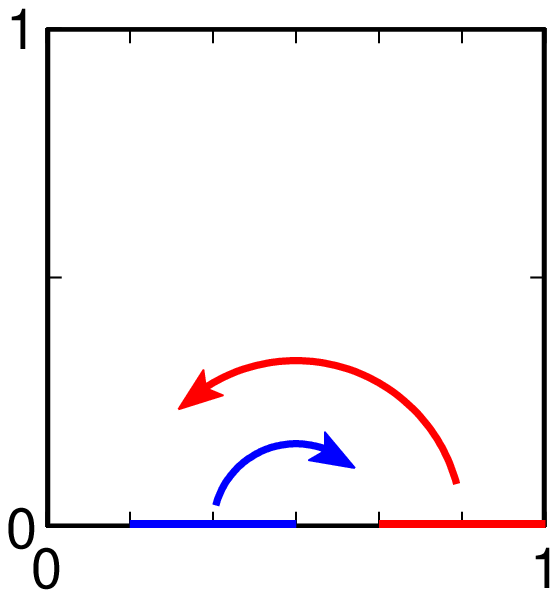}
}
\quad
{
\psfrag{0}{}
\psfrag{ 0 }{}
\psfrag{ 1 }{}
\psfrag{1}{}
\includegraphics[width=0.2\textwidth]{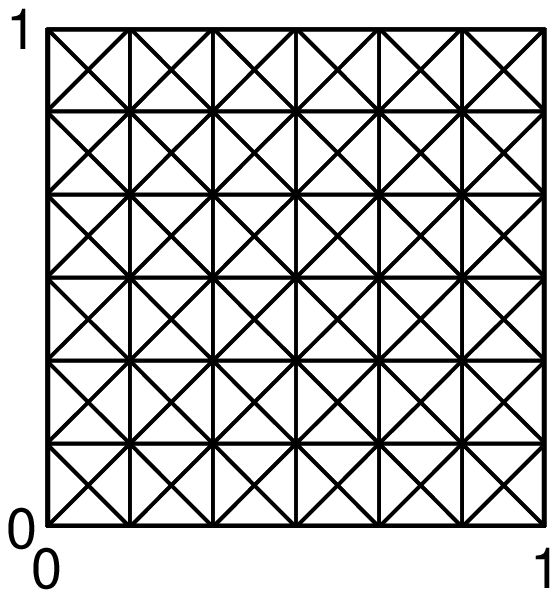}
}
\caption{Example from Section~\ref{section:example:afem2}: 
The left figure shows the geometry of the domain~$\Omega$, the support of the primal Dirichlet data (blue), the direction of the primal convective field (blue arrow), the support of the dual Dirichlet data (red), and the direction of the dual convective field (red arrow). The right figure shows the initial triangulation~$\mathcal{\TT}_0$ so that the inhomogeneous Dirichlet data belong to the discrete trace space $\SS^1(\TT_0|_\Gamma)$.}
\label{fig:fem2:MNtestcase}
\end{figure}

\begin{figure}[t]
\psfrag{0.1}{\scalebox{.5}{$0.1$}}
\psfrag{0.2}{\scalebox{.5}{$0.2$}}
\psfrag{0.3}{\scalebox{.5}{$0.3$}}
\psfrag{0.4}{\scalebox{.5}{$0.4$}}
\psfrag{0.5}{\scalebox{.5}{$0.5$}}
\psfrag{0.6}{\scalebox{.5}{$0.6$}}
\psfrag{0.7}{\scalebox{.5}{$0.7$}}
\psfrag{0.8}{\scalebox{.5}{$0.8$}}
\psfrag{0.9}{\scalebox{.5}{$0.9$}}
\psfrag{1.0}{\scalebox{.5}{$1.0$}}
\psfrag{estu}{\tiny$\eta_u$}
\psfrag{estz}{\tiny$\eta_z$}
\psfrag{estz*estu}{\tiny$\eta_u\eta_z$}
\psfrag{est}[c][c]{\tiny estimators}
\psfrag{error}[c][c]{\tiny error resp. estimators}
\psfrag{err}[c][c]{\tiny error}
\psfrag{Algorithm A}[c][c]{\tiny Algorithm A}
\psfrag{Algorithm B}[c][c]{\tiny Algorithm B}
\psfrag{Algorithm C}[c][c]{\tiny Algorithm C}
\psfrag{o2}{\tiny $\mathcal{O}(N^{-2})$}
\psfrag{o32}{\tiny $\mathcal{O}(N^{-3/2})$}
\psfrag{o12}{\tiny $\mathcal{O}(N^{-1/2})$}
\psfrag{o1}{\tiny $\mathcal{O}(N^{-1})$}
\psfrag{nE}[c][c]{\tiny number of elements $N=\#\TT_\ell$}
\includegraphics[scale=0.4]{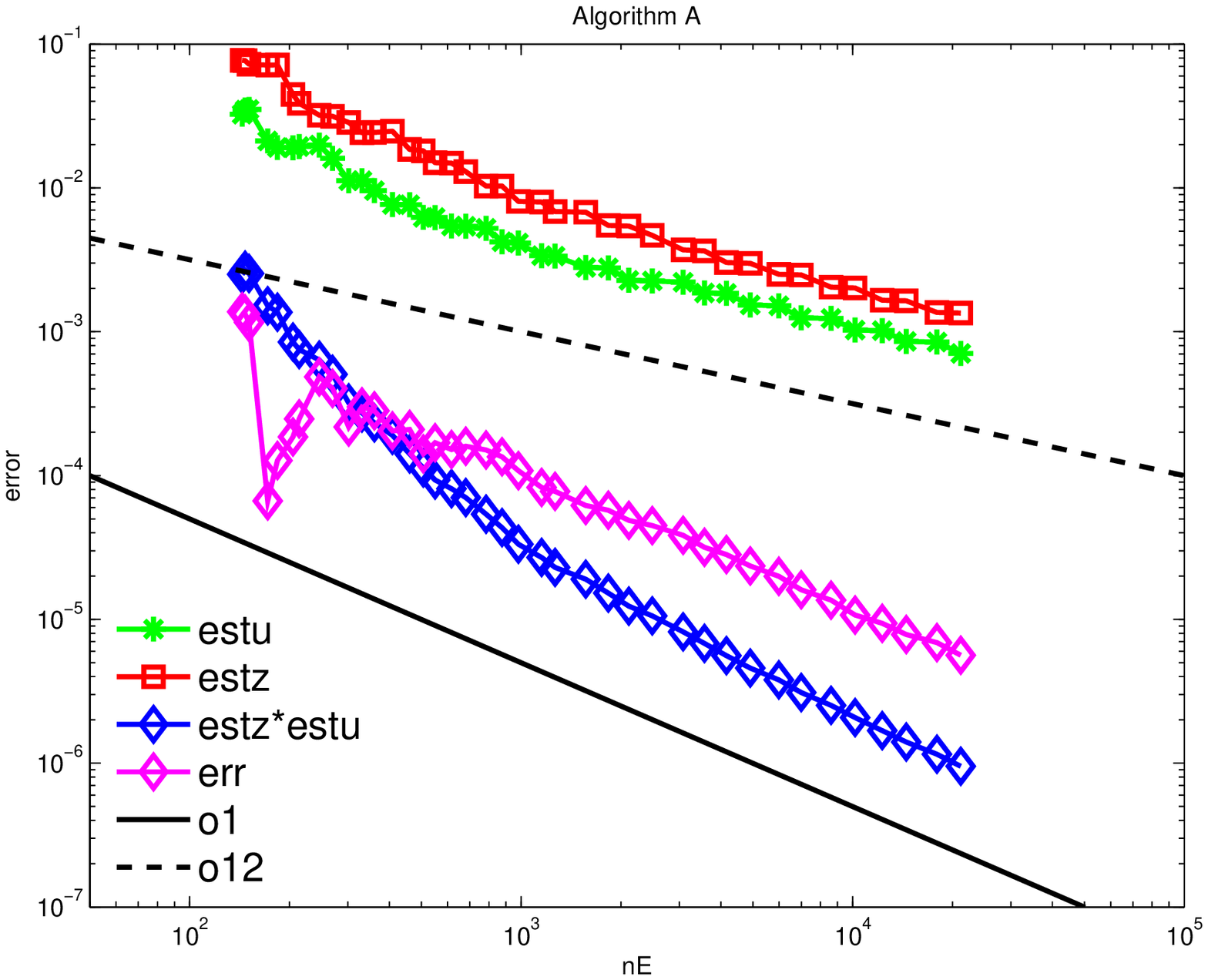}
\qquad
\includegraphics[scale=0.4]{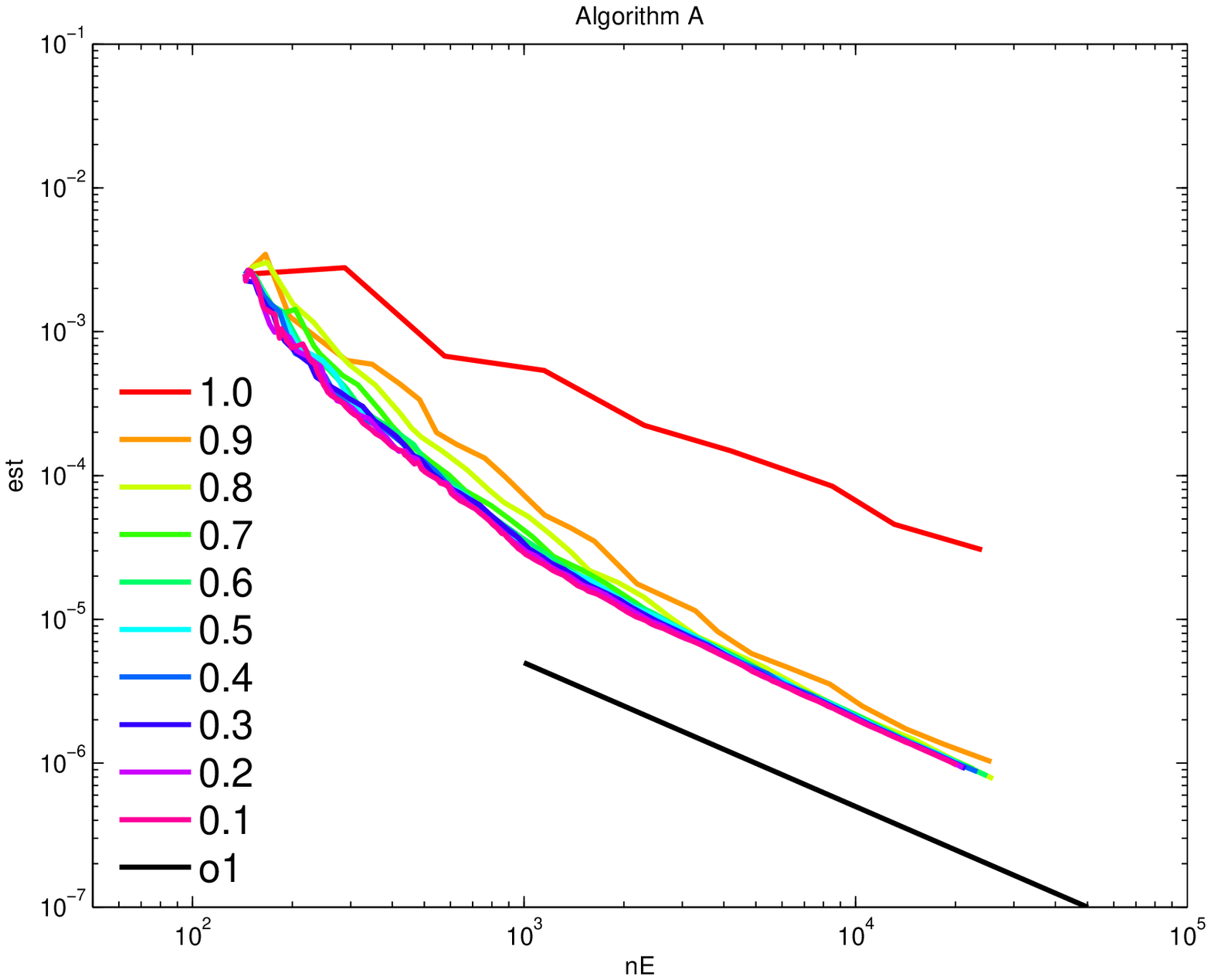}
\\[1ex]
\includegraphics[scale=0.4]{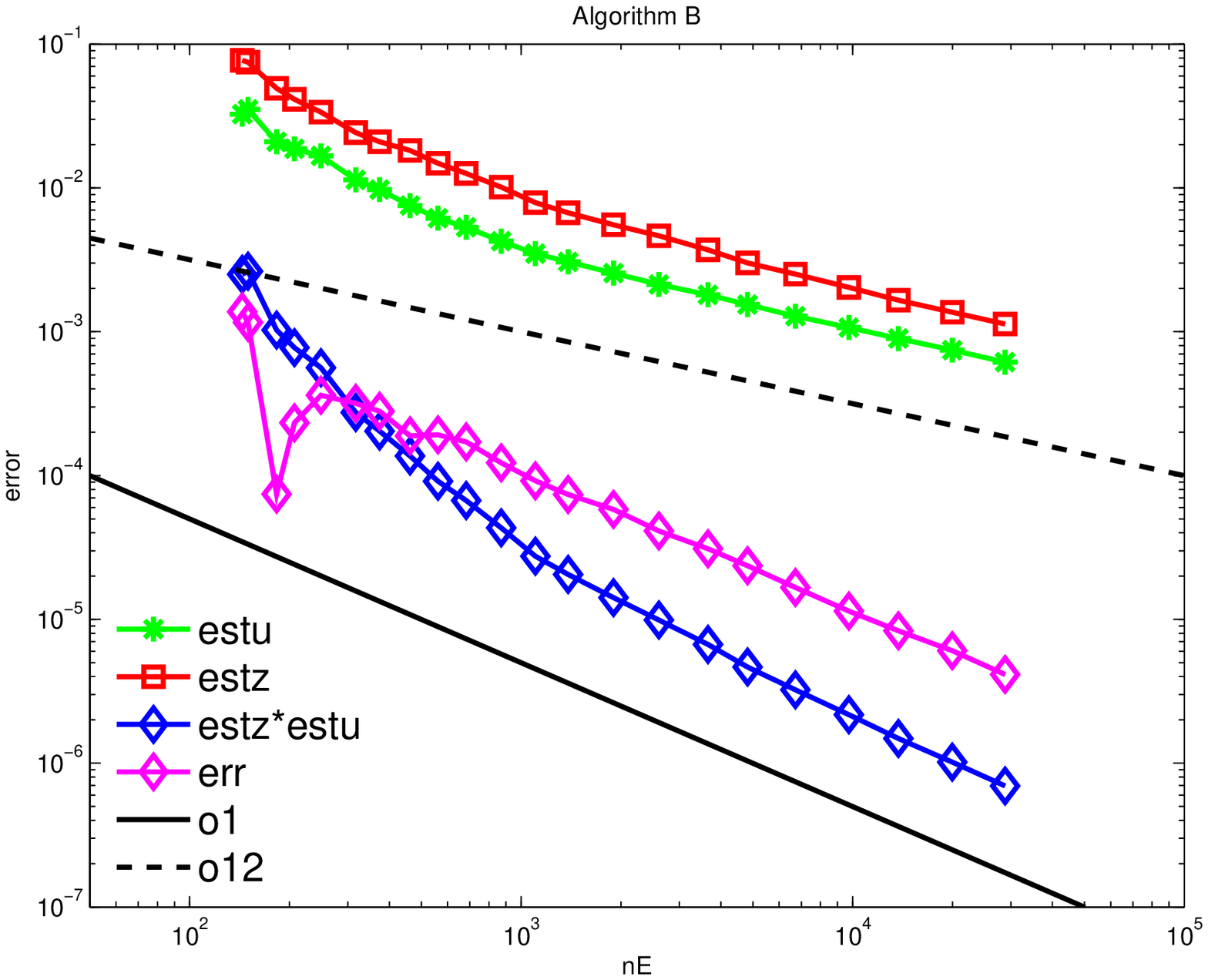}
\qquad
\includegraphics[scale=0.4]{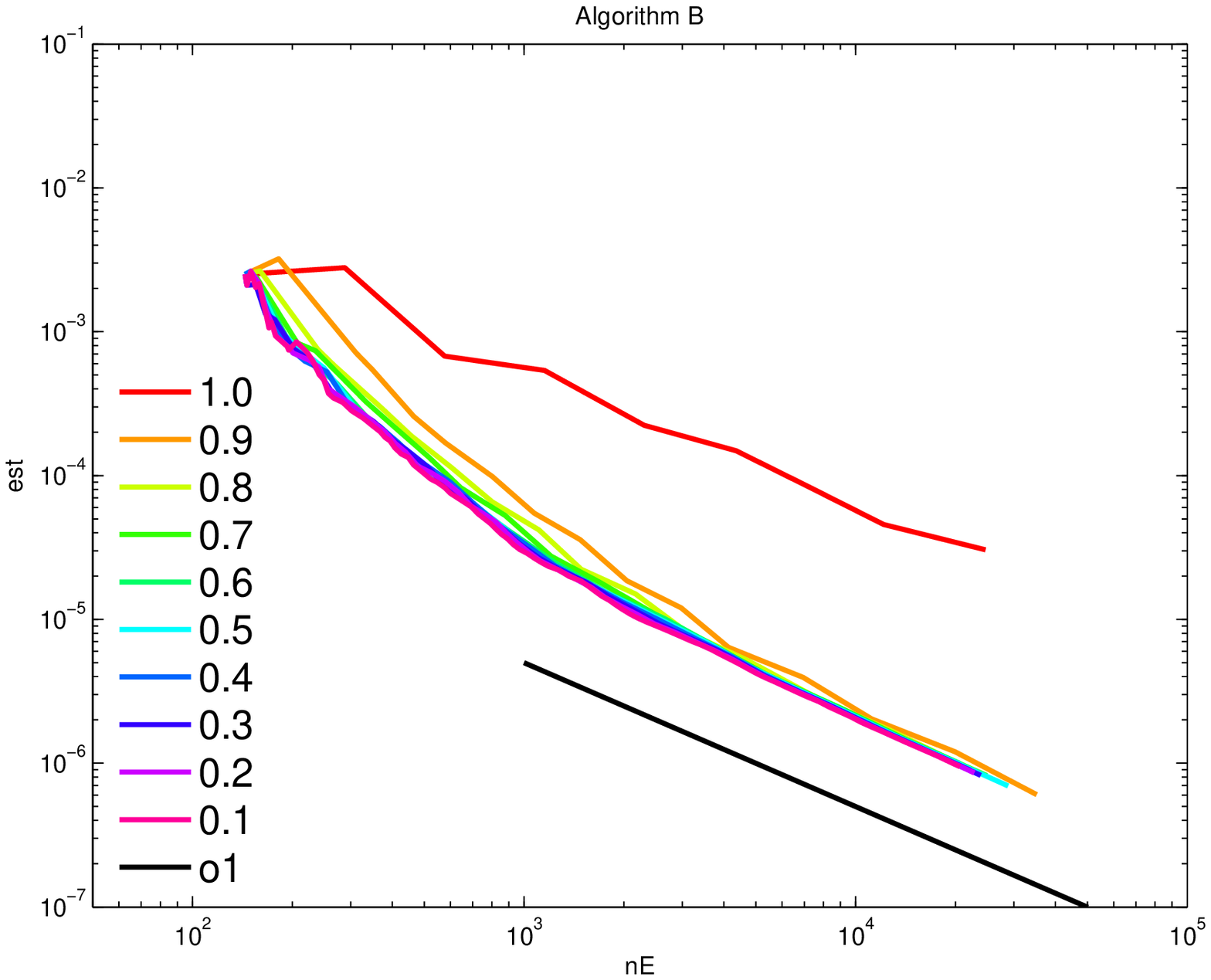}
\\[1ex]
\includegraphics[scale=0.4]{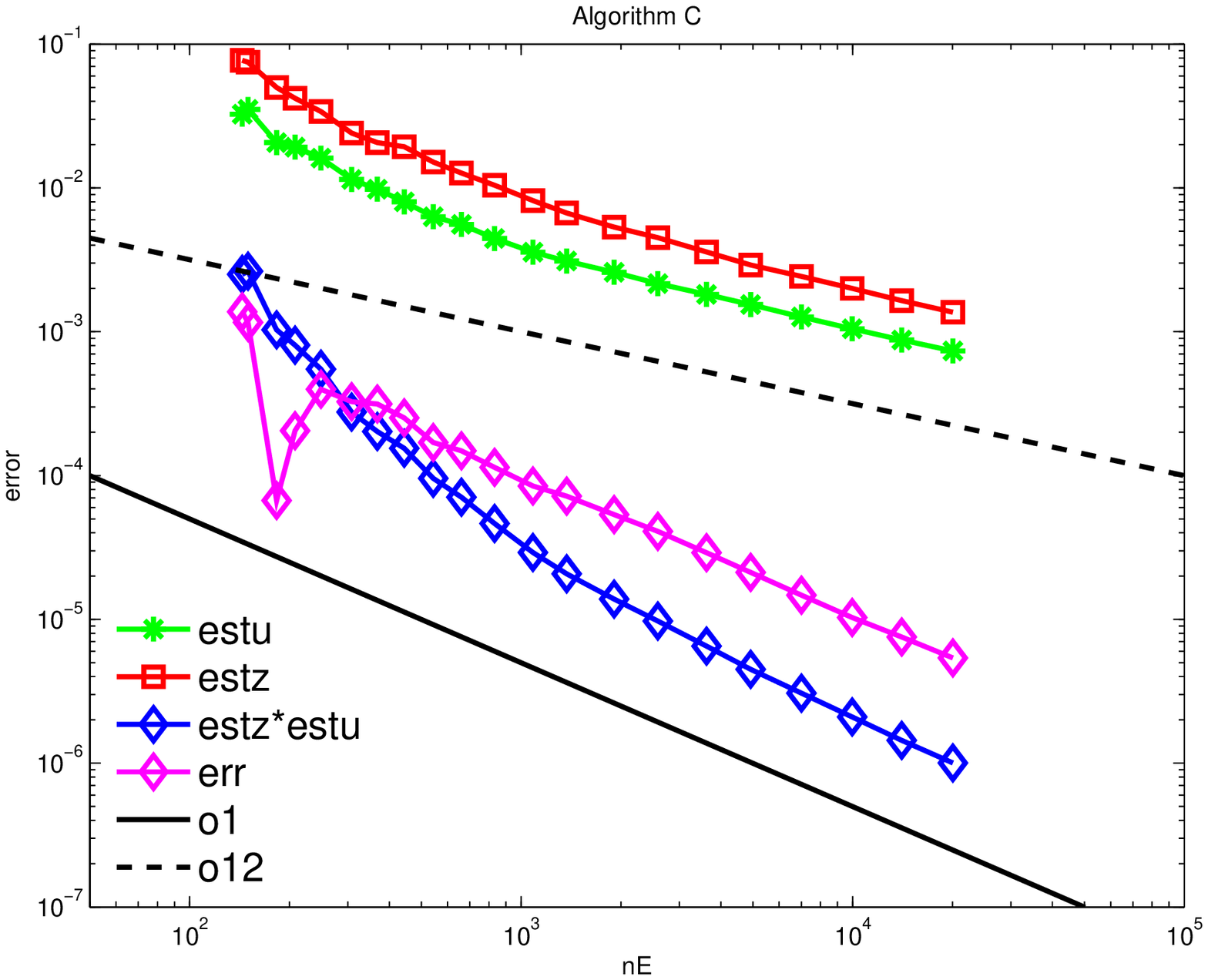}
\qquad
\includegraphics[scale=0.4]{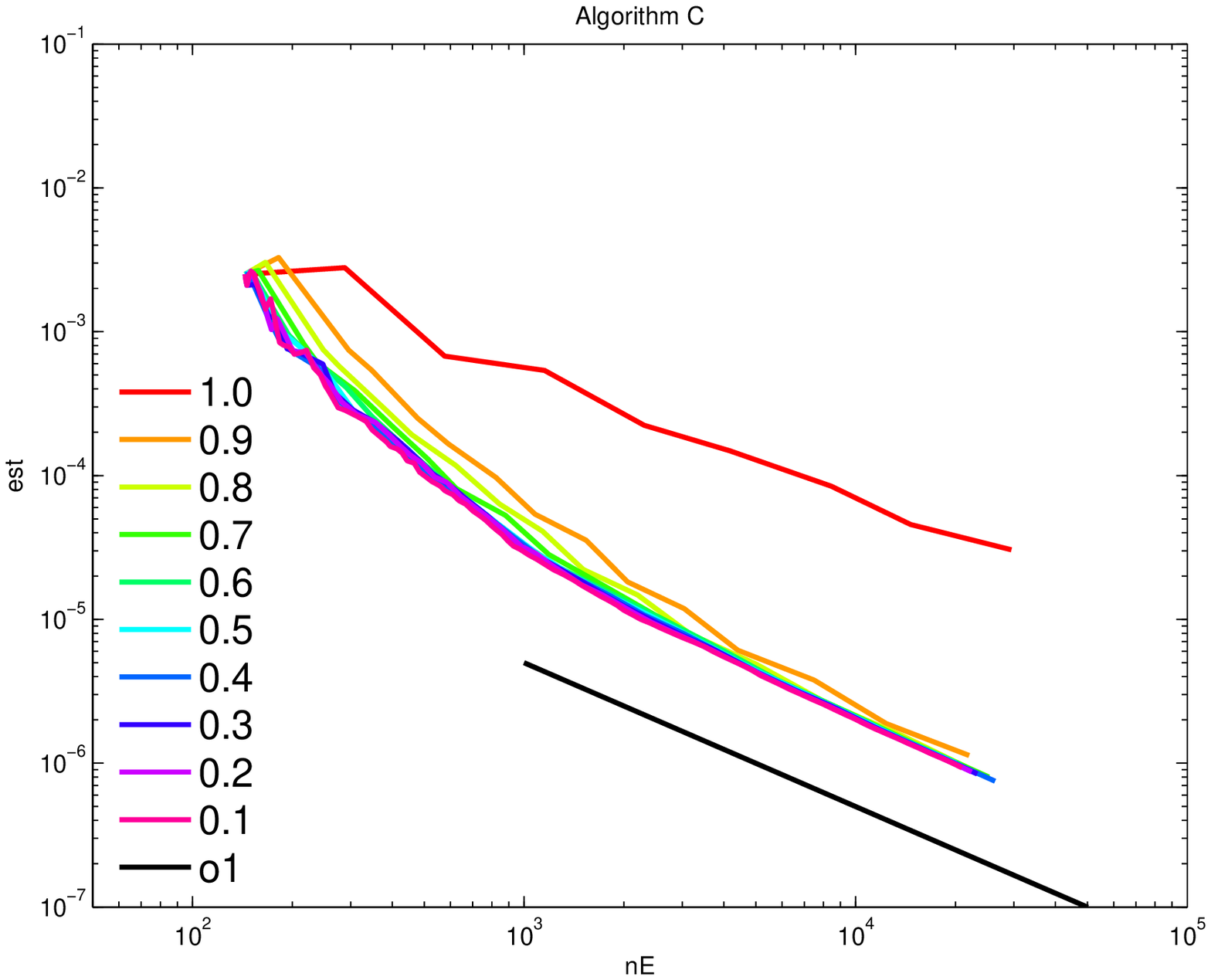}
\caption{Example from Section~\ref{section:example:afem2}: 
Over the numbers of elements $\#\TT_\ell$, we plot the estimators $\eta_{u,\ell}$ and $\eta_{z,\ell}$, the estimator product $\eta_{u,\ell}\eta_{z,\ell}$, as well as the goal error $|N_z(u) - N_{z,\ell}(U_\ell)|$ as output of  Algorithm~\ref{algorithm}--\ref{algorithm:bet} with $\theta=0.5$ (left) resp.\ the estimator product for various $\theta\in\{0.1,\dots,0.9\}$ as well as for $\theta=1.0$ which corresponds to uniform refinement. We consider $p=1$ and $\nu = 10^{-3}$ (right).}
\label{fig:fem2:QconvABC}
\end{figure}

\begin{figure}[t]
\psfrag{A}{\scalebox{.5}{Algorithm A}}
\psfrag{B}{\scalebox{.5}{Algorithm B}}
\psfrag{C}{\scalebox{.5}{Algorithm C}}
\psfrag{P}{\scalebox{.5}{adaptive algorithm for primal problem}}
\psfrag{D}{\scalebox{.5}{adaptive algorithm for dual problem}}
\psfrag{theta}[t]{\tiny{}parameter $\theta$}
\psfrag{ncum}{\tiny{}$N_{\mathrm{cum}}$}
\includegraphics[scale=0.4]{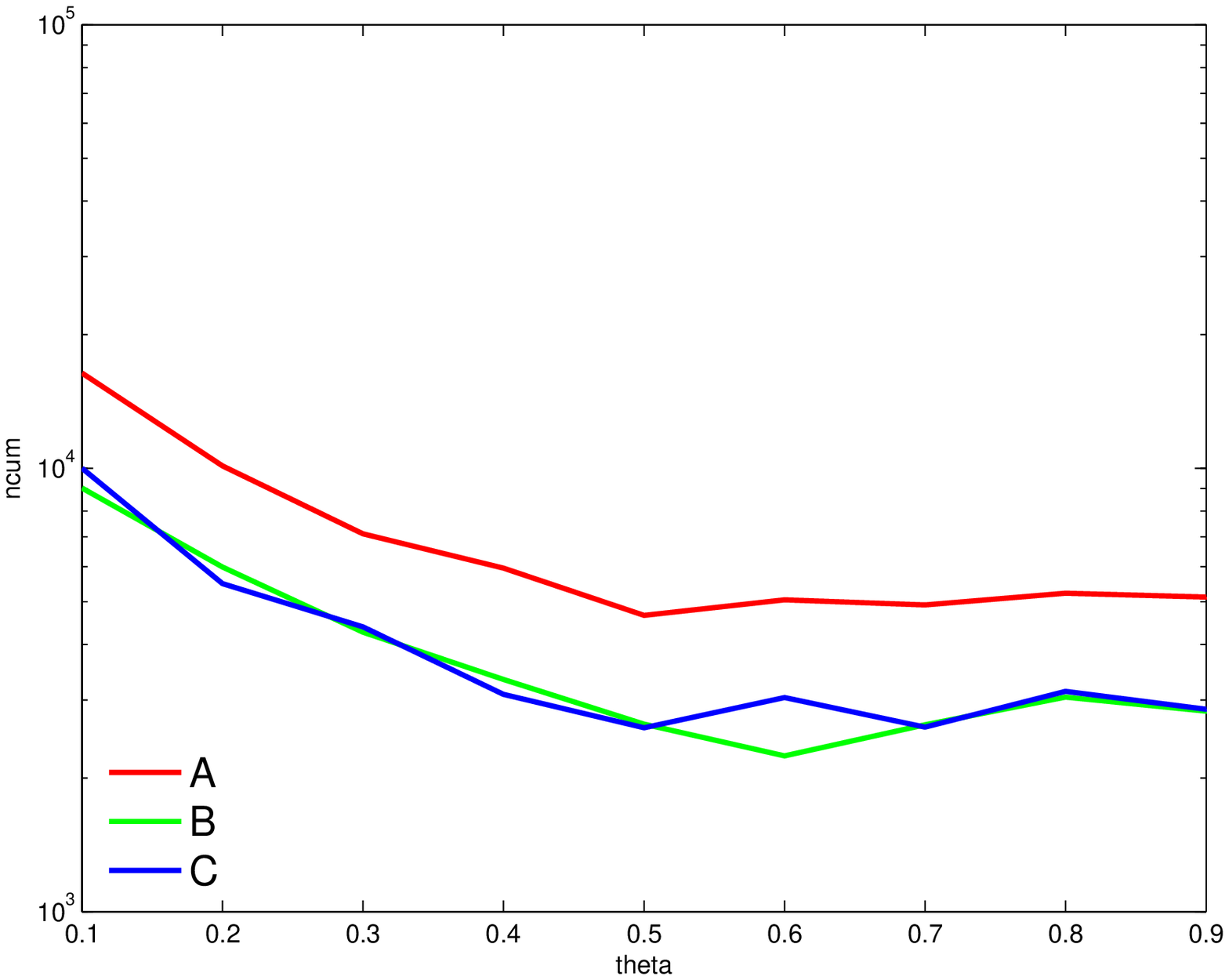}
\caption{Example from Section~\ref{section:example:afem2}: 
For Algorithm~\ref{algorithm}, \ref{algorithm:mod}, and~\ref{algorithm:bet} as well as standard (non-goal-oriented) AFEM driven by the primal error estimator resp.\ the dual error estimator, we plot the cumulative number of elements~$N_{\mathrm{cum}} := \sum_{j=0}^\ell \# \TT_j$ necessary to reach a prescribed accuracy~$\eta_{u,\ell} \eta_{z,\ell} \le 10^{-4}$ over $\theta \in \{0.1,\dots,0.9\}$ for $p=1$ and $\nu = 10^{-3}$.}
\label{fig:fem2:comparison} 
\end{figure}

\begin{figure}[t]
\newcommand{\clippedGraph}[1]{{\includegraphics[width=.33\textwidth,viewport=80 0 650 340,clip]{#1}}}
\newcommand{\clippedGraphh}[1]{{\includegraphics[width=.33\textwidth,viewport=105 0 653 265,clip]{#1}}}
{\tiny 
\begin{tabular}{c@{\quad}c@{\quad}c@{\quad}c@{\quad}c}
\clippedGraph{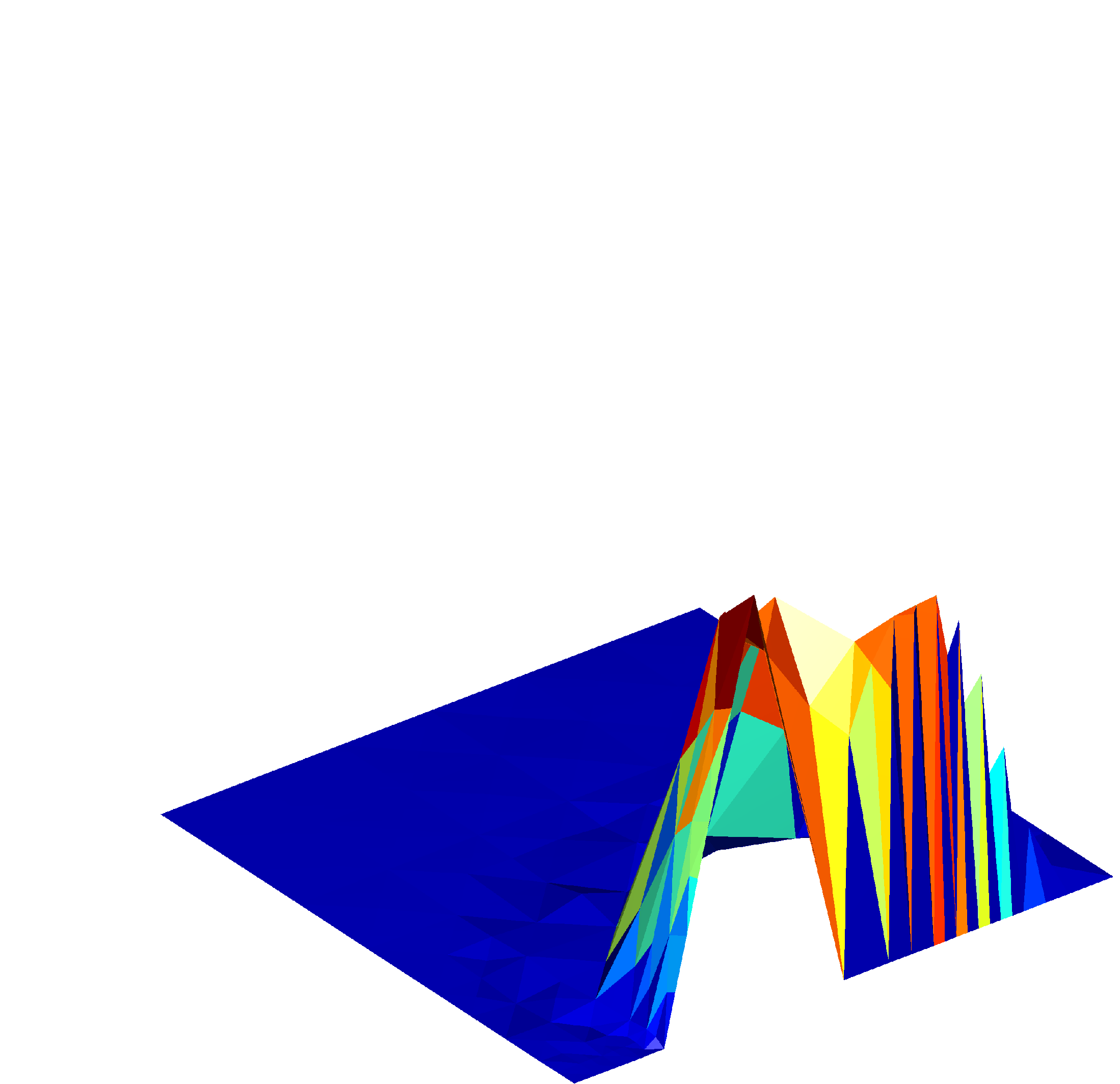}%
&
\clippedGraph{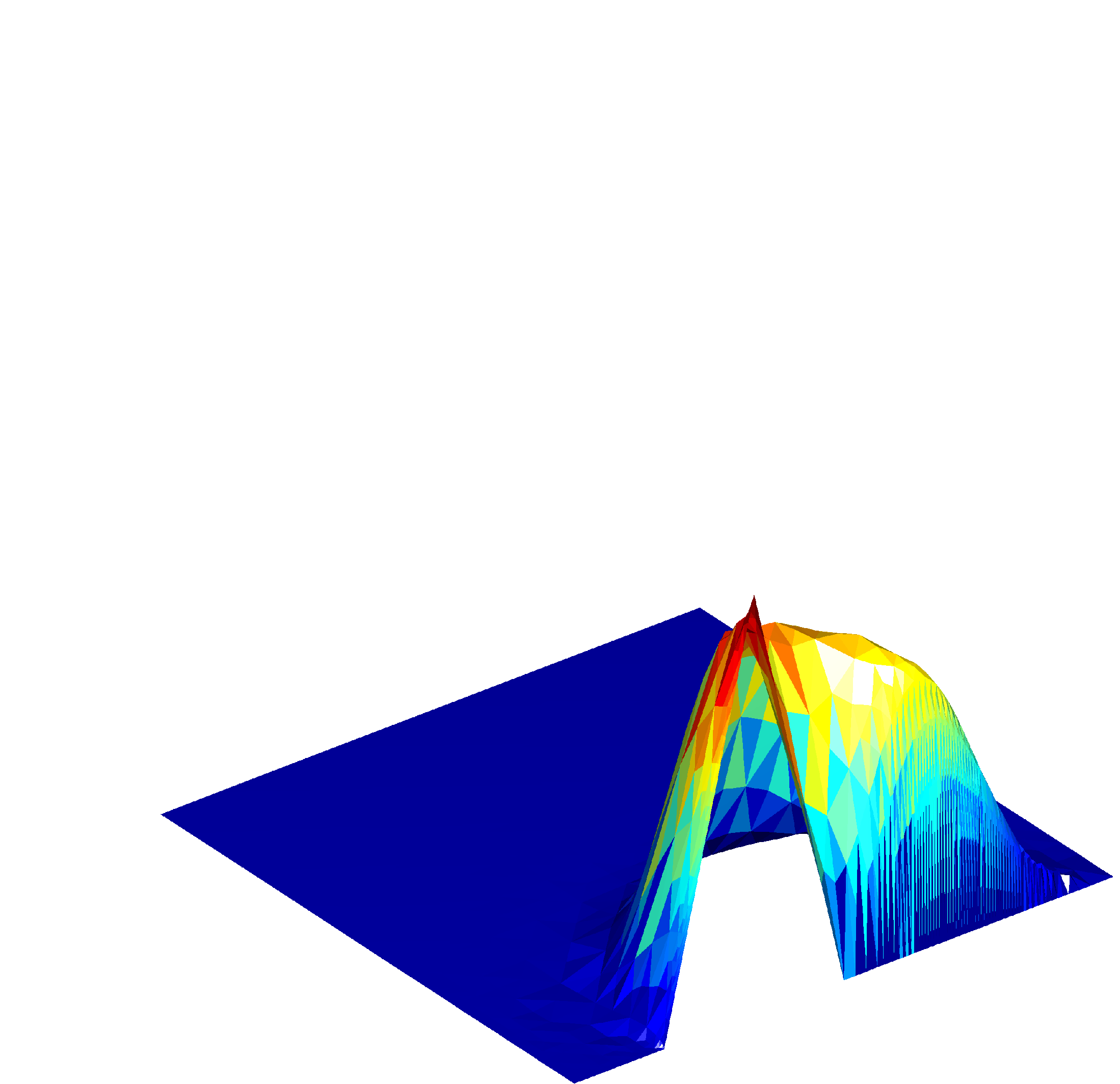}%
&
\clippedGraph{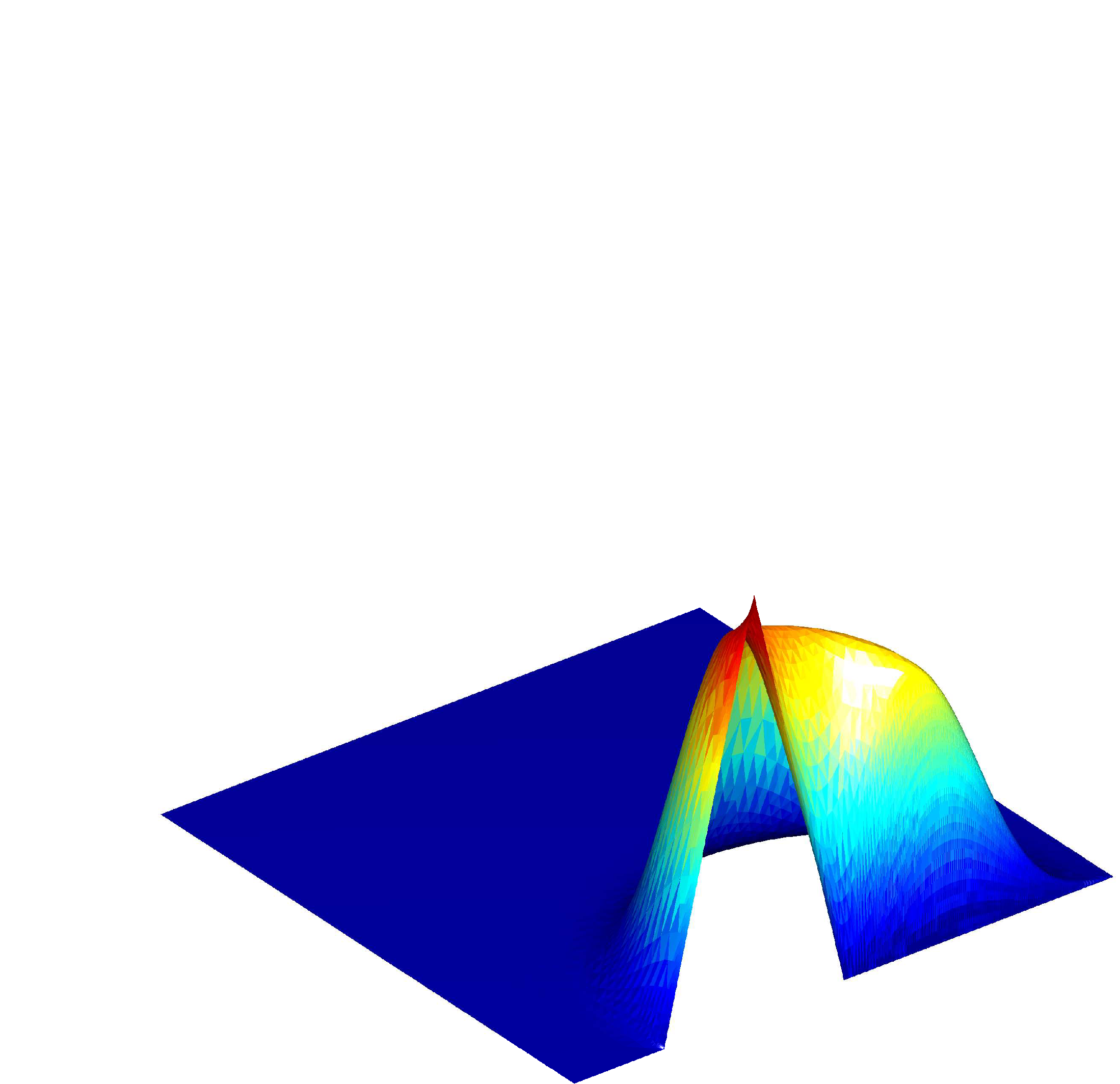}%
\\
\clippedGraph{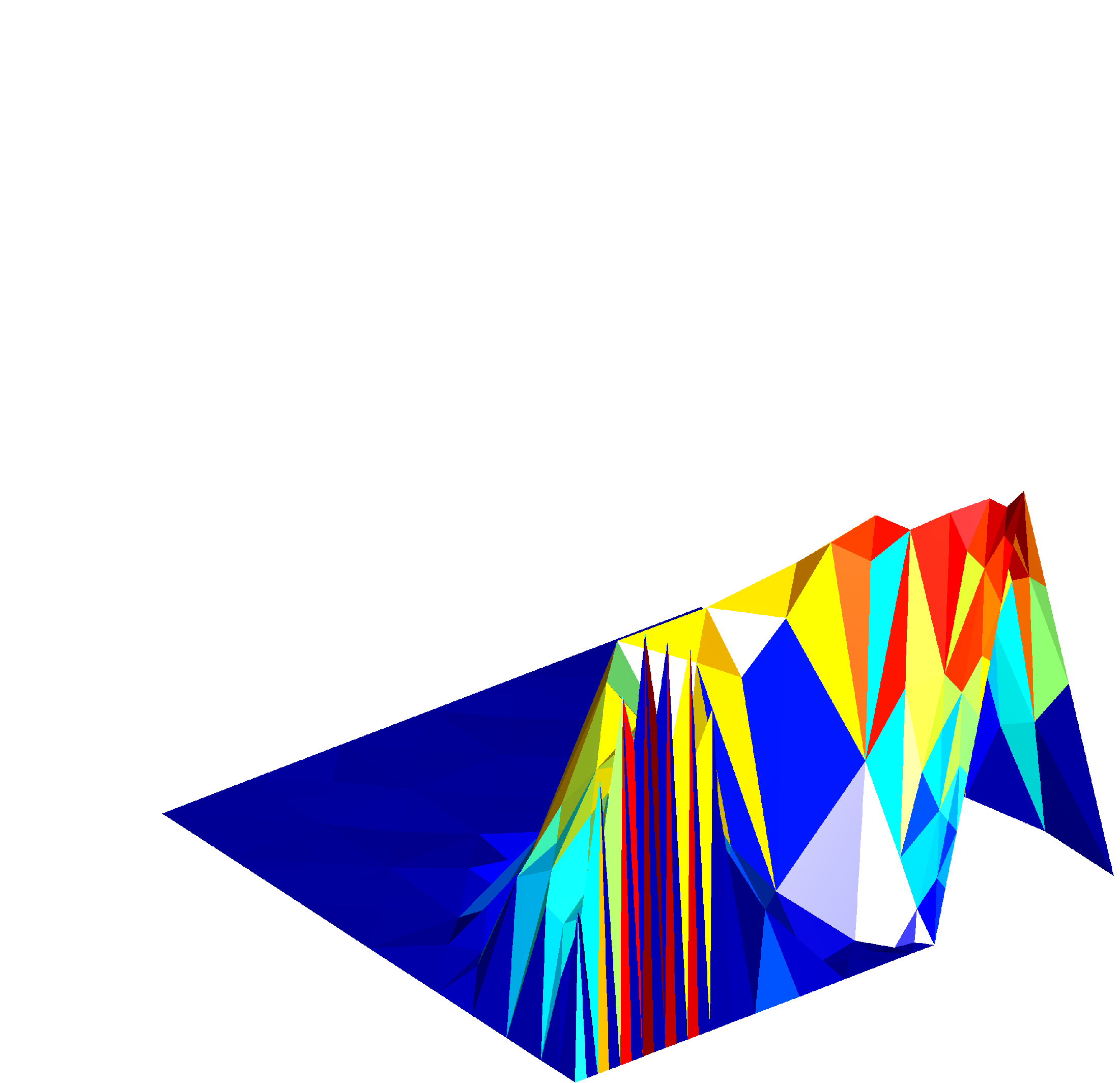}%
&
\clippedGraph{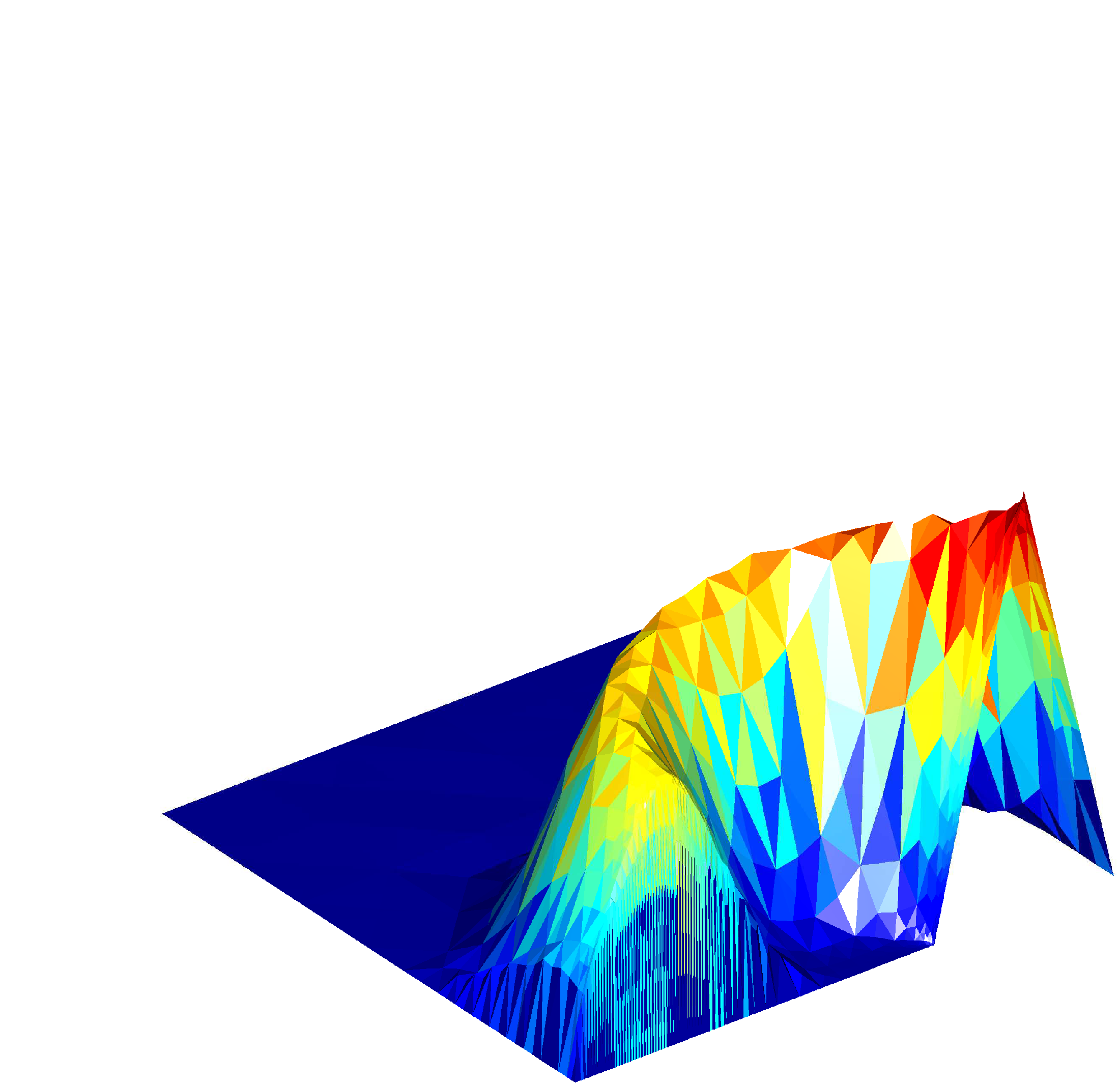}%
&
\clippedGraph{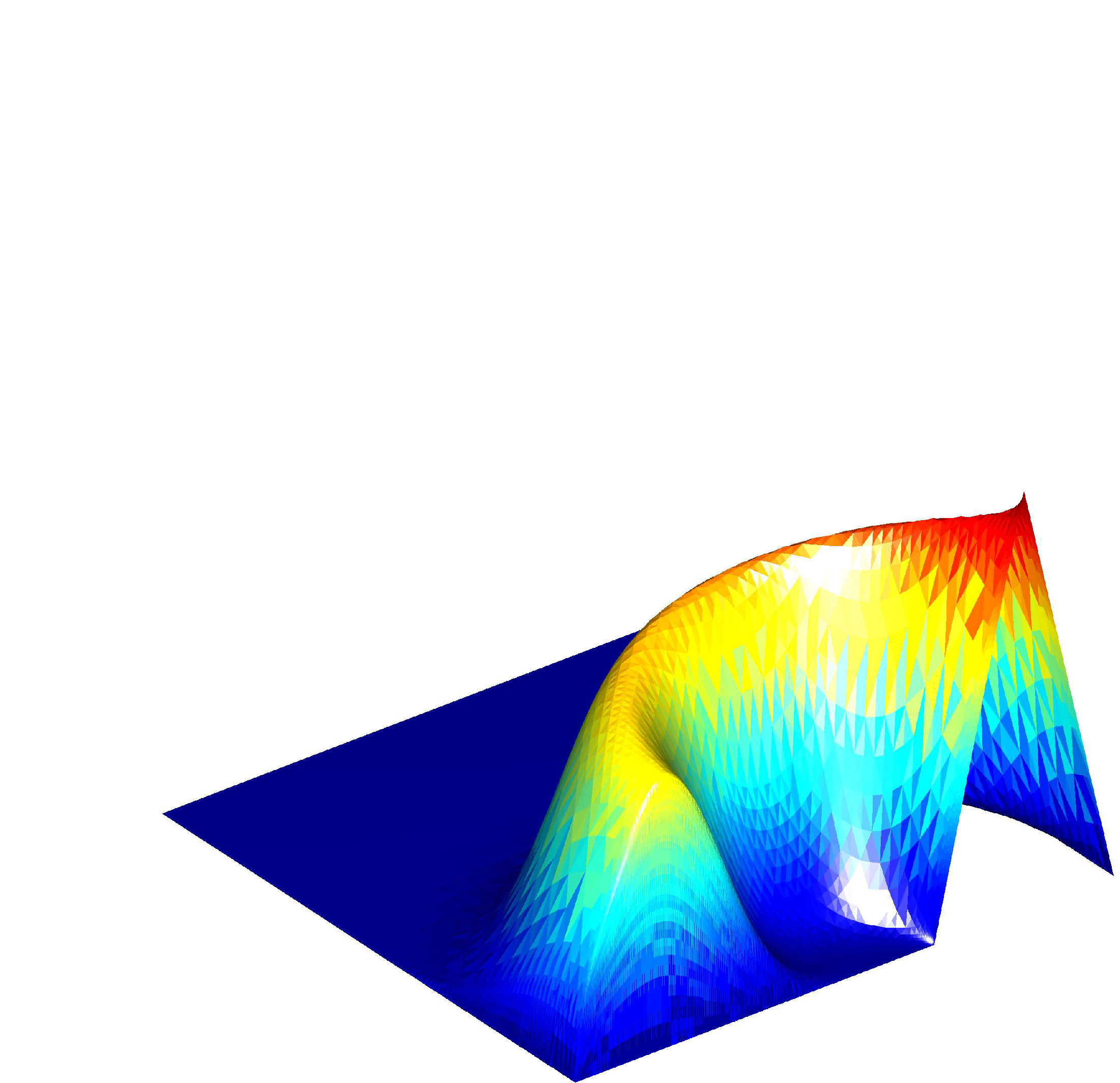}%
\\
\clippedGraphh{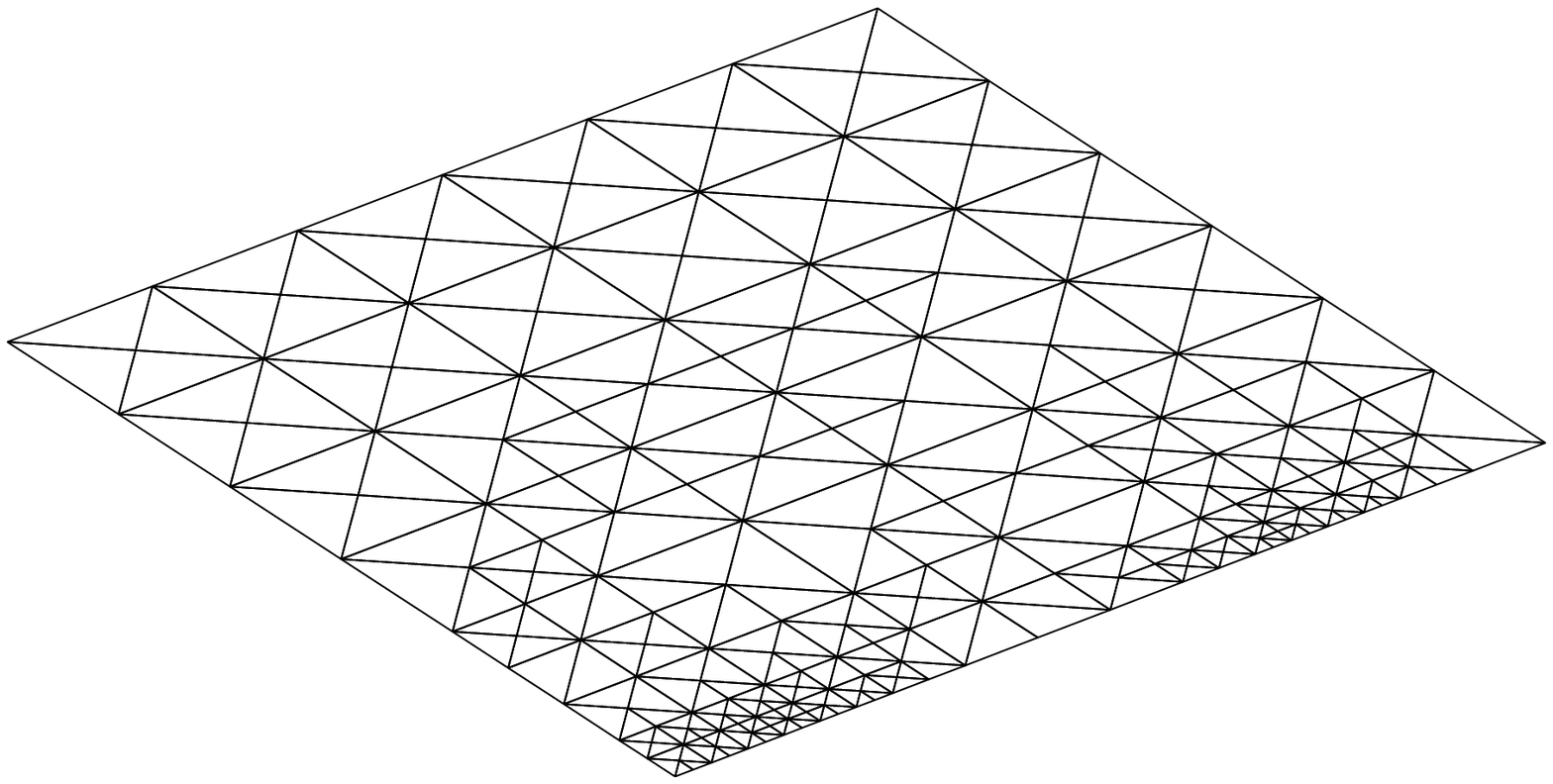}%
&
\clippedGraphh{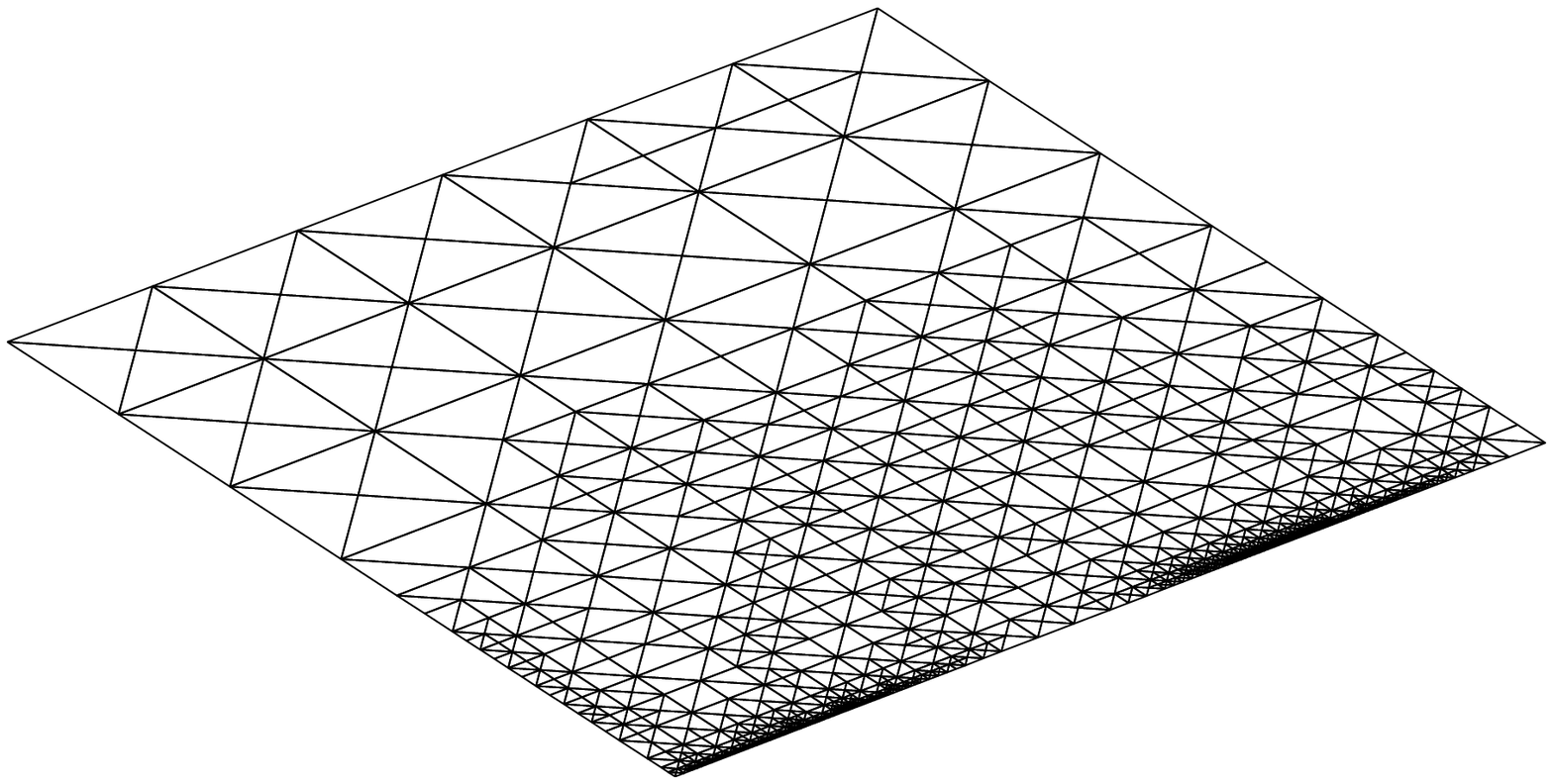}%
&
\clippedGraphh{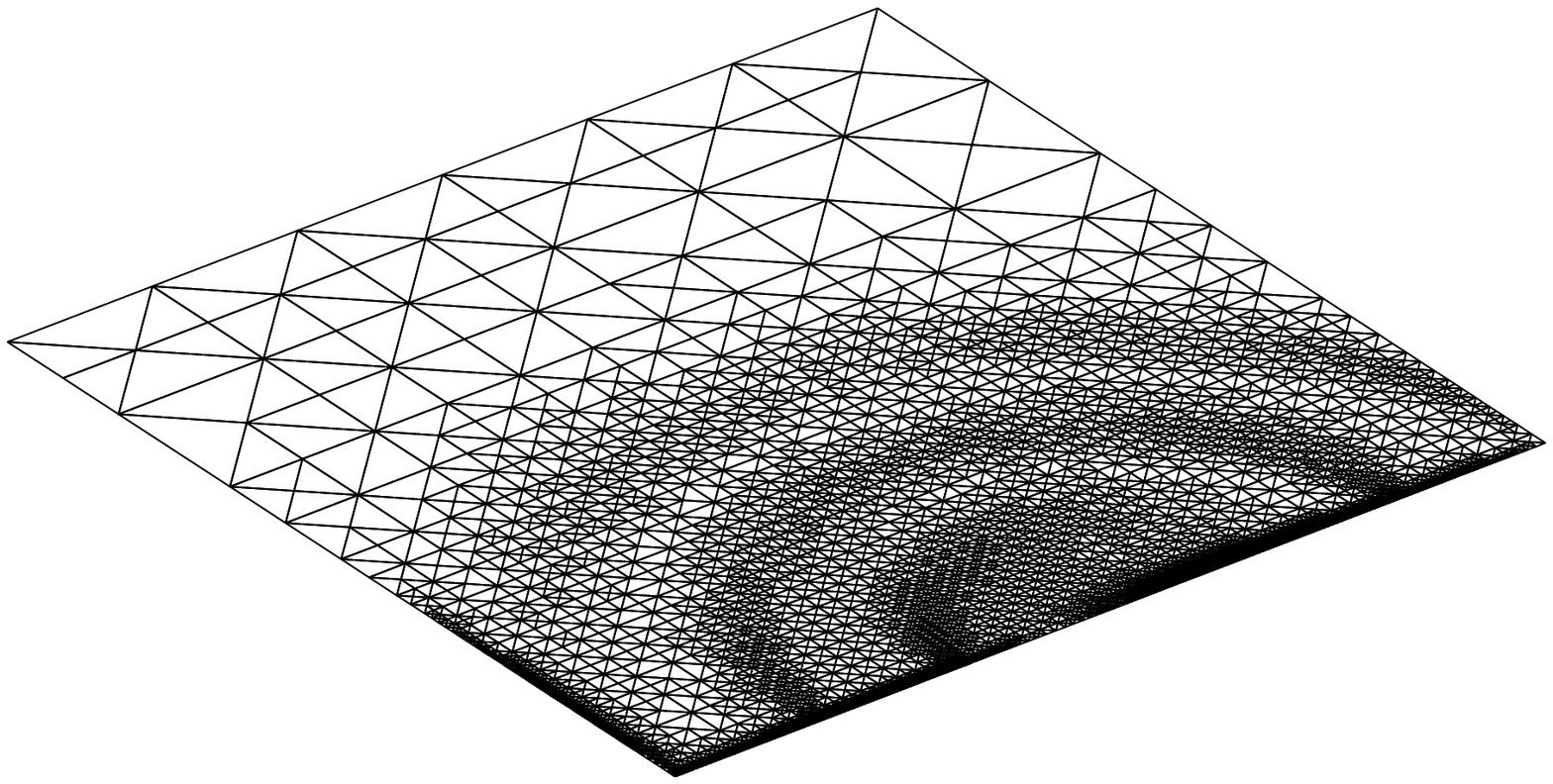}%
\\
{\tiny $\#\TT_{6} = 337$}
&
{\tiny $\#\TT_{12} = 1{,}798$}
&
{\tiny $\#\TT_{18} = 20{,}634$}
\end{tabular}
}
\caption{Example from Section~\ref{section:example:afem2}: 
Primal approximations $U_\ell$ (top), dual approximations $Z_\ell$ (middle) and adaptively generated meshes $\TT_\ell$ (bottom) for $\ell\in\{6,12,18\}$ (from left to right) as output of Algorithm~\ref{algorithm:mod} for $\theta=0.6$ and $\nu = 10^{-3}$. Although we use a non-stabilized Galerkin scheme, initial oscillations in unresolved boundary layers are picked up  immediately by the adaptive algorithm for both, the primal as the dual solution.}
\label{fig:fem2:approximations}
\end{figure}

\begin{figure}[t]
\psfrag{estu}{\tiny$\eta_u$}
\psfrag{estz}{\tiny$\eta_z$}
\psfrag{estz*estu}{\tiny$\eta_u\eta_z$}
\psfrag{est}[c][c]{\tiny estimators}
\psfrag{error}[c][c]{\tiny error resp. estimators}
\psfrag{err}[c][c]{\tiny error}
\psfrag{o12}{\tiny $\mathcal{O}(N^{-1/2})$}
\psfrag{o1}{\tiny $\mathcal{O}(N^{-1})$}
\psfrag{nE}[c][c]{\tiny number of elements $N=\#\TT_\ell$}
{%
\psfrag{Algorithm B (nu=1e-03)}[c][c]{\tiny Algorithm B ($\nu = 10^{-3}$)}
\psfrag{Algorithm B}{}
\includegraphics[scale=0.4]{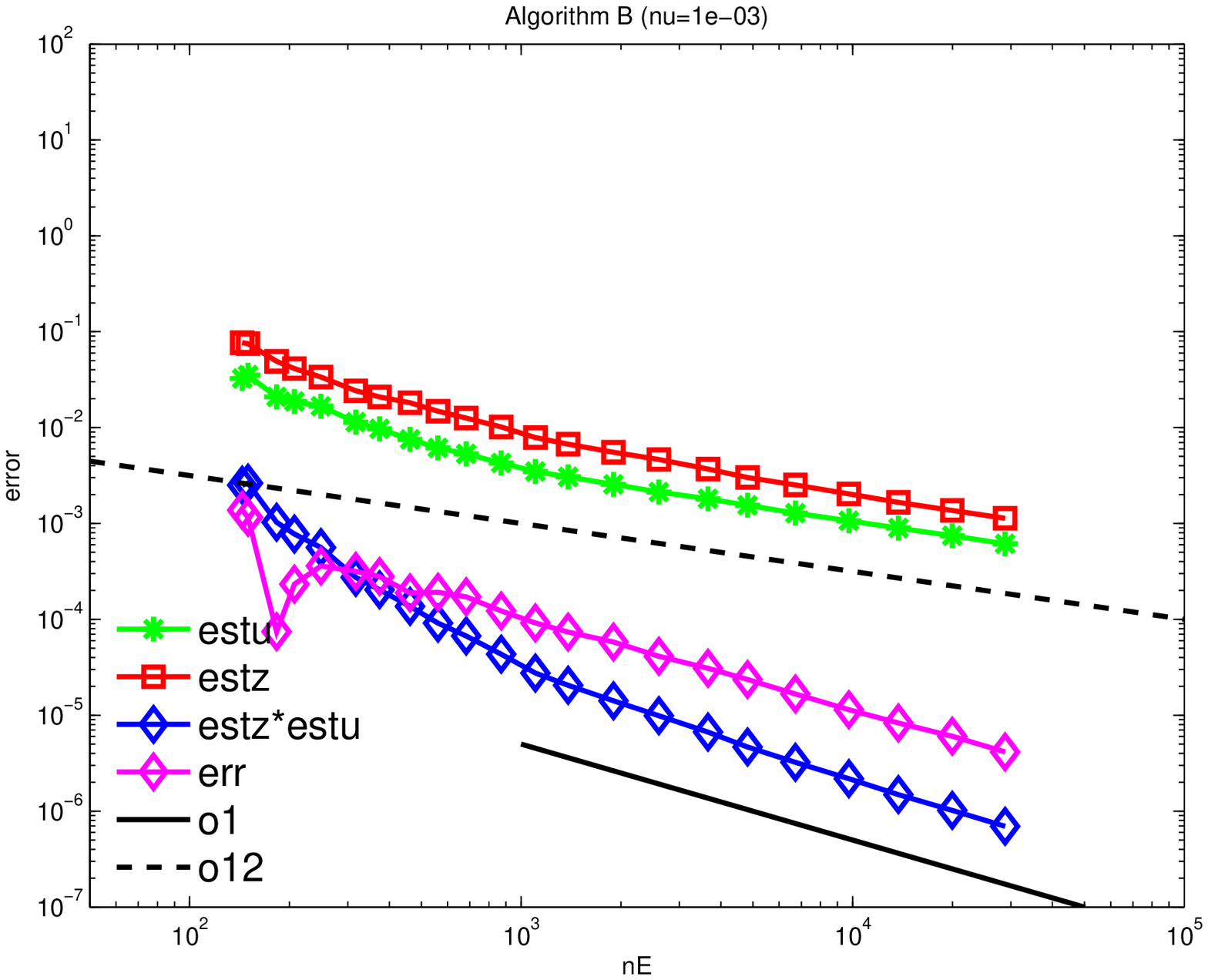}
\qquad
\begin{tabular}[b]{c}
\includegraphics[scale=0.24]{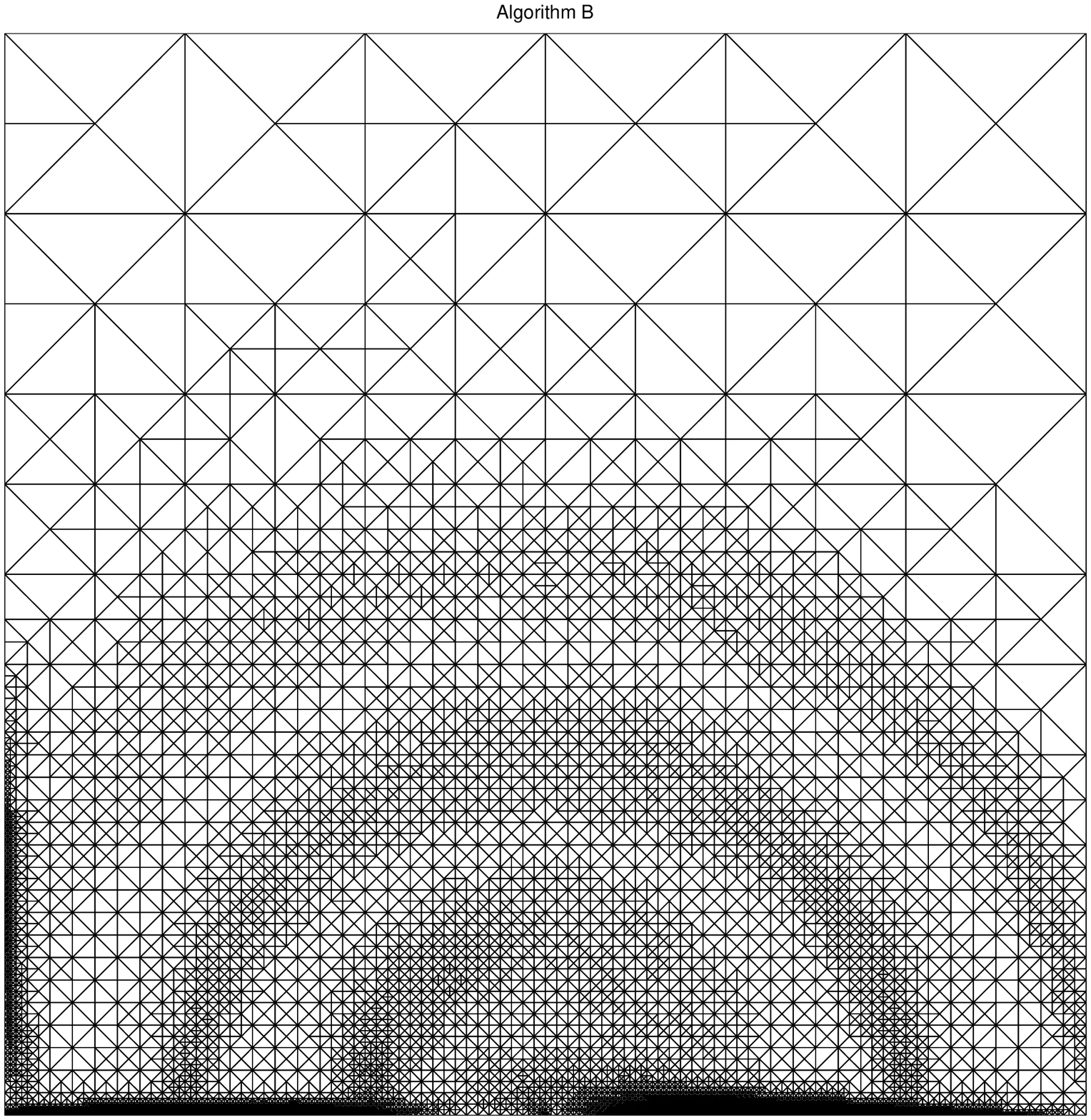}
\\
{\tiny $\#\TT_{22} = 28{,}839$} 
\end{tabular}
}
\\
{%
\psfrag{Algorithm B (nu=1e-04)}[c][c]{\tiny Algorithm B ($\nu = 10^{-4}$)}
\psfrag{Algorithm B}{}
\includegraphics[scale=0.4]{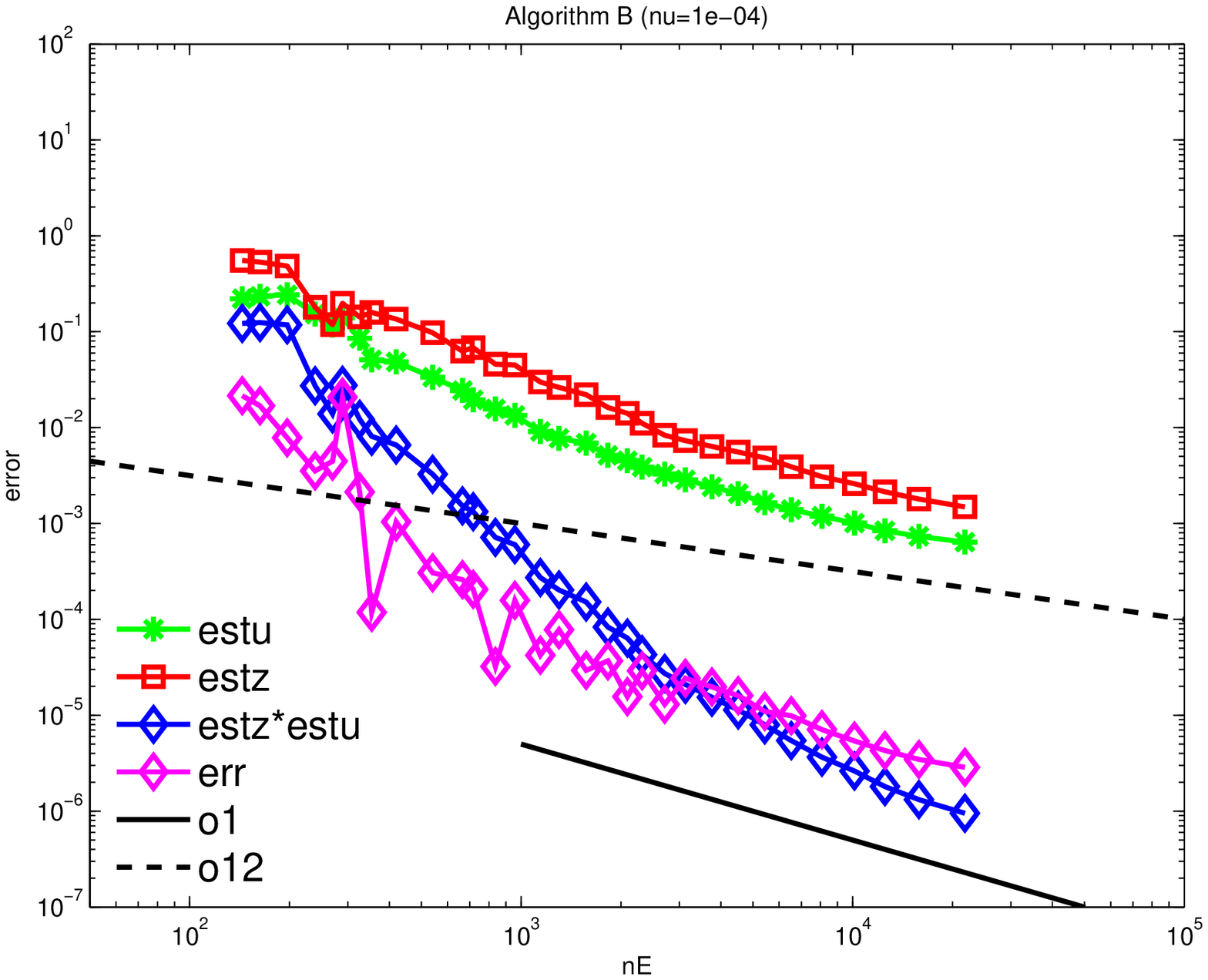}
\qquad
\begin{tabular}[b]{c}
\includegraphics[scale=0.24]{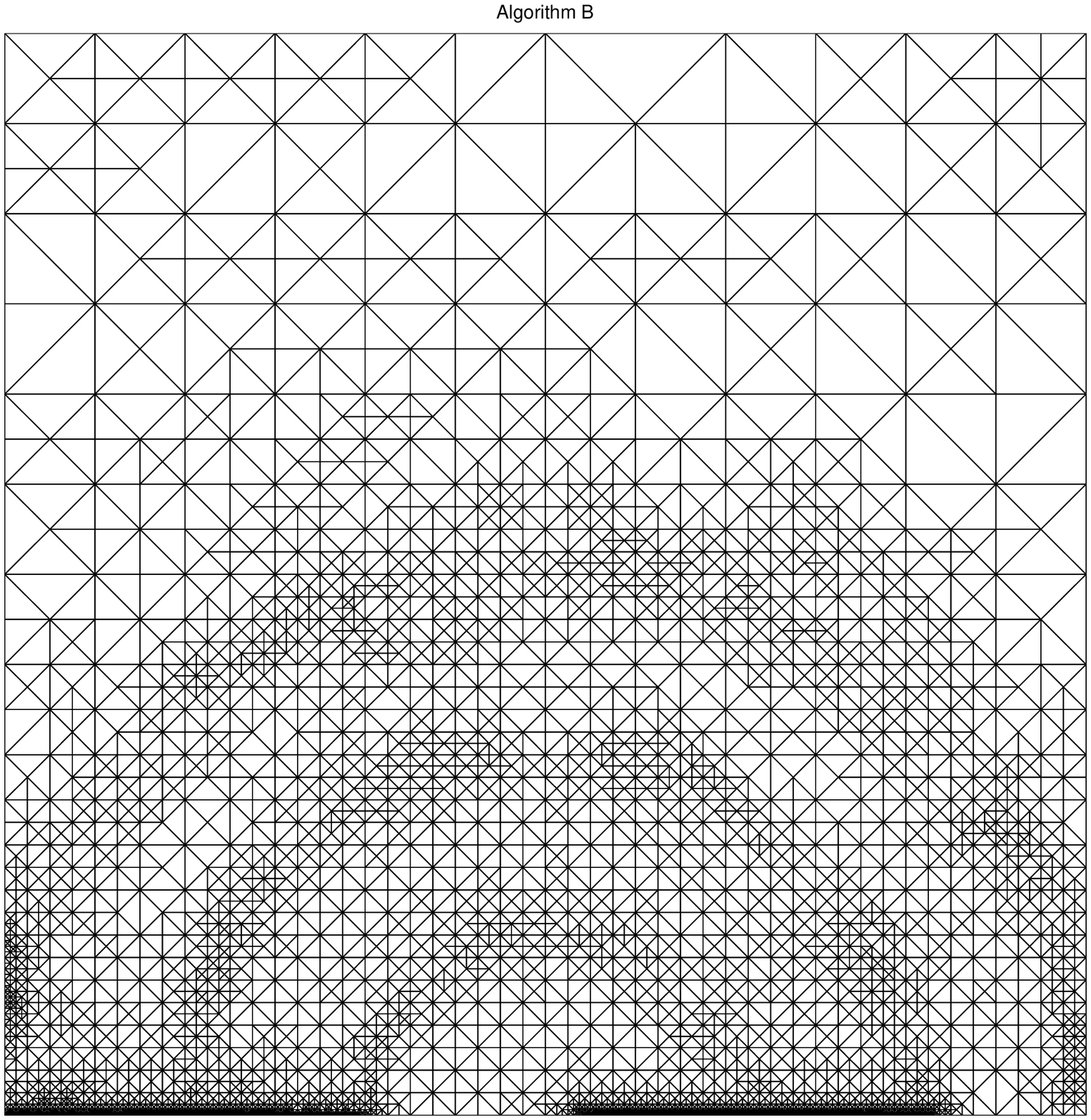}
\\
{\tiny $\#\TT_{31} = 21{,}815$} 
\end{tabular}
}
\\
{%
\psfrag{Algorithm B (nu=1e-05)}[c][c]{\tiny Algorithm B ($\nu = 10^{-5}$)}
\psfrag{Algorithm B}{}
\includegraphics[scale=0.4]{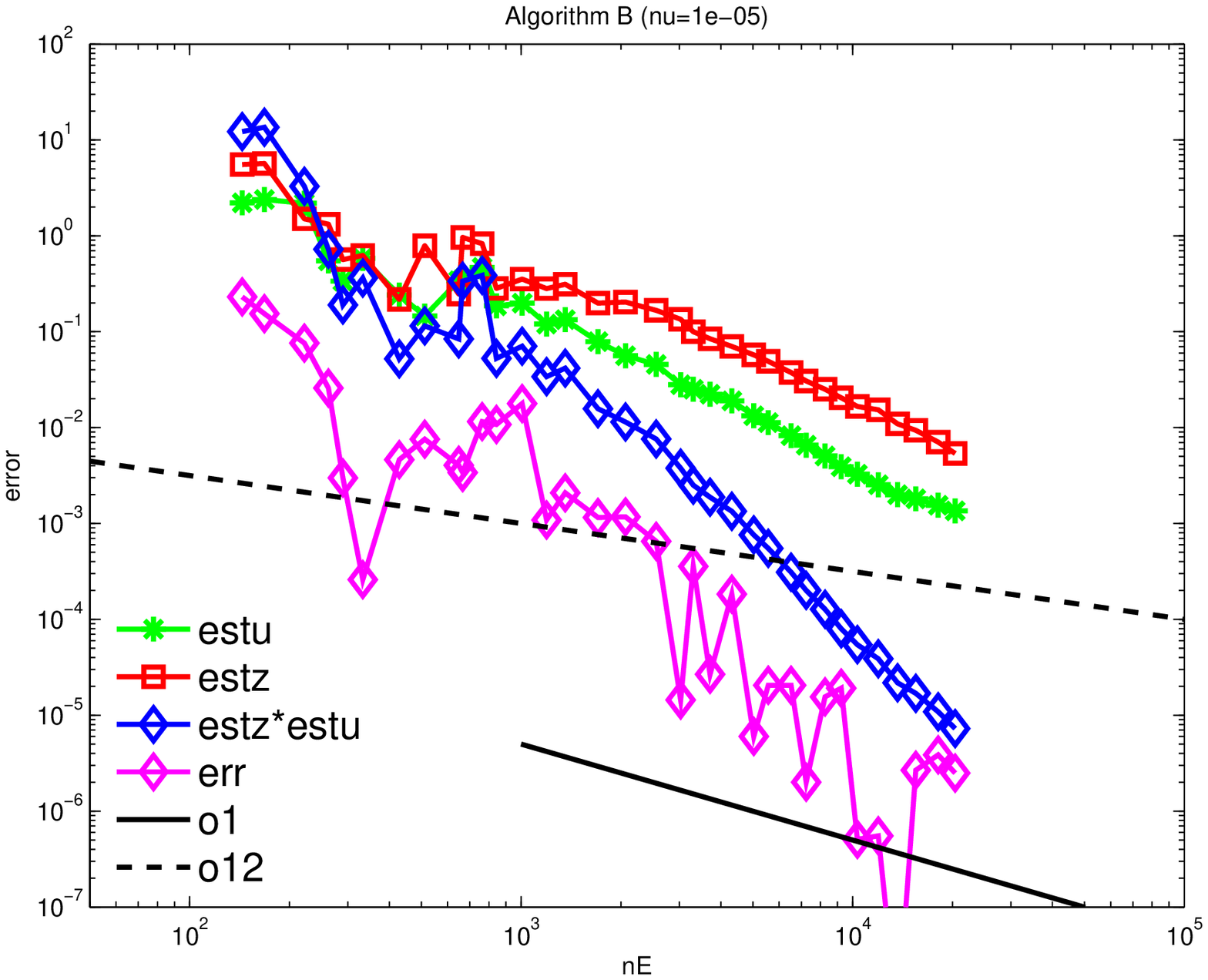}
\qquad
\begin{tabular}[b]{c}
\includegraphics[scale=0.24]{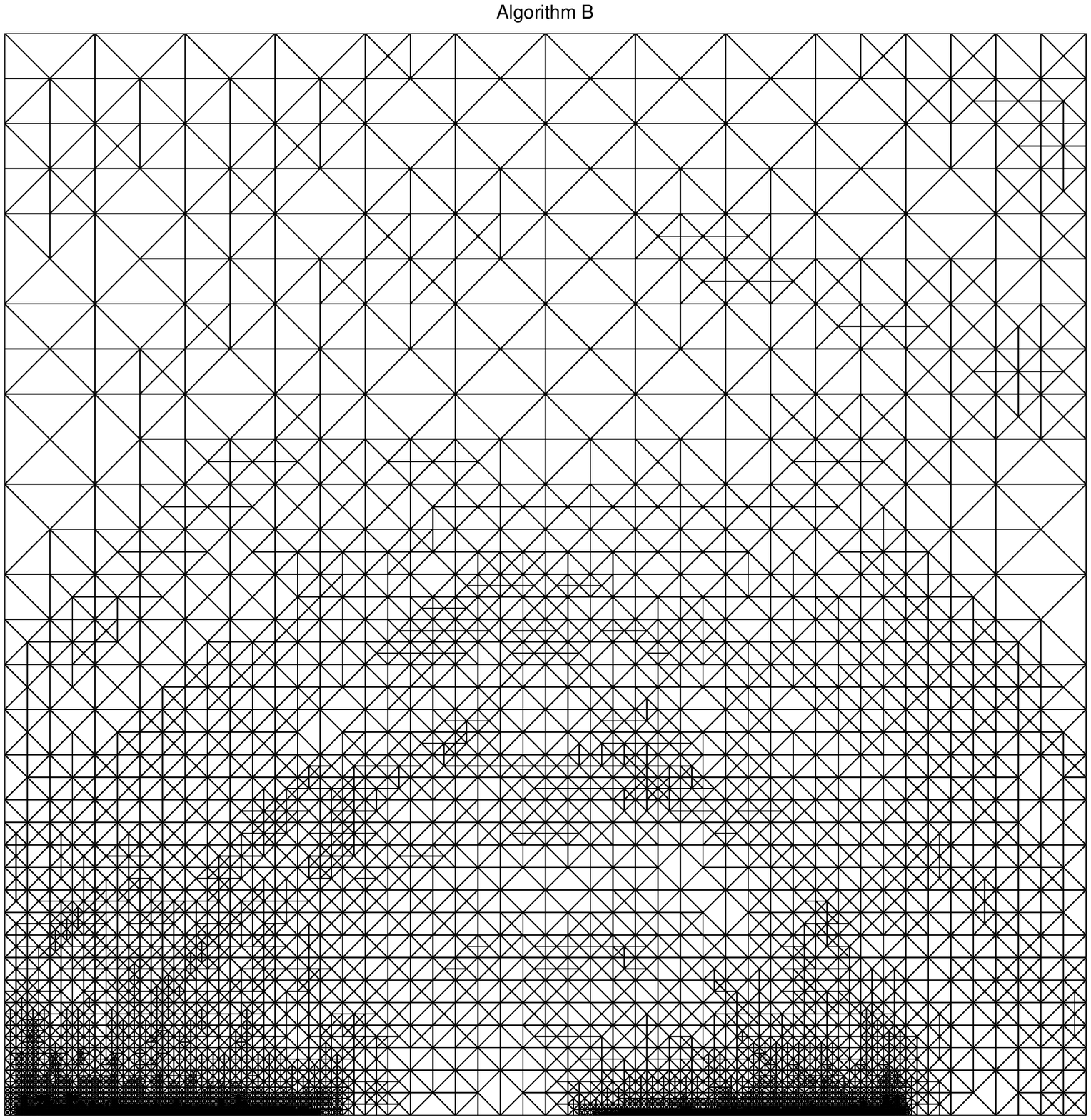}
\\
{\tiny $\#\TT_{34} = 20{,}361$} 
\end{tabular}
}
\caption{Example from Section~\ref{section:example:afem2}: 
To study the robustness of the goal-oriented algorithm with respect to the diffusion coefficient $\nu\in\{10^{-3},10^{-4},10^{-5}\}$ (left, from top to bottom), we plot the estimators $\eta_{u,\ell}$ and $\eta_{z,\ell}$, the estimator product $\eta_{u,\ell}\eta_{z,\ell}$, as well as the goal error $|N_z(u) - N_{z,\ell}(U_\ell)|$ as output of Algorithm~\ref{algorithm:mod} with $\theta=0.6$ over the numbers of elements $\#\TT_\ell$ (left). We show some related discrete meshes with \mbox{$> 20{,}000$} elements (right).}
\label{fig:fem2:robustness}
\end{figure}

\subsection{Numerical experiment II: Flux-oriented adaptive FEM for convection--diffusion}
\label{section:example:afem2}
We consider a numerical experiment similar to~\cite[Section~5.3]{mn} for some convection-diffusion problem in 2D. The goal of this experiment is to verify the optimal convergence of Algorithm~\ref{algorithm}--\ref{algorithm:bet} for the flux quantity of interest~\eqref{dpr:flux} and, moreover, to illustrate this for a \emph{nonsymmetric} second-order elliptic operator, which is covered by our theory. 

Let $\Omega = (0,1)^2 \subset \mathbb{R}^2$ be the unit square. We set~$\matrix{A} = \nu\matrix{I}$, with~$\nu>0$ the diffusion coefficient, $\matrix{b} = (y,\tfrac{1}{2}-x)$, which is a rotating convective field around $(\tfrac{1}{2},0)$, and $c=0$. According to $\operatorname{div} \matrix{b} = 0$, it holds
\begin{alignat*}{2}
 \mathcal{L} &= -\nu \Delta  + \matrix{b} \cdot \nabla 
 \qquad \text{and}
 \qquad 
 \mathcal{L}^T &= -\nu \Delta  - \matrix{b} \cdot \nabla \,.
\end{alignat*}

We set~$f(v) = 0$ and consider non-homogeneous Dirichlet data on~$\partial \Omega$ for the primal problem, a pulse, defined by the continuous piecewise linear function
\begin{alignat*}{2}
  u_{\rm Dir}(x,y) &= \begin{cases}
    6 ( x-\tfrac{1}{6} ) \quad & \tfrac{1}{6} \le x < \tfrac{1}{3} 
    \,,\, y=0
\\
    6( \tfrac{1}{2} - x ) \quad & \tfrac{1}{3} \le x < \tfrac{1}{2} 
    \,,\, y=0
\\ 
    0 & \text{otherwise}\,.
  \end{cases}
\end{alignat*}
Note that $u_{\rm Dir}$ trivially extends to some discrete function $u_{\rm Dir}\in \SS^1(\TT_0)$ if $\TT_0$ is chosen appropriately. Therefore, we can rewrite the problem into a homogeneous Dirichlet problem. To that end, write $u=u_0 + u_{\rm Dir}$ with $u_0\in H^1_0(\Omega)$ and solve
\begin{align*}
a(u_0,v)= f(v) - a(u_{\rm Dir}, v)\quad\text{for all }v\in H^1_0(\Omega).
\end{align*}
Note that the additional term on the right-hand side is of the form ${\rm div}\boldsymbol{\lambda} + \lambda$ for some $\TT_0$-element wise constant $\boldsymbol{\lambda}$ and some $\lambda\in L^2(\Omega)$. A direct computation shows that the weighted-residual error estimator with respect to $u_0$ coincides with $\eta_{u,\ell}$. Arguing as in the proof of Theorem~\ref{theorem:flux}, we see that the estimator satisfies the axioms~\eqref{ass:stable}--\eqref{ass:reliable}. Altogether, the problem thus fits in the frame of our analysis.

The primal solution corresponds to the clockwise convection--diffusion of this pulse. 
We choose the boundary weight function $\Lambda:\partial\Omega \rightarrow \mathbb{R}$ as a shifted version of the above pulse: 
\begin{alignat*}{2}
  \Lambda(x,y) &= \begin{cases}
    6 ( x-\tfrac{2}{3} ) \quad & \tfrac{2}{3} \le x < \tfrac{5}{6} 
    \,,\, y=0
\\
    6( 1 - x ) \quad & \tfrac{5}{6} \le x < 1 
    \,,\, y=0
\\ 
    0 & \text{otherwise}\,.
  \end{cases}
\end{alignat*}
The dual solution corresponds to the counter-clockwise convection--diffusion of this pulse. For small~$\nu$, the (primal and dual) pulses are transported from~$\partial \Omega$ into $\Omega$ and eventually back to~$\partial\Omega$ where a boundary layer develops. 
See Figure~\ref{fig:fem2:MNtestcase} (left) for an illustration of the supports of the primal and dual Dirichlet data, and the primal and dual convective fields. 

All discrete approximations are computed with lowest-order finite elements of degree~$p=1$. The uniform initial triangulation $\TT_0$ is as shown in Figure~\ref{fig:fem2:MNtestcase} (right) 
ensures that the (primal and dual) Dirichlet data belong to the discrete trace space~$\SS^1(\TT_0|_\Gamma)$.

As shown in Figure~\ref{fig:fem2:QconvABC},
Algorithm~\ref{algorithm}--\ref{algorithm:bet} yield optimal convergence rates for the flux quantity of interest. For~$\nu=10^{-3}$ and a large range of values of~$\theta\in\{0.1,\dots,0.9\}$, we observe the optimal convergence rate $\OO(N^{-1})$, while uniform mesh-refinement appears to be slightly suboptimal. 

To compare the overall performance of the different algorithms, Figure~\ref{fig:fem2:comparison} visualizes over different marking parameters $\theta\in\{0.1,\dots,0.9\}$ the cumulative number of elements $N_{\rm cum}$ which is necessary to reach a prescribed accuracy of~$\eta_{u,\ell} \eta_{z,\ell} \le 10^{-4}$ vs.\ the marking parameter $\theta\in\{0.1,\dots,0.9\}$; see~\eqref{eq:Ncum} for the definition and interpretation of $N_{\rm cum}$. For Algorithms~\ref{algorithm}--\ref{algorithm:bet}, we observe that $N_{\mathrm{cum}}$ is smallest for relatively large values $\theta\ge 0.5$, with Algorithm~\ref{algorithm} being less efficient than Algorithm~\ref{algorithm:mod} and~\ref{algorithm:bet}. Overall, Algorithm~\ref{algorithm:mod} with~$\theta=0.6$ seems to be the best choice. 

Figure~\ref{fig:fem2:approximations} shows several approximations and meshes obtained with Algorithm~\ref{algorithm:mod}. Because~$\nu=10^{-3}$ is relatively small, both the primal and the dual solution have significant boundary layers. These layers as well as the weak singularities coming from the kinks in the Dirichlet data are well captured by the adaptive algorithm. 

Figure~\ref{fig:fem2:robustness} illustrates the effect of varying~$\nu\in\{10^{-3},10^{-4},10^{-5}\}$. The optimal convergence rate of the estimator product is observed for the indicated values of~$\nu$, however, the pre-asymptotic regime is longer for smaller values of~$\nu$. This is to be expected, as the hidden constant in~\eqref{eq:flux:estimator} depends on the reliability constant for the estimators, which in turn depends on~$\nu$.

\section{Goal oriented BEM}\label{section:bem}

\noindent
In this section, we extend ideas from~\cite{fkmp} and prove that our abstract frame of 
convergence and optimality of goal-oriented adaptivity applies, in particular, to the BEM.

\subsection{Model problem}\label{abem:model}
Let $\Gamma\subseteq \partial\Omega$ denote some relatively open 
boundary part of the Lipschitz domain $\Omega\subset \R^d$, $d=2,3$. Given $F,\Lambda\in H^1(\Gamma)$, we aim to compute the weighted boundary flux
\begin{align}\label{eq:bem:goal}
g(u):=\int_\Gamma \Lambda u\,ds,
\end{align}
where $u$ solves the weakly-singular integral equation
\begin{align}\label{ex:bem}
 \slp u(x):=\int_\Gamma G(x,y)u(y)\,dy = F(x)\quad\text{almost everywhere on }\Gamma.
\end{align}
Here, $G:\R^2\setminus\{0\}\to\R$ denotes the Newton kernel
\begin{align*}
 G(x,y):=\begin{cases}
          -\frac{1}{2\pi}\log|x-y| &\text{for }d=2,\\
          \frac{1}{4\pi|x-y|}&\text{for }d=3.
         \end{cases}
\end{align*}
The simple-layer operator extends to a linear and continuous operator $\slp:\,\widetilde H^{-1/2}(\Gamma)\to H^{1/2}(\Gamma)$, where $H^{1/2}(\Gamma):=\set{\widehat v|_\Gamma}{\widehat v\in H^1(\Omega)}$ is the trace space of $H^1(\Omega)$ and $\widetilde H^{-1/2}(\Gamma)$ denotes its dual space with respect to the extended $L^2$-scalar product; see, e.g.,~\cite{mclean,hw,ss} for the mapping properties of $\slp$ and the functional analytic setting.
For $d=3$ as well as supposed that ${\rm diam}(\Omega)<1$ for $d=2$, the induced bilinear form 
\begin{align*}
 \bform{u}{v}:=\dual{\slp u}{v}
 := \int_\Gamma(\slp u)(x)v(x)\,dx\quad\text{for }u,v\in\XX:=\widetilde{H}^{-1/2}(\Gamma)
\end{align*}
is continuous, symmetric,  and $\widetilde{H}^{-1/2}(\Gamma)$-elliptic.
In particular, $\enorm{v}^2:=\bform{v}{v}$ defines an equivalent norm on $\widetilde{H}^{-1/2}(\Gamma)$. Moreover, the problem fits in the frame of Section~\ref{section:motivation}. 
More precisely and according to the Hahn-Banach theorem,~\eqref{ex:bem} is equivalent to~\eqref{eq:primal}, where the
right-hand side of~\eqref{eq:primal} reads $f(v):=\int_\Gamma Fv\,dx$. Moreover, the goal functional from~\eqref{eq:bem:goal} satisfies $g\in H^{-1/2}(\Gamma)^*=H^{1/2}(\Gamma)$, where the integral is understood as the duality pairing between $H^{-1/2}(\Gamma)$ and its dual $H^{1/2}(\Gamma)$.

\subsection{Discretization}\label{abem:discretization}
Let $\TT_\star$ be a regular triangulation of $\Gamma$ into affine line segments for $d=2$ resp.\ flat surface triangles for $d=3$. For each element $T\in\TT_\star$, let $\gamma_T:T_{\rm ref}\to T$ be an affine bijection, where the reference element is $T_{\rm ref} = [0,1]$ for $d=2$ resp.\ $T_{\rm ref} = {\rm conv}\{(0,0),(0,1),(1,0)\}$ for $d=3$. For some polynomial degree $p\ge1$, define
\begin{align*}
 \XX_\star := \PP^p(\TT_\star) 
 := \set{V_\star:\Gamma\to\R}{V_\star\circ\gamma_T\in\PP^p(T_{\rm ref}) \text{ for all $T\in\TT_\star$}},
\end{align*}
where $\PP^p(T_{\rm ref}) := \set{q\in L^2(T_{\rm ref})}{q \text{ is polynomial of degree $\le p$ on $T_{\rm ref}$}}$. Let $U_\star,Z_\star\in\XX_\star$ be the unique BEM solutions of~\eqref{eq:primal:discrete} resp.~\eqref{eq:dual:discrete}, i.e.,
\begin{subequations} \label{eq:abem:discrete}
\begin{align}
 &U_\star\in\PP^p(\TT_\star)
 \quad\text{such that}\quad
 a(U_\star,V_\star) = f(V_\star)\quad\text{for all }V_\star\in \PP^p(\TT_\star),\\
 &Z_\star\in\PP^p(\TT_\star)
 \quad\text{such that}\quad
 a(V_\star,Z_\star) = g(V_\star)\quad\text{for all }V_\star\in \PP^p(\TT_\star).
\end{align}
\end{subequations}

\subsection{Residual error estimator}
The residual error estimators from~\cite{cms} for the discrete primal problem~\eqref{eq:primal:discrete} and the discrete dual problem~\eqref{eq:dual:discrete} read
\begin{align}\label{ex:bem:estimator:primal}
\eta_{u,\star}(T)^2:=h_T\norm{\nabla (\slp U_\star-F)}{L^2(T)}^2 \quad\text{and}\quad 
\eta_{z,\star}(T)^2:=h_T\norm{\nabla (\slp Z_\star-\Lambda)}{L^2(T)}^2.
\end{align}
The error estimators satisfy reliability~\eqref{eq:reliable:laxmilgram}; see, e.g.,~\cite{cms}.
The abstract analysis of Section~\ref{section:motivation} thus results in
\begin{align}\label{eq:abem:estimator}
 |g(u)-g(U_\star)| \lesssim \eta_{u,\star}\eta_{z,\star},
\end{align}
and we aim for optimal convergence of the right-hand side.

\subsection{Verification of axioms}
With 2D newest vertex bisection from~\cite{stevenson:nvb} for $d=3$ resp. the extended 1D bisection from~\cite{eps65} for $d=2$ as mesh-refinement strategy, 
the assumptions of Section~\ref{section:finemesh} are satisfied. It remains to verify 
the axioms~\eqref{ass:stable}--\eqref{ass:reliable},
where $\dist{w}{\TT_\ell}{\TT_\star}:=\enorm{W_\ell-W_\star}\simeq \norm{W_\ell-W_\star}{\widetilde H^{-1/2}(\Gamma)}$.

\begin{theorem}
Consider the model problem of Section~\ref{abem:model}. Then, the conforming discretization~\eqref{eq:abem:discrete} of Section~\ref{abem:discretization} with the
residual error estimators~\eqref{ex:bem:estimator:primal} satisfies stability~\eqref{ass:stable}, reduction~\eqref{ass:reduction} with $\q{reduction}=2^{-1/(d-1)}$,
quasi-orthogonality~\eqref{ass:orthogonal}, and discrete reliability~\eqref{ass:reliable} with $\RR_w(\TT_\ell,\TT_\star) = \set{T\in\TT_\ell}{T\cap\bigcup(\TT_\ell\backslash\TT_\star)\neq \emptyset}$, i.e., $\R_w(\TT_\ell,\TT_\star)$ consists of the refined elements plus one additional layer.
In particular, the Algorithms~\ref{algorithm}--\ref{algorithm:bet} are linearly convergent with optimal rates in the sense of Theorem~\ref{theorem:linear}, \ref{theorem:optimal}, \ref{theorem:optimal:mod}, and~\ref{theorem:optimal:bet} for the upper bound in~\eqref{eq:abem:estimator}.
\end{theorem}

\begin{proof}
The assumptions~\eqref{ass:stable}--\eqref{ass:reduction} and~\eqref{ass:reliable} are proved in~\cite[Proposition~4.2, Proposition~5.3]{fkmp} for the lowest-order case. The general case is proved in~\cite{part1}.
The quasi-orthogonality~\eqref{ass:orthogonal} follows from symmetry of $a(\cdot,\cdot)$ and~\eqref{ass:reliable}; see Remark~\ref{remark:laxmilgram}.
\end{proof}

\begin{figure}
\psfrag{z0}{\tiny$z_0$}
\psfrag{x}{\tiny$x$}
\psfrag{y}{\tiny$y$}
\psfrag{dual}{\scalebox{.5}{dual}}
\psfrag{primal}{\scalebox{.5}{primal}}
\psfrag{parameter domain}[cc][cc]{\tiny parameter domain $0\leq s \leq 2$}
\psfrag{nE = 279}[cc][cc]{\tiny Algorithm~\ref{algorithm:mod}, $\#\TT_{20}=279$}
 \includegraphics[scale=0.4]{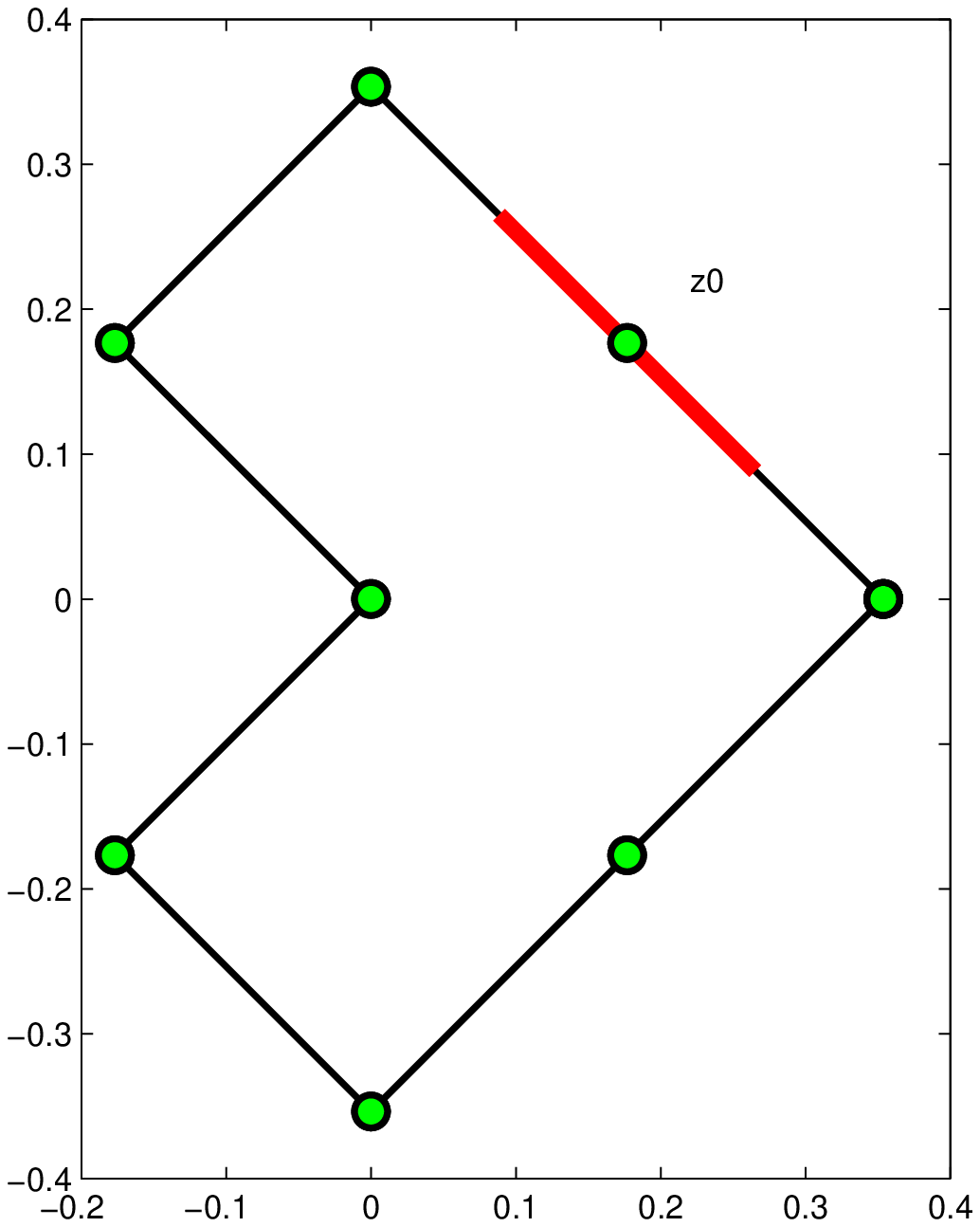}
  \includegraphics[scale=0.4]{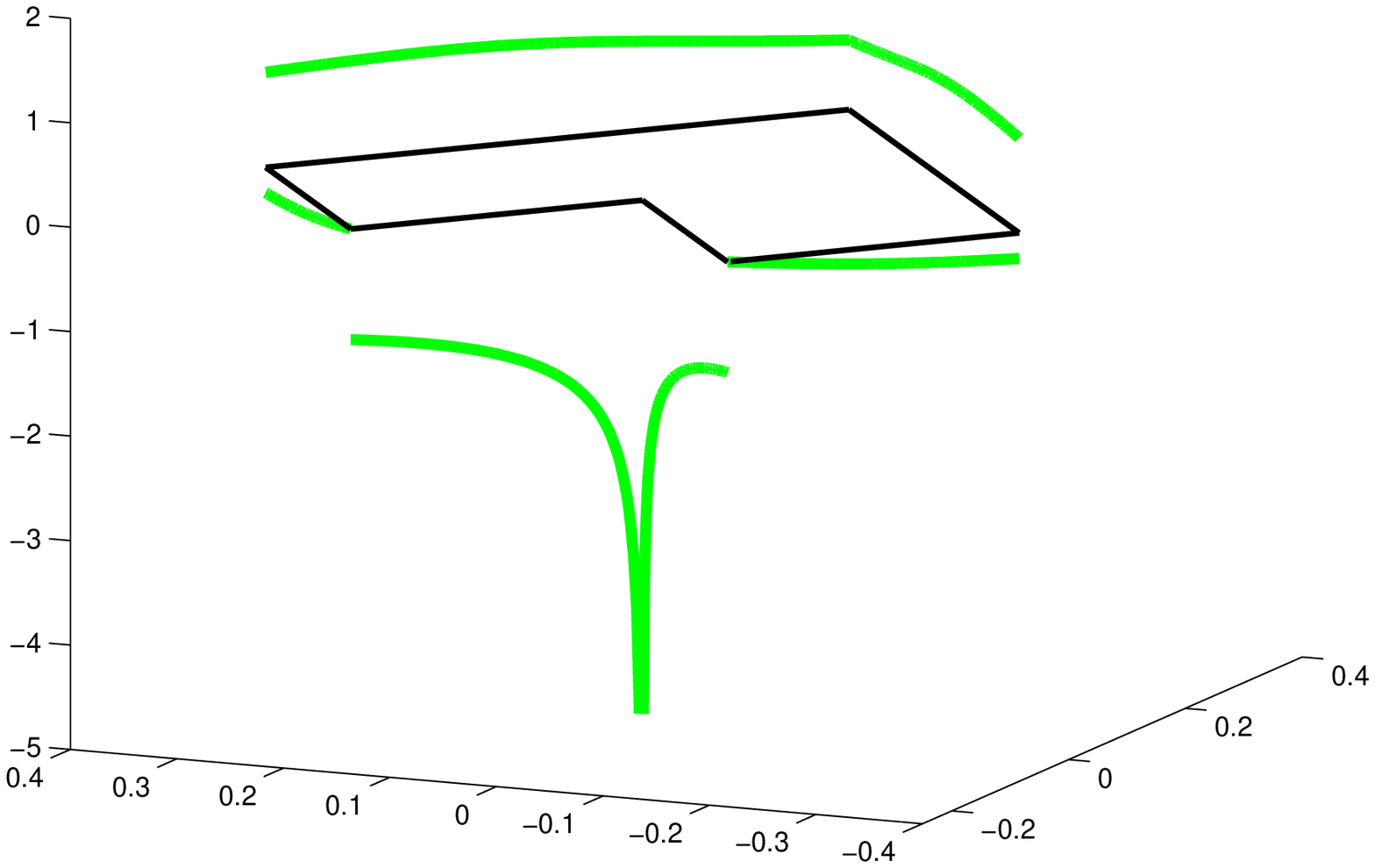}\vspace{2mm}
   \includegraphics[scale=0.4]{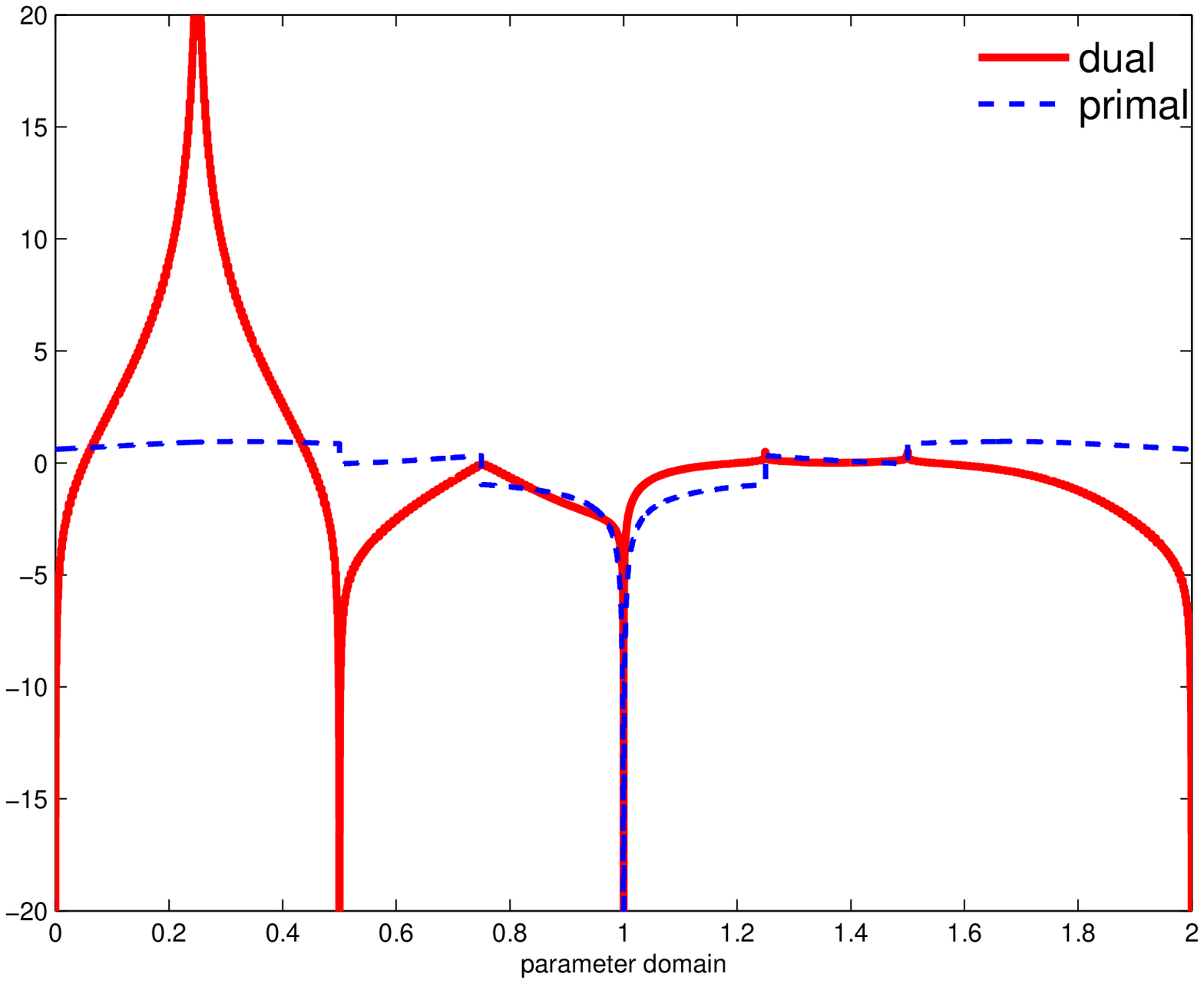}
    \includegraphics[scale=0.4]{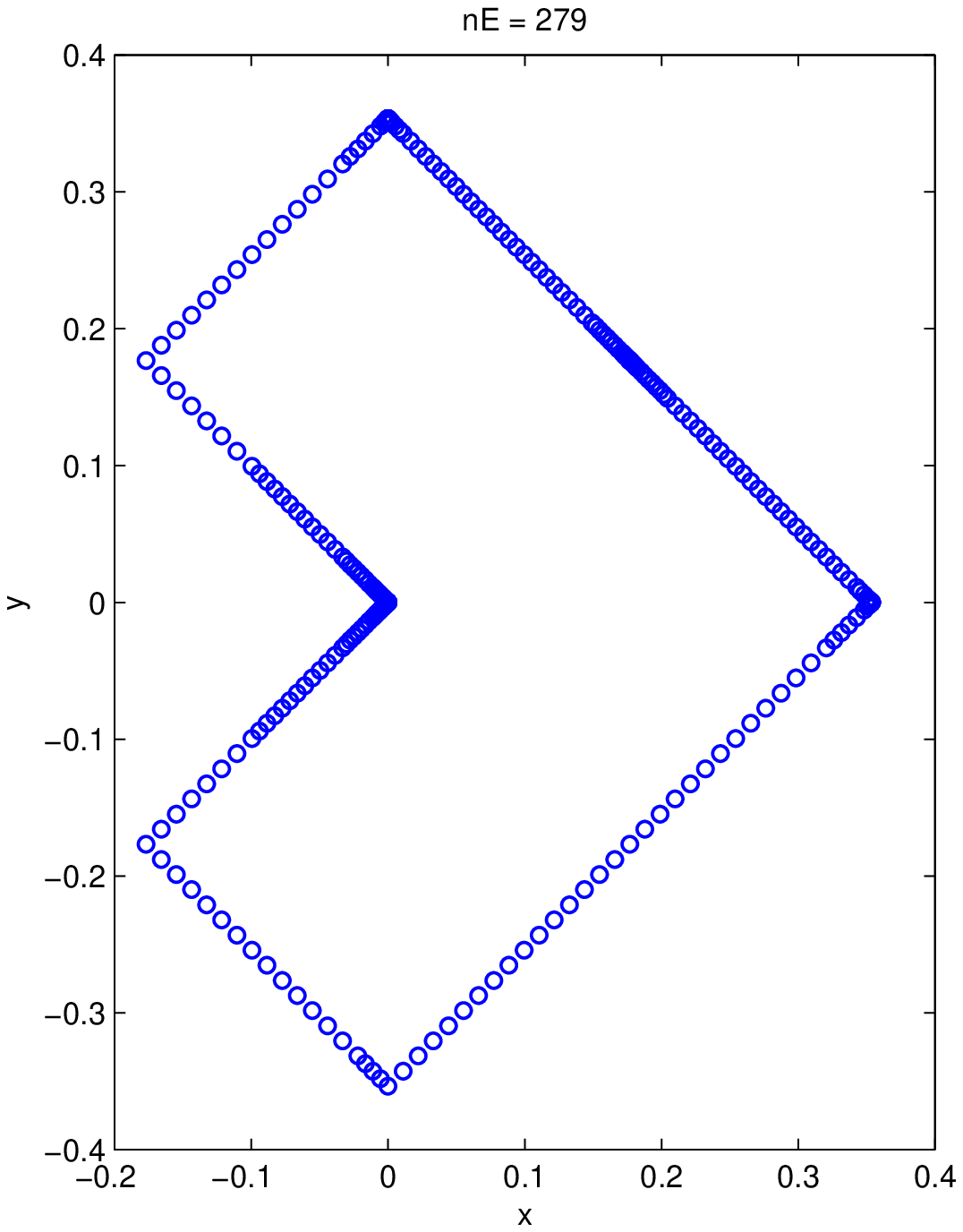}
 \caption{Example from Section~\ref{section:ex:conf} with conforming weight: Domain $\Omega$ with initial triangulation $\TT_0$ (upper left), exact solution $u$ plotted over the boundary (upper right), exact primal and dual solution plotted over the arc-length, where $s=1$ corresponds to the reentrant corner and $s=0.25$ corresponds to $z_0$ (lower left), and adaptive mesh with $\#\TT_{20} = 279$ elements generated by Algorithm~\ref{algorithm:mod} with $\theta=0.5$ (lower right).} 
 \label{fig:bemdom}
\end{figure}

\begin{figure}
\psfrag{0.1}{\scalebox{.5}{$0.1$}}
\psfrag{0.2}{\scalebox{.5}{$0.2$}}
\psfrag{0.3}{\scalebox{.5}{$0.3$}}
\psfrag{0.4}{\scalebox{.5}{$0.4$}}
\psfrag{0.5}{\scalebox{.5}{$0.5$}}
\psfrag{0.6}{\scalebox{.5}{$0.6$}}
\psfrag{0.7}{\scalebox{.5}{$0.7$}}
\psfrag{0.8}{\scalebox{.5}{$0.8$}}
\psfrag{0.9}{\scalebox{.5}{$0.9$}}
\psfrag{1}{\scalebox{.5}{$1.0$}}
\psfrag{Algorithm A}[cc][cc]{\tiny Algorithm A}
\psfrag{Algorithm B}[cc][cc]{\tiny Algorithm B}
\psfrag{Algorithm C}[cc][cc]{\tiny Algorithm C}
\psfrag{estu}{\scalebox{.5}{$\eta_u$}}
\psfrag{estz}{\scalebox{.5}{$\eta_z$}}
\psfrag{estz*estu}{\scalebox{.6}{$\eta_u\eta_z$}}
\psfrag{o3}{\scalebox{.5}{$\OO(N^{-3})$}}
\psfrag{o32}{\scalebox{.5}{$\OO(N^{-3/2})$}}
\psfrag{r3}{\scalebox{.5}{$\OO(N^{-3})$}}
\psfrag{r12}{\scalebox{.5}{$\OO(N^{-1/2})$}}
\psfrag{r32}{\scalebox{.5}{$\OO(N^{-3/2})$}}
\psfrag{r43}{\scalebox{.5}{$\OO(N^{-4/3})$}}
\psfrag{err}{\scalebox{.5}{error}}
\psfrag{error}[cc]{\tiny error resp.\ estimators}
\psfrag{nE}[cc]{\tiny number of elements $N=\#\TT_\ell$}
\psfrag{error}[cc]{\tiny error resp.\ estimators}%
 \includegraphics[scale=0.4]{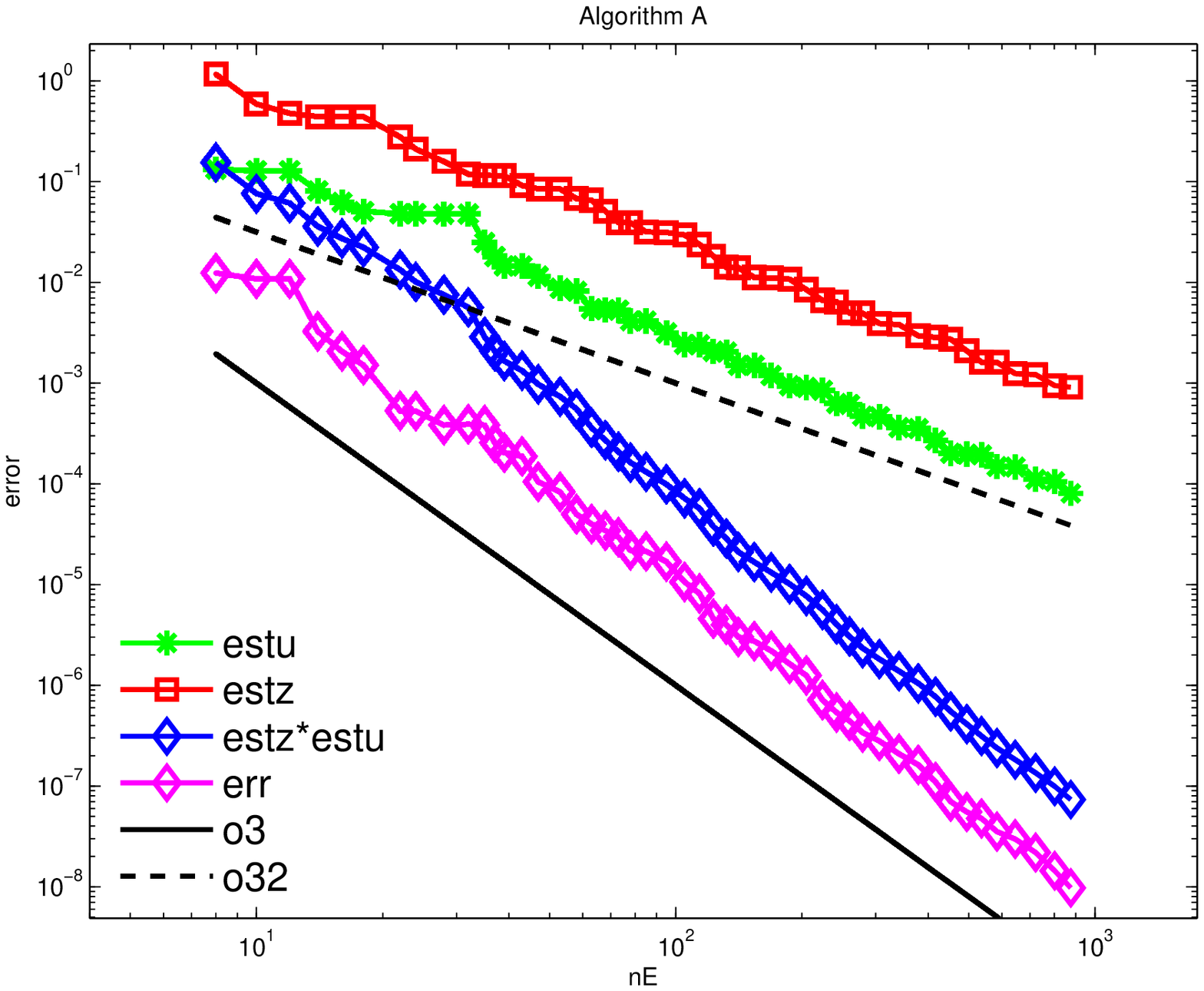} 
\psfrag{error}[cc]{\tiny error}%
 \includegraphics[scale=0.4]{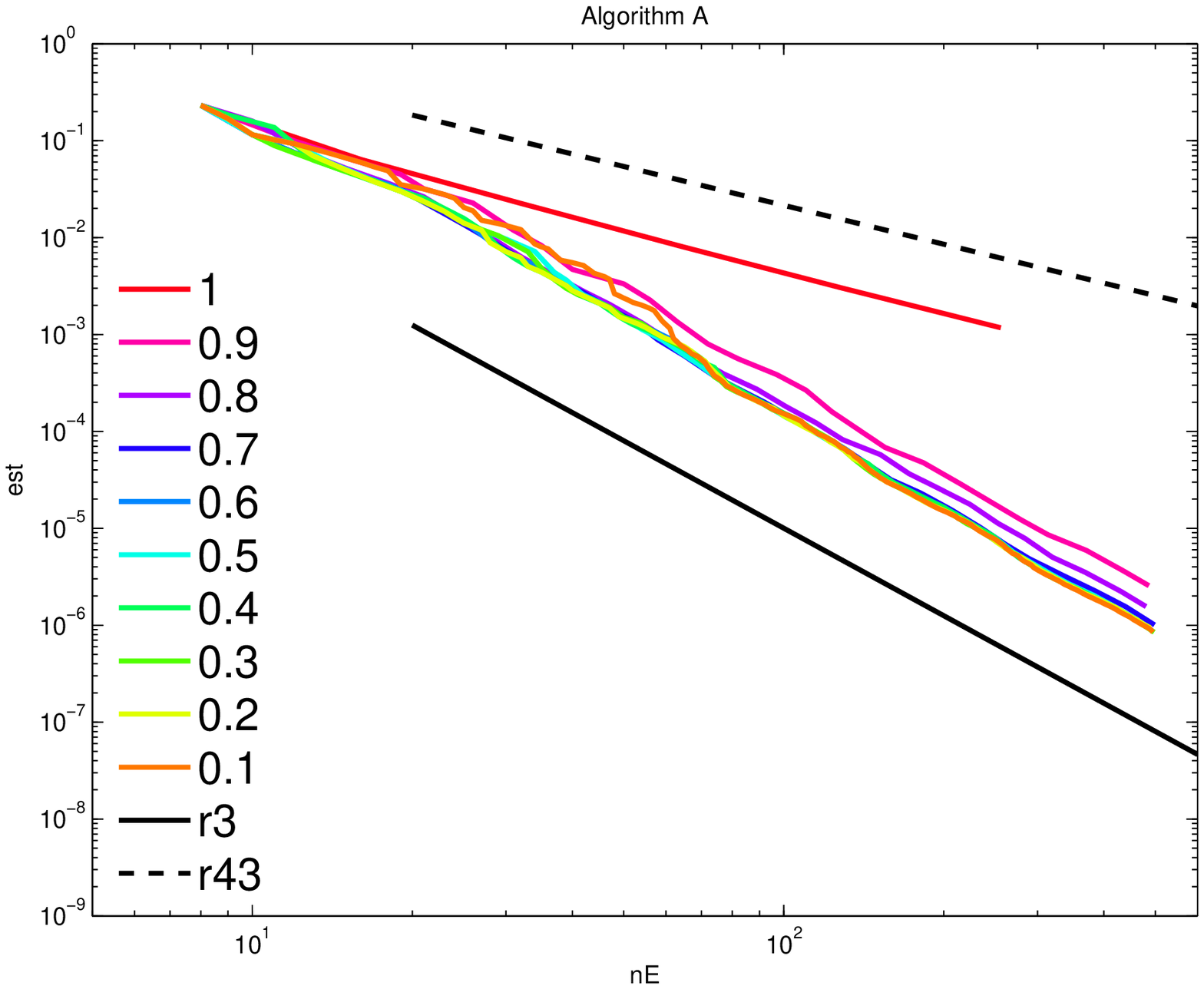}\vspace{2mm}
\psfrag{error}[cc]{\tiny error resp.\ estimators}%
  \includegraphics[scale=0.4]{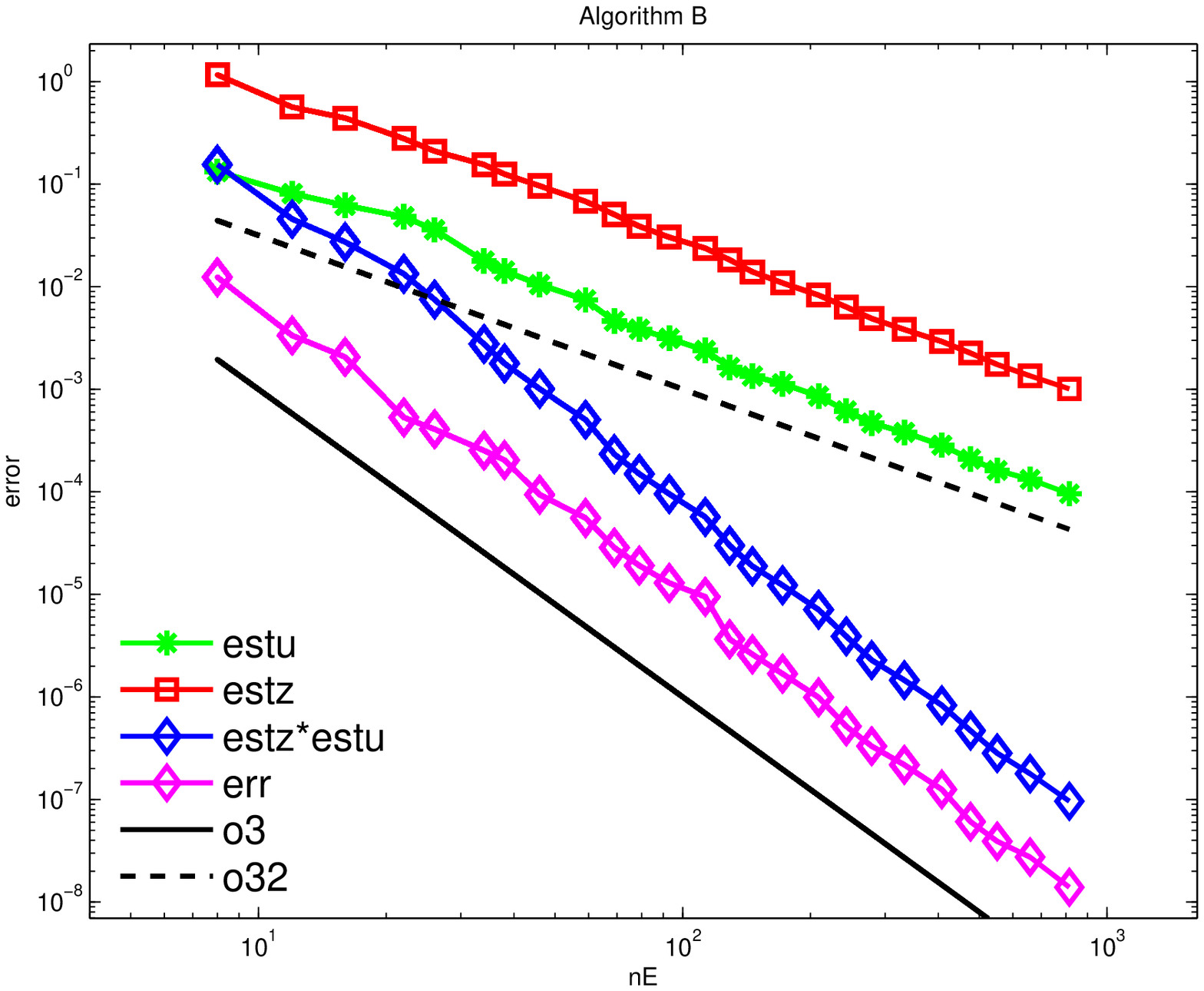} 
\psfrag{error}[cc]{\tiny error}%
  \includegraphics[scale=0.4]{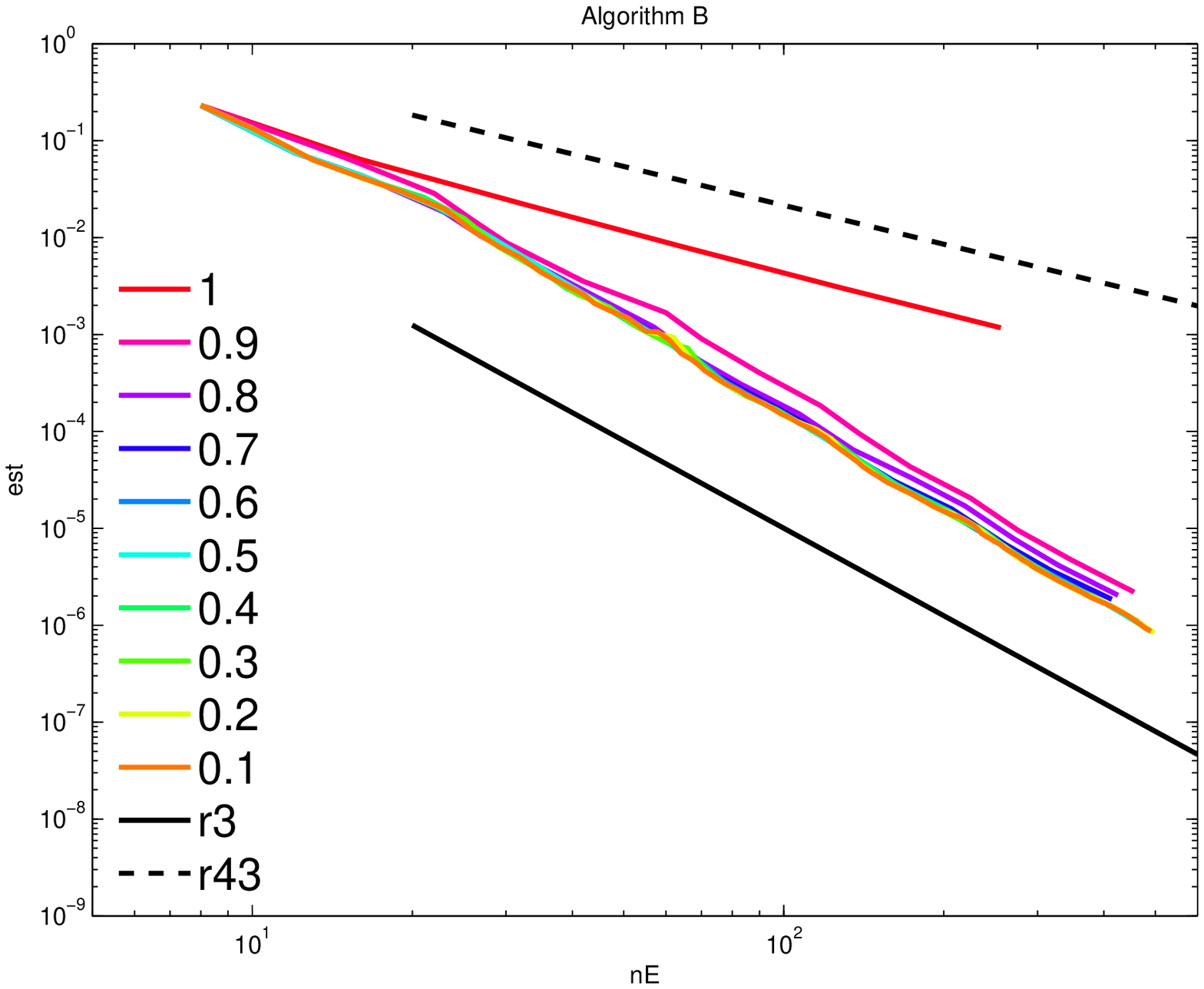}\vspace{2mm}
\psfrag{error}[cc]{\tiny error resp.\ estimators}%
  \includegraphics[scale=0.4]{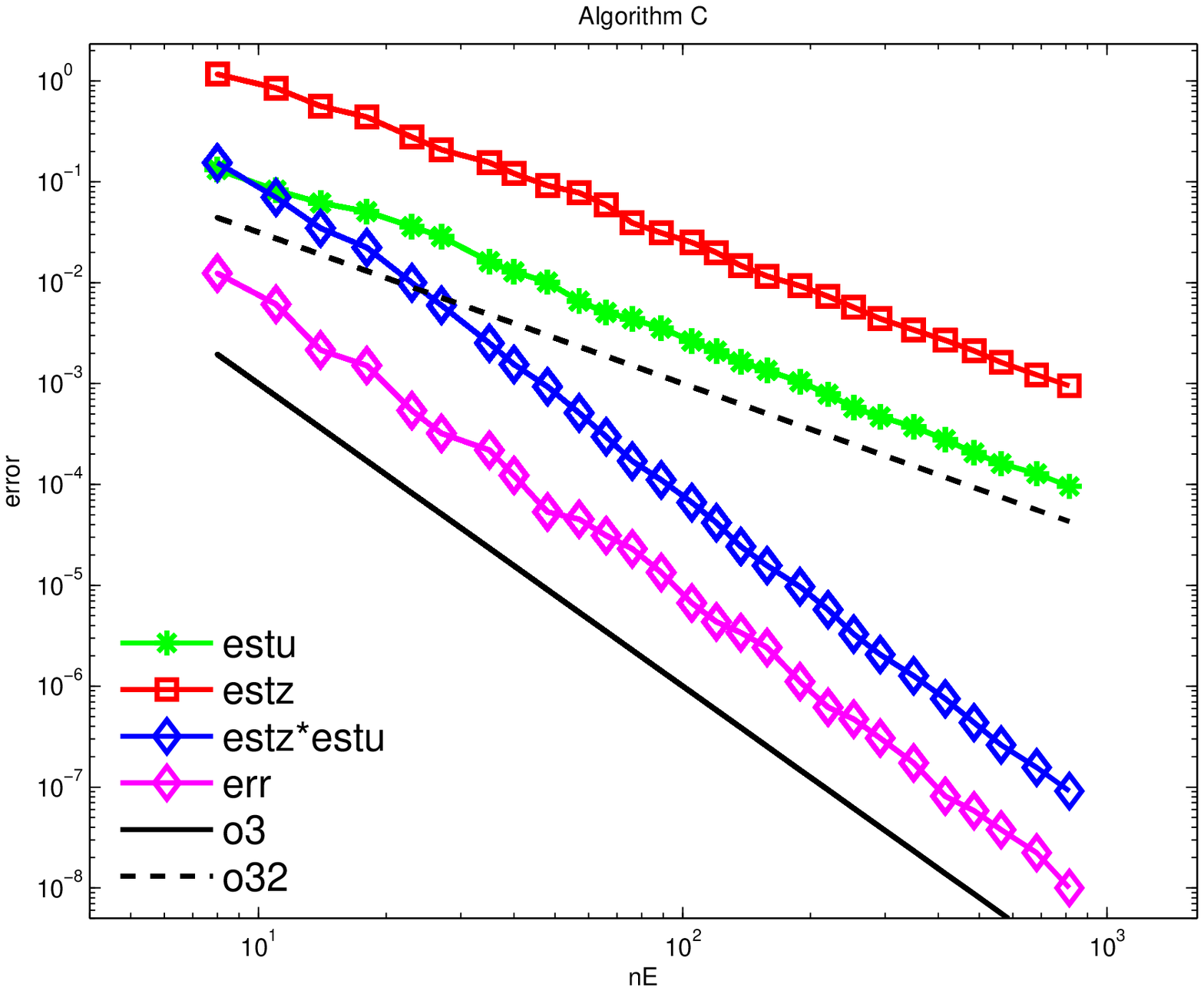} 
\psfrag{error}[cc]{\tiny error}%
  \includegraphics[scale=0.4]{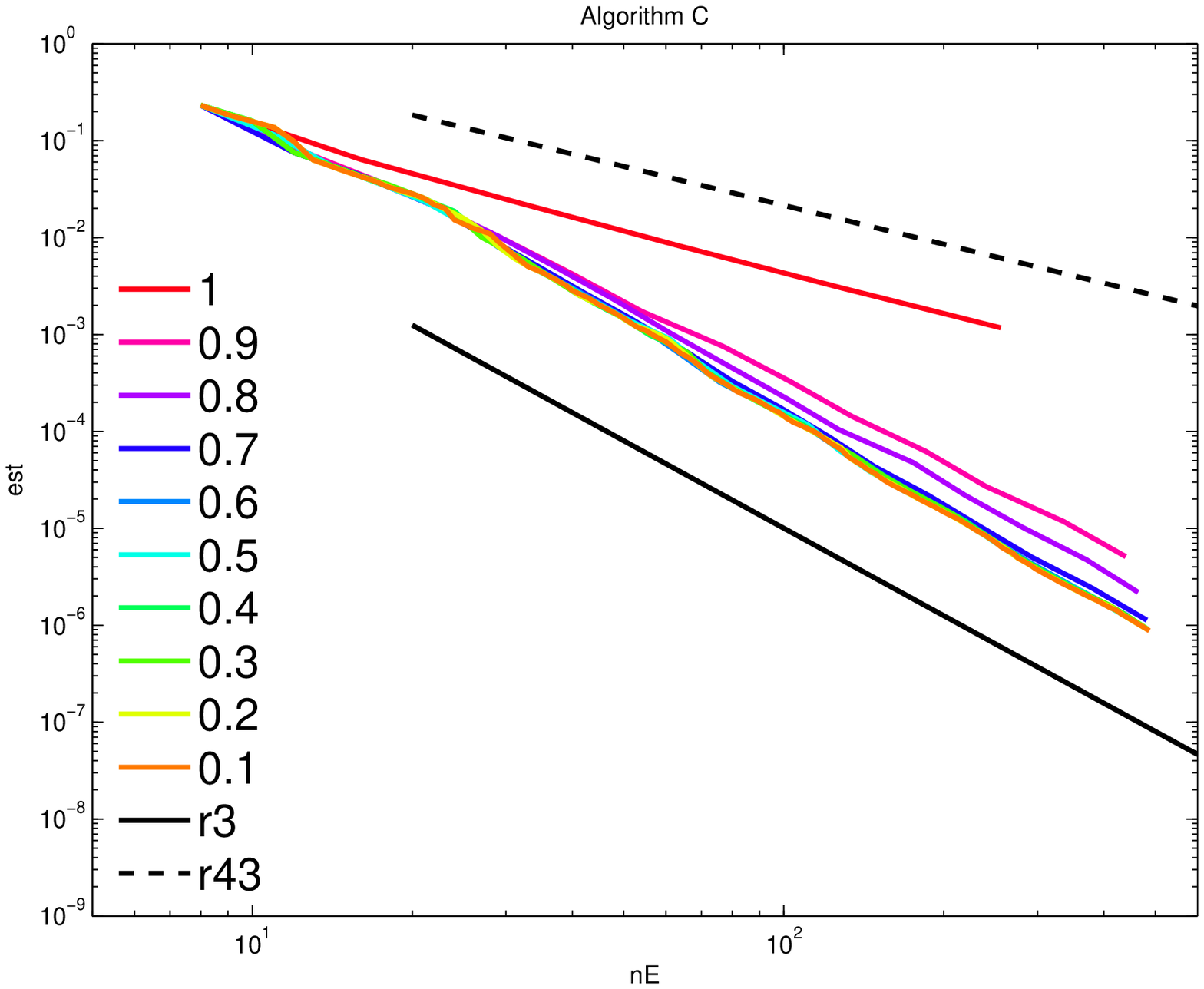}
 \caption{Example from Section~\ref{section:ex:conf} with conforming weight: 
 Over the number of elements $\#\TT_\ell$, we plot the estimators and the goal error $|g(u)-g(U_\ell)|$ as output of Algorithms~\ref{algorithm}--\ref{algorithm:bet} for $\theta=0.5$ (left) resp.\ the goal error $|g(u)-g(U_\ell)|$ for various $\theta\in\{0.1,\dots,0.9\}$ as well as for $\theta=1.0$ which corresponds to uniform refinement (right).}
 \label{fig:alg}
\end{figure}

\begin{figure}
\psfrag{A}{\scalebox{.5}{Algorithm A}}
\psfrag{B}{\scalebox{.5}{Algorithm B}}
\psfrag{C}{\scalebox{.5}{Algorithm C}}
\psfrag{error}[cc]{\tiny error}
\psfrag{theta}[cc]{\tiny parameter $\theta$}
\psfrag{ncum}[cc]{\tiny $N_{\rm cum}$}
 \includegraphics[scale=0.4]{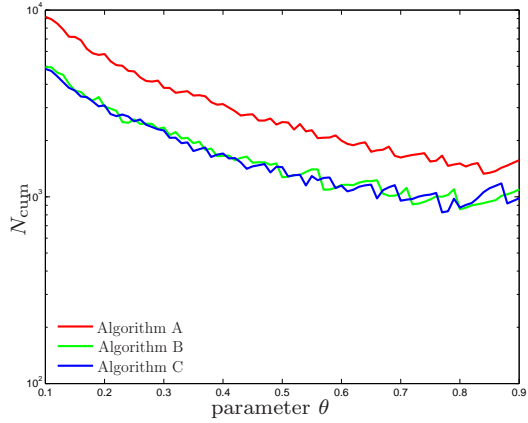}
 \caption{Example from Section~\ref{section:ex:conf} with conforming weight: Over different values of $\theta$, we plot the cumulative number of elements $N_{\rm cum}:=\sum_{k=0}^\ell\#\TT_k$ necessary for Algorithms~\ref{algorithm}--\ref{algorithm:bet} to achieve an error accuracy $|g(u)-g(U_\ell)|<10^{-6}$.}
 \label{fig:alg2}
\end{figure}

\subsection{Numerical experiment with conforming weight function}\label{section:ex:conf}
Let $\Omega\subset \R^2$ denote the $L$-shaped domain shown in Figure~\ref{fig:bemdom} which satisfies $\diam(\Omega)=1/\sqrt{2}$. On the boundary $\Gamma:=\partial\Omega$, consider the function
$\phi(x):=r^{2/3}\cos(2\alpha/3)$ for polar coordinates $r(x),\alpha(x)$ with origin $(0,0)$. Consider the model problem~\eqref{ex:bem} with
\begin{align*}
 F:=(\dlp + 1/2)\phi,
\end{align*}
where $\dlp: H^{1/2+s}(\Gamma)\to H^{1/2+s}(\Gamma)$, for all $-1/2\le s\le1/2$, denotes the double-layer potential, which is formally defined as ($n_y$ denotes the outer unit normal on $\Gamma$ at $y$)
\begin{align*}
 \dlp \phi(x):=\int_\Gamma \frac{(x-y)\cdot n_y}{|x-y|^2}\phi(y)\,dy.
\end{align*}
For these choices, it is known~\cite{hw,mclean,ss} that~\eqref{ex:bem} is equivalent to the Laplace-Dirichlet problem
\begin{align*}
 \Delta P = 0 \text{ in $\Omega$ subject to Dirichlet boundary conditions }P = \phi\text{ on }\Gamma,
\end{align*}
and the exact solution of~\eqref{ex:bem} is the normal derivative of $P$, 
\begin{align*}
 u(x)=\partial_{n_x}P(x)\quad\text{for all }x\in\Gamma.
\end{align*}

We define the initial mesh $\TT_0$ as shown in Figure~\ref{fig:bemdom}.
As weight function $\Lambda\in \SS^1(\TT_0)$, we consider the hat function defined by $\Lambda(z_0)=1$ and $\Lambda(z)=0$ for all
other nodes $z$ of $\TT_0$ (the node $z_0$ is indicated in Figure~\ref{fig:bemdom}).

For the lowest-order case $p=0$ and $\theta=0.5$ in Algorithm~\ref{algorithm}--\ref{algorithm:bet},
Figure~\ref{fig:alg} shows the convergence rates of the error estimators $\eta_u$, $\eta_z$, their product $\eta_u\eta_z$, and the error in the goal functional $|g(u)-g(U_\ell)|$. Moreover, 
we compare the convergence rate of the error in the goal functional for different values of $\theta\in\{0.1,\dots,0.9\}$. For either choice of $\theta$ and all adaptive algorithms, we observe the optimal convergence rate $(\#\TT_\ell)^{-3/2}$ for the respective error estimators as well as $(\#\TT_\ell)^{-3}$ for the error in the goal functional. 

For different values of $\theta\in\{0.1,\dots,0.9\}$, Figure~\ref{fig:alg2} plots the cumulative number of elements $N_{\rm cum}:=\sum_{k=0}^\ell\#\TT_k$ necessary to reach a given error tolerance $10^{-6}$. We observe that for all three algorithms a large $\theta\approx 0.8$ seems to be optimal. Moreover, Algorithms~\ref{algorithm:mod}--\ref{algorithm:bet} show comparable performance which is clearly superior to that of Algorithm~\ref{algorithm} in the whole range of $\theta$.

\begin{figure}
\psfrag{Algorithm A}[cc][cc]{\tiny Algorithm A}
\psfrag{Algorithm B}[cc][cc]{\tiny Algorithm B}
\psfrag{Algorithm C}[cc][cc]{\tiny Algorithm C}
\psfrag{estu}{\scalebox{.5}{$\eta_u$}}
\psfrag{estz}{\scalebox{.5}{$\eta_z$}}
\psfrag{estz*estu}{\scalebox{.6}{$\eta_u\eta_z$}}
\psfrag{o3}{\scalebox{.5}{$\OO(N^{-3})$}}
\psfrag{o32}{\scalebox{.5}{$\OO(N^{-3/2})$}}
\psfrag{err}{\scalebox{.5}{error}}
\psfrag{error}[cc]{\tiny error resp.\ estimators}
\psfrag{nE}[cc]{\tiny number of elements $N=\#\TT_\ell$}
 \includegraphics[scale=0.4]{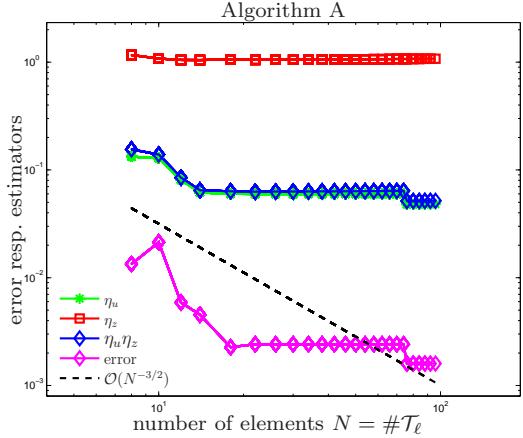}
 \caption{Example from Section~\ref{section:ex:nonconf} with non-conforming weight: Counterexample to show that rescaling is necessary. We plot the output of Algorithm~\ref{algorithm} for $\theta=0.5$ without rescaling of the estimators, i.e., $\eta_u=\eta_u^\eps$ and $\eta_z=\eta_u^\eps$ with $\eps=0$. We do not observe convergence at all.}
 \label{fig:counter}
\end{figure}

\begin{figure}
\psfrag{z}{\tiny$z_0$}
\psfrag{x}{\tiny$x$}
\psfrag{y}{\tiny$y$}
\psfrag{dual}{\scalebox{.5}{dual}}
\psfrag{primal}{\scalebox{.5}{primal}}
\psfrag{parameter domain}[cc][cc]{\tiny parameter domain}
\psfrag{nE = 299}[cc][cc]{\tiny Algorithm~\ref{algorithm:mod}, $\#\TT_{25}=299$}
   \includegraphics[scale=0.4]{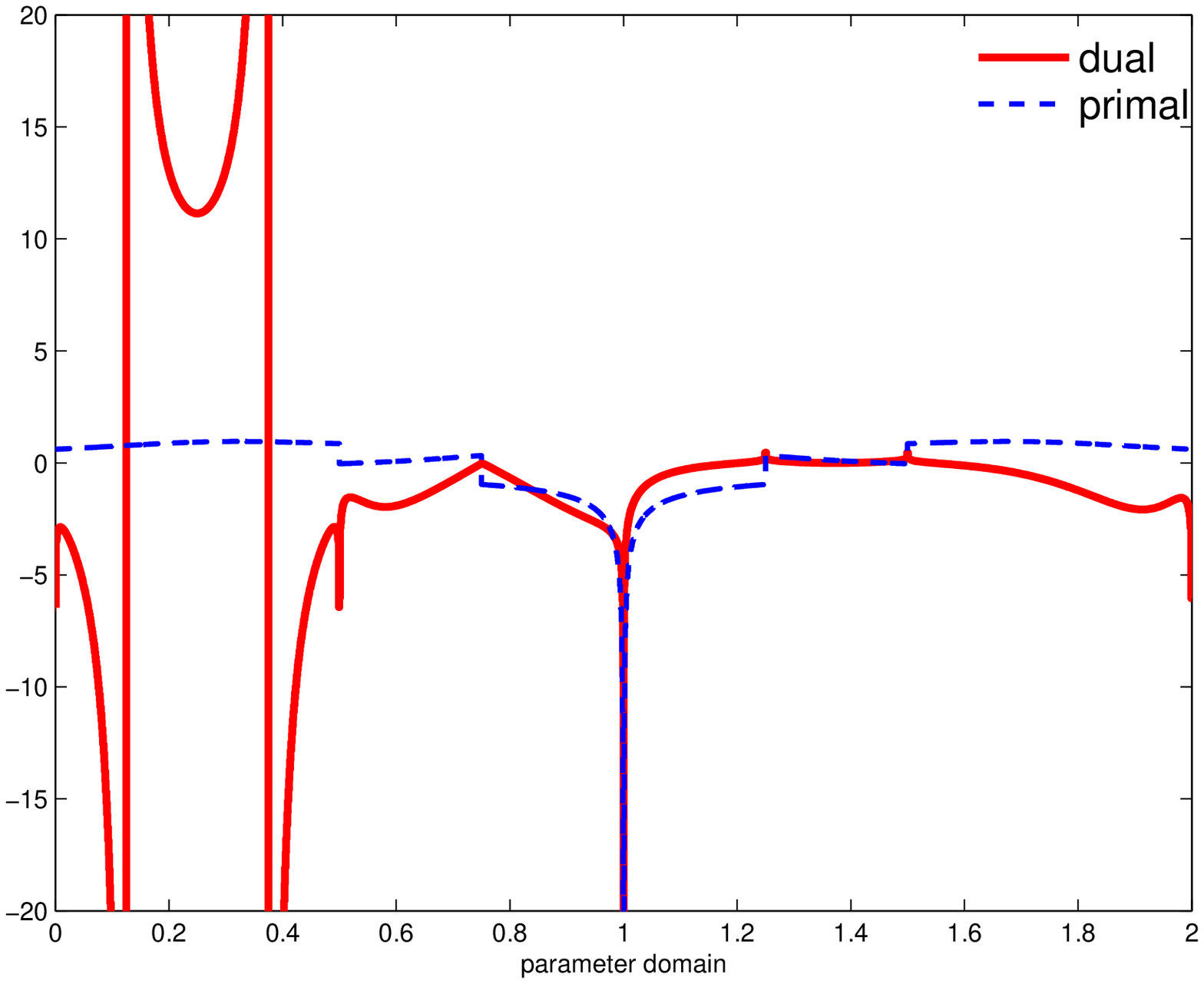}
    \includegraphics[scale=0.4]{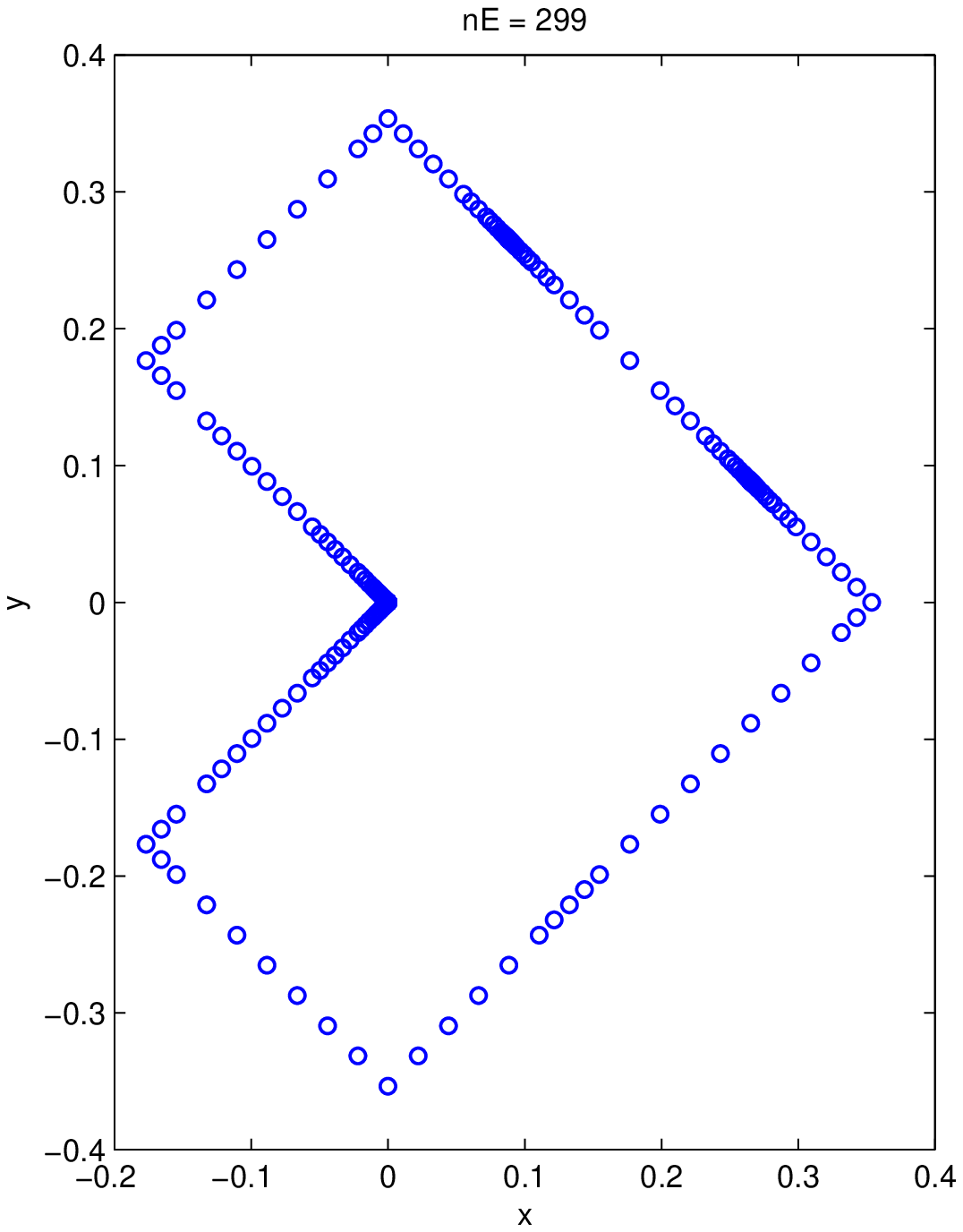}
 \caption{Example from Section~\ref{section:ex:nonconf} with non-conforming weight: Piecewise constant approximation of primal and dual solution plotted over the arc-length (left), and adaptive mesh with $\#\TT_{25} = 299$ elements generated by Algorithm~\ref{algorithm:mod} for $\theta=0.5$ and with the rescaled estimators~\eqref{ex:bem:estimator:primalrescaled} with $\eps=0.3$ (right).}
 \label{fig:bemdom2}
\end{figure}

\begin{figure}
\psfrag{0.1}{\scalebox{.5}{$0.1$}}
\psfrag{0.2}{\scalebox{.5}{$0.2$}}
\psfrag{0.3}{\scalebox{.5}{$0.3$}}
\psfrag{0.4}{\scalebox{.5}{$0.4$}}
\psfrag{0.5}{\scalebox{.5}{$0.5$}}
\psfrag{0.6}{\scalebox{.5}{$0.6$}}
\psfrag{0.7}{\scalebox{.5}{$0.7$}}
\psfrag{0.8}{\scalebox{.5}{$0.8$}}
\psfrag{0.9}{\scalebox{.5}{$0.9$}}
\psfrag{1}{\scalebox{.5}{$1.0$}}
\psfrag{Algorithm A}[c][c]{\tiny Algorithm A}
\psfrag{Algorithm B}[c][c]{\tiny Algorithm B}
\psfrag{Algorithm C}[c][c]{\tiny Algorithm C}
\psfrag{estu}{\scalebox{.5}{$\eta_u^\eps$}}
\psfrag{estz}{\scalebox{.5}{$\eta_z^\eps$}}
\psfrag{estz*estu}{\scalebox{.6}{$\eta_u^\eps\eta_z^\eps$}}
\psfrag{o3}{\scalebox{.5}{$\OO(N^{-3})$}}
\psfrag{o32}{\scalebox{.5}{$\OO(N^{-3/2})$}}
\psfrag{r3}{\scalebox{.5}{$\OO(N^{-3})$}}
\psfrag{r12}{\scalebox{.5}{$\OO(N^{-1/2})$}}
\psfrag{r32}{\scalebox{.5}{$\OO(N^{-3/2})$}}
\psfrag{r43}{\scalebox{.5}{$\OO(N^{-4/3})$}}
\psfrag{err}{\scalebox{.5}{error}}
\psfrag{nE}[cc]{\tiny number of elements $N=\#\TT_\ell$}
\psfrag{error}[cc]{\tiny error resp.\ estimators}
 \includegraphics[scale=0.4]{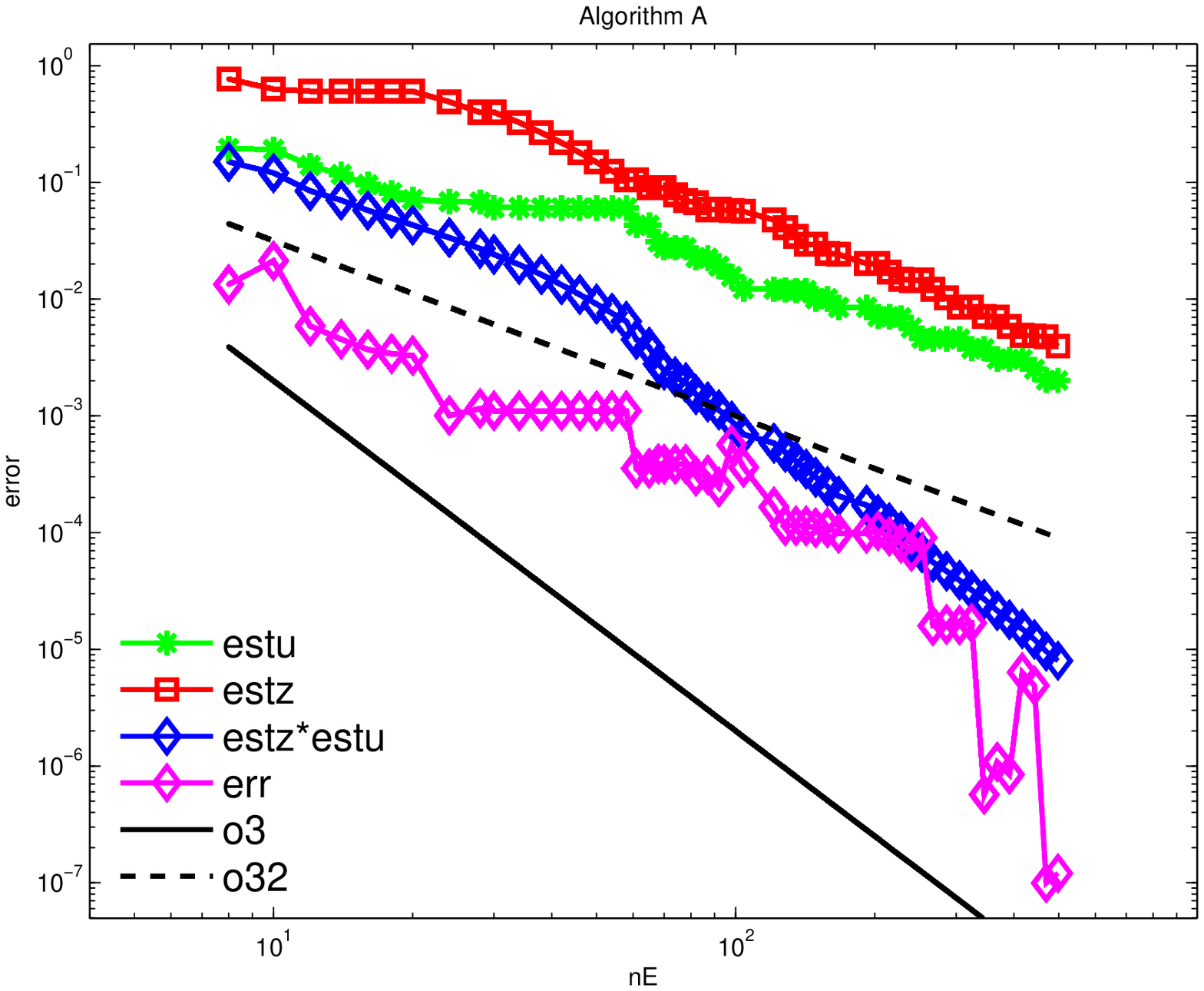} 
\psfrag{error}[cc]{\tiny error}
 \includegraphics[scale=0.4]{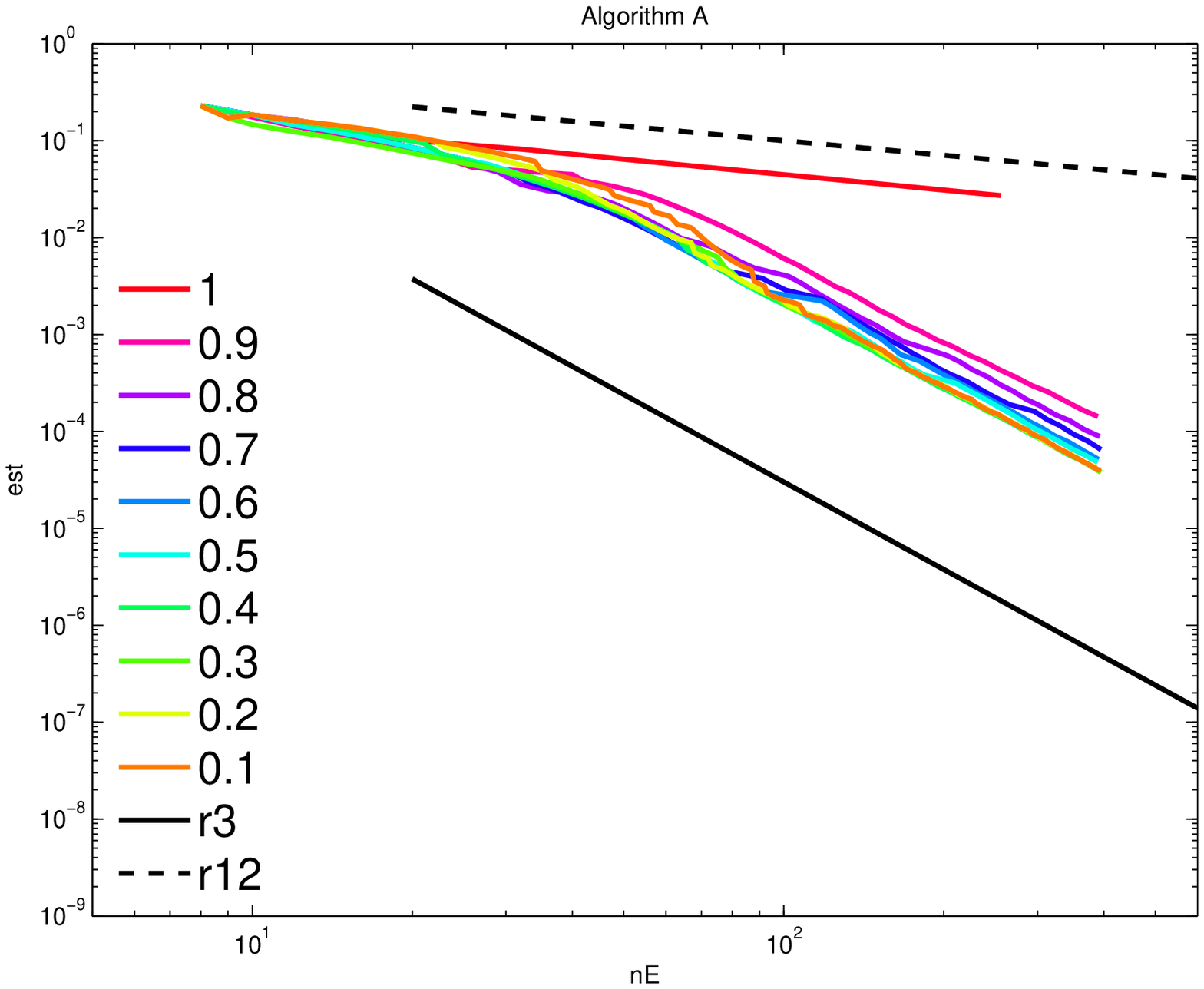}\vspace{2mm}
\psfrag{error}[cc]{\tiny error resp.\ estimators}
  \includegraphics[scale=0.4]{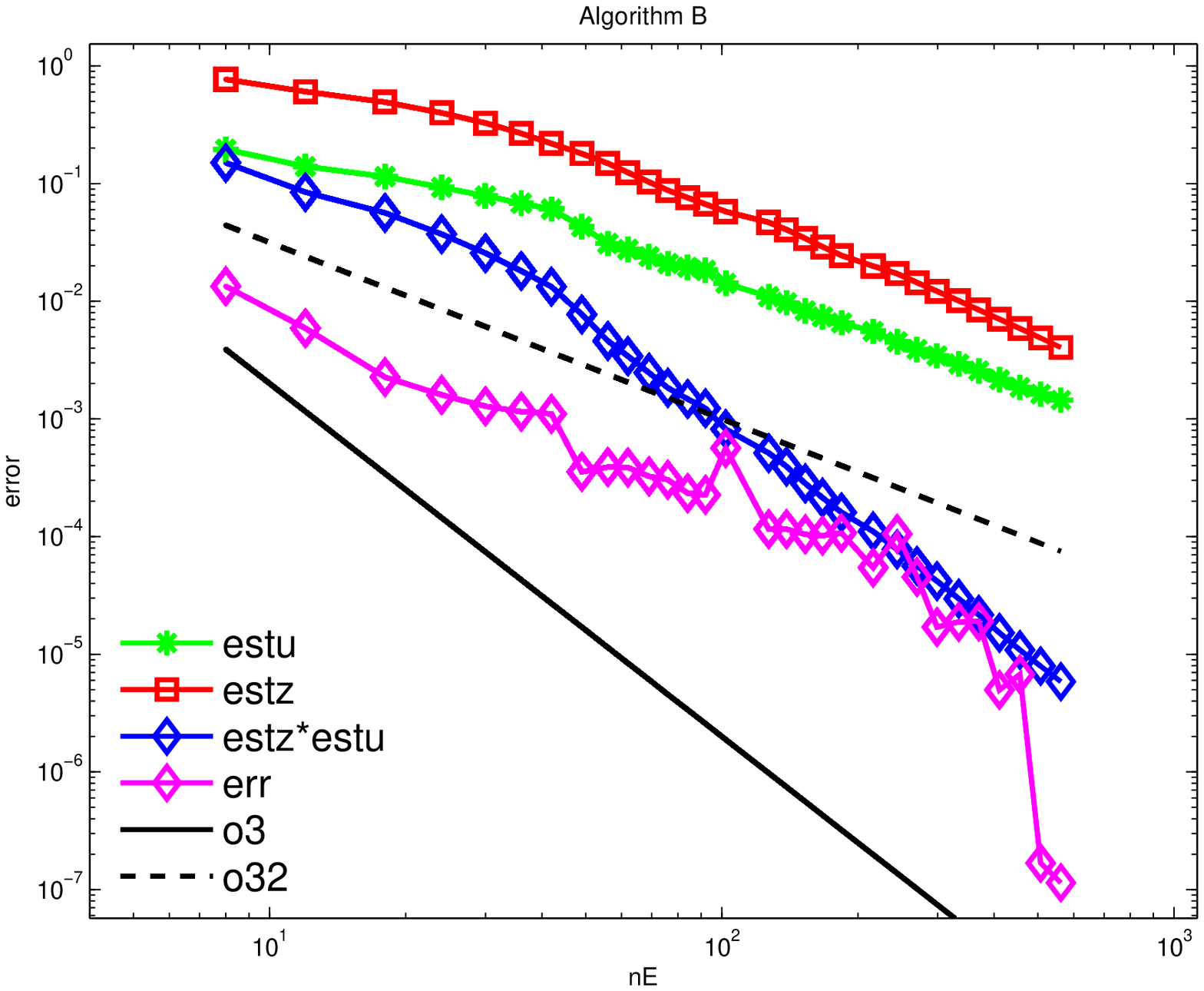} 
\psfrag{error}[cc]{\tiny error}
  \includegraphics[scale=0.4]{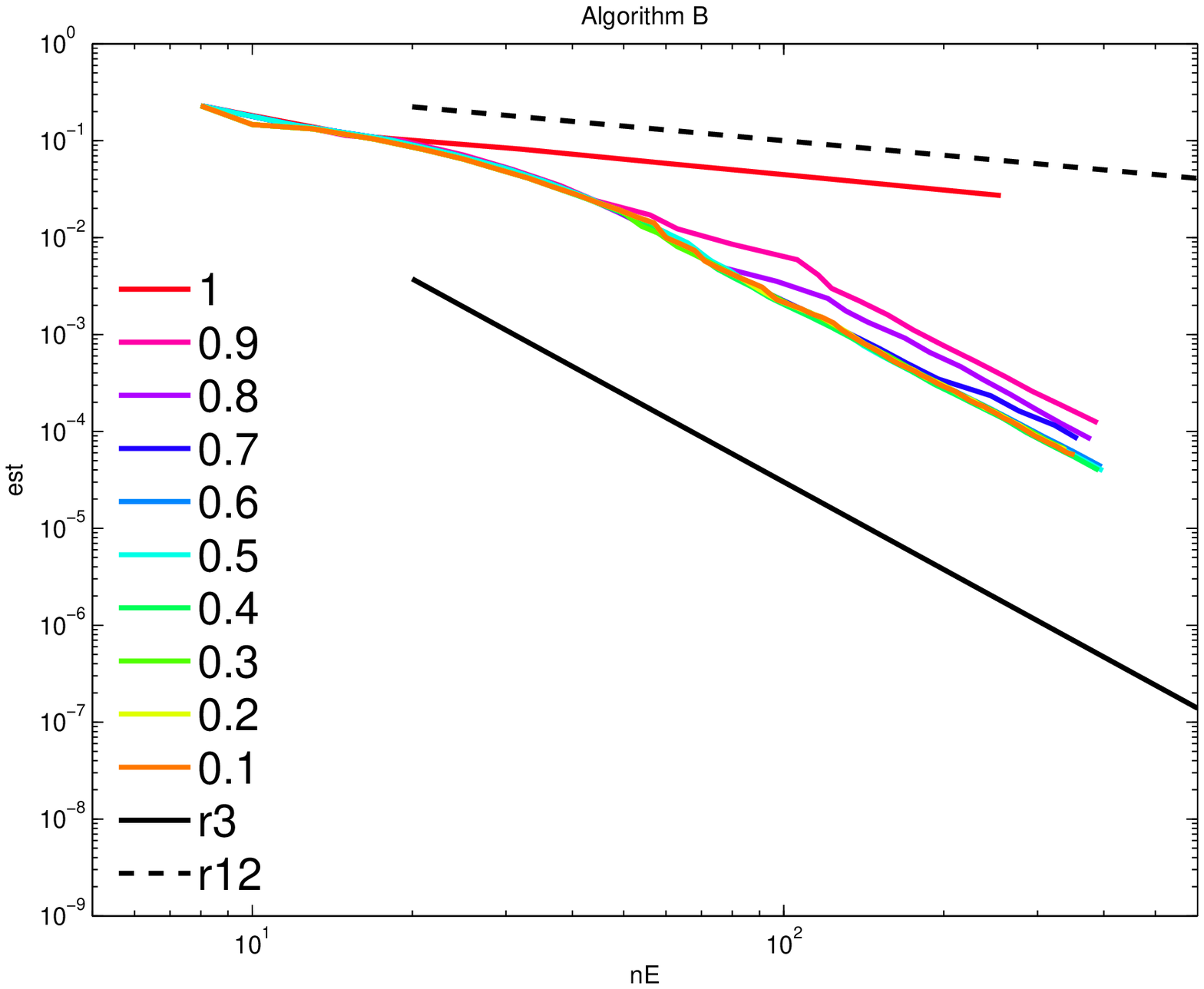}\vspace{2mm}
\psfrag{error}[cc]{\tiny error resp.\ estimators}
  \includegraphics[scale=0.4]{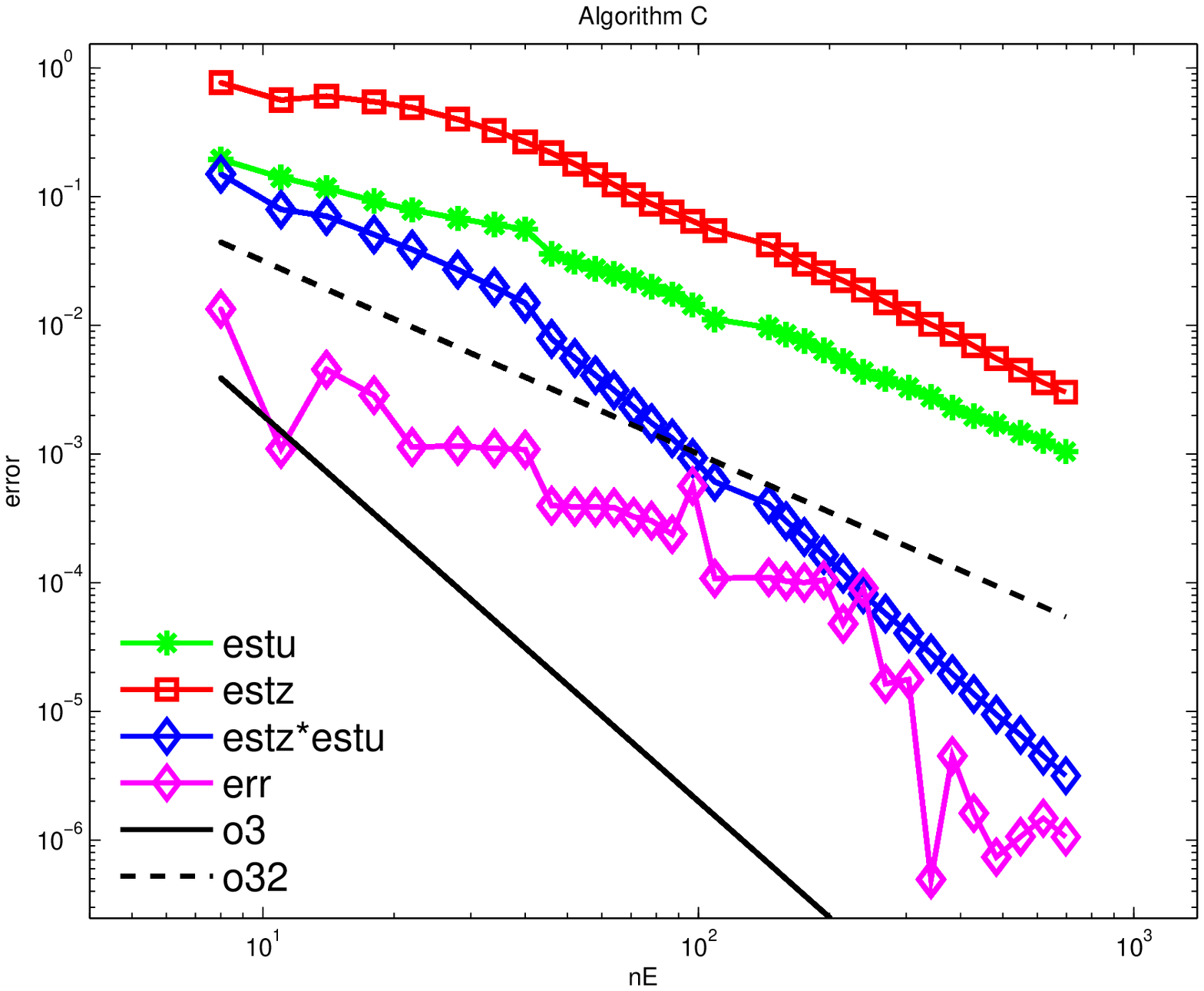} 
\psfrag{error}[cc]{\tiny error}
  \includegraphics[scale=0.4]{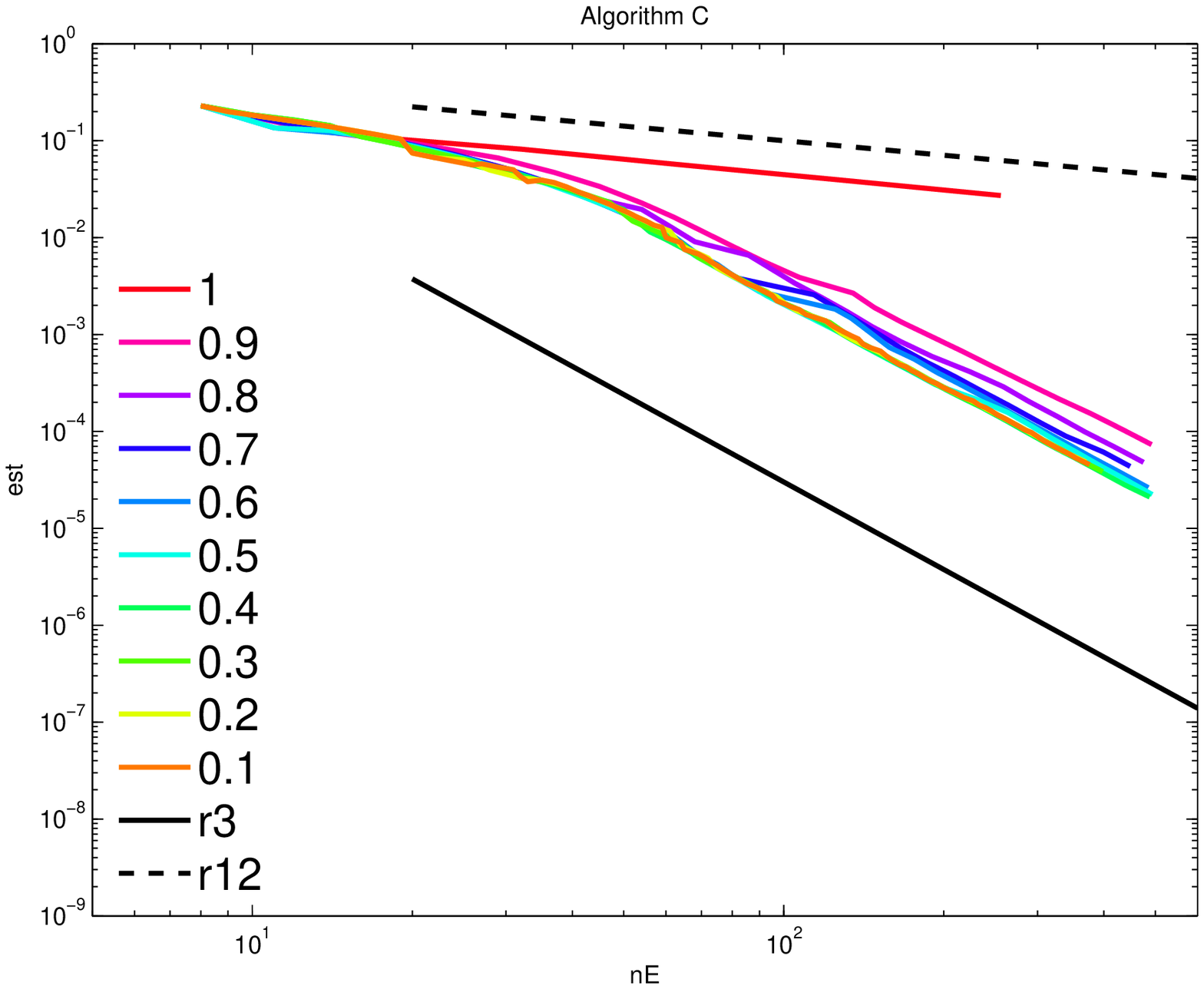}
 \caption{Example from Section~\ref{section:ex:nonconf} with non-conforming weight: 
  Over the number of elements $\#\TT_\ell$, we plot the estimators and the goal error $|g(u)-g(U_\ell)|$ as output of Algorithms~\ref{algorithm}--\ref{algorithm:bet} for $\theta=0.5$ (left) resp.\ the goal error $|g(u)-g(U_\ell)|$ for various $\theta\in\{0.1,\dots,0.9\}$ as well as for $\theta=1.0$ which corresponds to uniform refinement (right), where we use the rescaled estimators~\eqref{ex:bem:estimator:primalrescaled} with $\eps=0.3$.}
 \label{fig:algmod}
\end{figure}

\begin{figure}
\psfrag{A}{\scalebox{.5}{Algorithm A}}
\psfrag{B}{\scalebox{.5}{Algorithm B}}
\psfrag{C}{\scalebox{.5}{Algorithm C}}
\psfrag{error}[cc]{\tiny error}
\psfrag{theta}[cc]{\tiny parameter $\theta$}
\psfrag{ncum}[cc]{\tiny $N_{\rm cum}$}
 \includegraphics[scale=0.5]{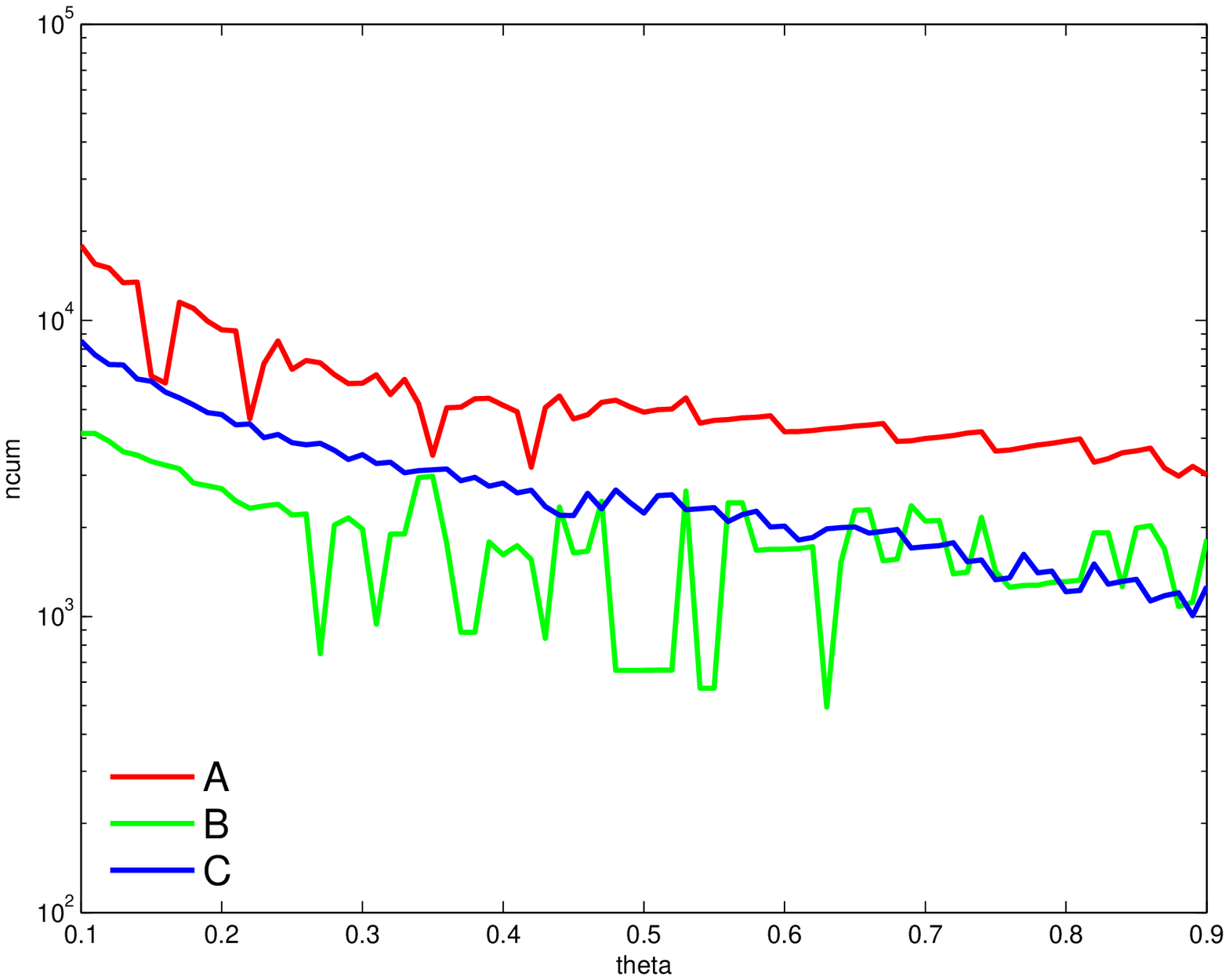}
 \caption{Example from Section~\ref{section:ex:nonconf} with non-conforming weight: Over different values of $\theta$, we plot the cumulative number of elements $N_{\rm cum}:=\sum_{k=0}^\ell\#\TT_k$ necessary for Algorithms~\ref{algorithm}--\ref{algorithm:bet} to achieve an error accuracy $|g(u)-g(U_\ell)|<10^{-6}$. We use the rescaled estimators~\eqref{ex:bem:estimator:primalrescaled} with $\eps=0.3$.}
 \label{fig:algmod2}
\end{figure}

\begin{figure}
\psfrag{nEcum}{\scalebox{.5}{$N_{\rm cum}$}}
\psfrag{theta}[bc][tc]{\scalebox{.5}{$\theta$}}
\psfrag{par}[bc][tc]{\scalebox{.5}{$\eps$}}
 \includegraphics[scale=0.5]{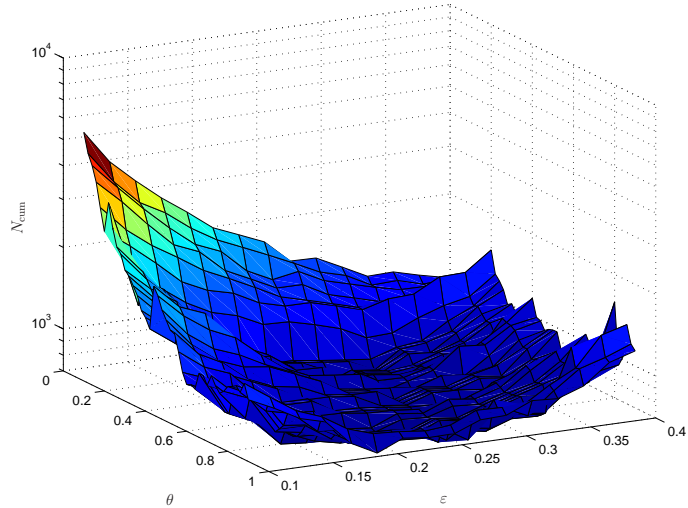}
 \caption{Example from Section~\ref{section:ex:nonconf} with non-conforming weight: We plot the cumulative number of elements $N_{\rm cum}$ necessary to reach the tolerance $\eta_{u,\ell}\eta_{z,\ell}\leq 10^{-2}$ for different values of $\theta\in\{0.1,\ldots,0.9\}$ and different scaling parameters $\eps\in\{0.1,\ldots,0.4\}$. We observe that $\eps\approx 0.3$ gives the best performance.}
 \label{fig:blabla}
\end{figure}

\subsection{Numerical experiment with non-conforming weight function}\label{section:ex:nonconf}
We consider the same setting as in Section~\ref{section:ex:conf}, with the only difference
that $\Lambda$ is the characteristic function of $\Gamma_\Lambda\subsetneqq \Gamma$, i.e., $\Lambda(x)=1$ on $\Gamma_\Lambda$ and $\Lambda(x) = 0$ on $\Gamma\setminus\Gamma_\Lambda$. We choose $\Gamma_\Lambda\subsetneqq \Gamma$ as the part of $\Gamma$ which is marked in \textcolor{red}{red} in Figure~\ref{fig:bemdom}. This implies that the goal functional takes the form
\begin{align*}
g(u):=\int_{\Gamma_\Lambda}u\,ds.
\end{align*}
Note that $\Lambda\notin H^{1/2}(\Gamma)\supset H^{1}(\Gamma)$, but only $\Lambda \in H^{1/2-\eps}(\Gamma)$ for all $\eps>0$. 
In particular, $g\notin H^{-1/2}(\Gamma)^*=H^{1/2}(\Gamma)$.
Consequently, this example is \emph{not covered} by the theory of the previous sections. This is also reflected by the numerical results, if the adaptive algorithms are naively employed; see Figure~\ref{fig:counter}, where we do not observe convergence at all.

To account for the fact that $\Lambda\notin H^1(\Gamma)$, we approximate $\Lambda$ in each adaptive step by the continuous function $\Lambda_\ell\in\SS^1(\TT_\ell)$ defined by $\Lambda_\ell(z)=1$ for all nodes $z$ of $\TT_\ell$ with $z\in \Gamma_\Lambda$ and $\Lambda_\ell(z)=0$ for all other nodes. Convergence $\Lambda_\ell\to\Lambda$ is assured by marking of the two elements where $\Lambda_\ell$ is not constant in each adaptive step. Since there clearly holds $\norm{\Lambda_\ell}{H^{1/2}(\Gamma)}\to \infty$ as $\ell\to\infty$, we need to rescale the error estimators for the primal and the dual problem, respectively. Given $\eps>0$, define
\begin{align}\label{ex:bem:estimator:primalrescaled}
\eta_{u,\ell}^\eps(T)^2:=h_T^{1-\eps}\norm{\nabla (\slp U_\ell-F)}{L^2(T)}^2 \quad\text{and}\quad 
\eta_{z,\ell}^\eps(T)^2:=h_T^{1+\eps}\norm{\nabla (\slp Z_\ell-\Lambda_\ell)}{L^2(T)}^2.
\end{align}
Since a thorough analysis is beyond the scope of this paper, 
we only provide a heuristic motiviation for this rescaling: With $\Lambda_\ell\approx\Lambda$, the error in the goal functional is estimated by
\begin{align*}
|g(u)-g(U_\ell)|\approx|\int_\Gamma (u-U_\ell)(\Lambda_\ell-\slp Z_\ell)\,ds|&\leq \norm{u-U_\ell}{H^{-1/2+\eps}(\Gamma)}\norm{\Lambda_\ell -\slp Z_\ell}{H^{1/2-\eps}(\Gamma)}\\
&\lesssim \eta_{u,\ell}^\eps\eta_{z,\ell}^\eps.
\end{align*}
Since $\sup_{\ell\in\N}\norm{\Lambda_\ell}{H^{1/2-\eps}(\Gamma)}<\infty$, the last estimate is even rigorous and follows from appropriate Poincar\'e inequalities; see, e.g.,~\cite{c97,cms}.

For $\eps=0.3$, $\theta=0.5$, and lowest order BEM $p=0$,
Figure~\ref{fig:algmod} shows the convergence rates of the error estimators $\eta_u$, $\eta_z$, their product $\eta_u\eta_z$, and the error in the goal functional $|g(u)-g(U_\ell)|$. Moreover, we compare the convergence rates of the error in the goal functional for different values of $\theta\in\{0.1,\dots,0.9\}$. Except for $\theta=0.9$ and Algorithm~\ref{algorithm:bet}, we observe for either choice of $\theta$ and all adaptive algorithms the optimal convergence rate $(\#\TT_\ell)^{-3}$ for the error in the goal functional as well as the estimator product. 

Figure~\ref{fig:algmod2} plots the cumulative number of elements $N_{\rm cum}=\sum_{k=0}^\ell\#\TT_k$ necessary to reach a given error tolerance $10^{-6}$ for different values of $\theta\in\{0.1,\dots,0.9\}$. We observe that for Algorithms~\ref{algorithm}-\ref{algorithm:bet} a large $\theta\approx 0.8$ seems to be optimal, whereas Algorithm~\ref{algorithm:mod} shows optimal behavior for $0.3\leq \theta\leq 0.7$. Overall, Algorithm~\ref{algorithm:mod} seems to be the best choice in this experiment.

\section{Conclusions \& Open Questions}
\label{section:conclusion}

\subsection{Analytical results}
We have derived an abstract framework to prove convergence with optimal algebraic rates for goal-oriented adaptivity for finite element methods and boundary element methods. While the analysis of prior works~\cite{bet,ms} was tailored to the Poisson model problem resp.\ symmetric boundary integral formulations~\cite{pointabem}, our approach which is inspired by~\cite{axioms}, is {\sl a~priori} independent of the model problems and covers general linear second-order elliptic PDEs and fixed order elements in the frame of the Lax-Milgram lemma. Following~\cite{ckns}, our argument avoids the discrete efficiency and hence the interior node property of the mesh-refinement required in~\cite{bet,ms}. Following~\cite{ffp}, our argument uses the concept of a general quasi-orthogonality which allows to work beyond symmetric problems and, unlike~\cite{mn,hp14}, to avoid any assumption on the initial mesh $\TT_0$. As firstly observed in~\cite{dirichlet3d} and later used in~\cite{ffp,axioms}, the convergence and quasi-optimality analysis relies essentially only on reliability of the error estimator (see axioms~\eqref{ass:stable}--\eqref{ass:reliable}), while efficiency is only used to characterize the estimator-based approximation classes in terms of the so-called total error, i.e., error plus data oscillations (see Lemma~\ref{lemma:approximationclass}). In addition to the algorithm from~\cite{ms} (Algorithm~\ref{algorithm}), we gave a thorough analysis for the algorithm from~\cite{bet} (Algorithm~\ref{algorithm:bet}) without additional assumptions on the given data. Moreover, we proposed a variant of the algorithms from~\cite{ms} and~\cite{hp14} (Algorithm~\ref{algorithm:mod}). All three algorithms are proved to be linearly convergent with optimal algebraic rates (see Theorem~\ref{theorem:linear}, \ref{theorem:optimal}, \ref{theorem:optimal:mod}, \ref{theorem:optimal:bet}), where theory guarantees linear convergence for all marking parameters $0<\theta\le1$, while optimal convergence rates are qualitatively guaranteed for $0<\theta<\theta_\star$ (Algorithm~\ref{algorithm}--\ref{algorithm:mod}) resp.\ $0<\theta<\theta_\star/2$ (Algorithm~\ref{algorithm:bet}) for some {\sl a~priori} bound $0<\theta_\star<1$ which depends on the given problem.

\subsection{Empirical results}
To underline our analysis, we considered three different problems: First (Section~\ref{example1:afem}), we computed an example from~\cite{ms} which considers finite elements for the Poisson model problem with some right-hand side $f = {\rm div}\, \boldsymbol{f}$ and goal function $g = {\rm div}\, \boldsymbol{g}$ for some piecewise constant vector fields $\boldsymbol{f},\boldsymbol{g}:\Omega\to\R^d$. Essentially for all choices of adaptivity parameters $0.1 \le \theta \le 0.9$, we observed optimal convergence behavior of the goal-oriented adaptive algorithms, while standard adaptivity leads to a reduced order of convergence. Second (Section~\ref{section:example:afem2}), we modified an example from~\cite{mn} with a non-symmetric operator, where the goal is the evaluation of the flux for some finite element computation. All goal-oriented adaptive algorithms are robust with respect to the choice of $0.1 \le \theta \le 0.9$. Finally (Section~\ref{section:ex:conf}), we considered an example in the frame of the boundary element method, where the goal was some local flux evaluation. Again, all goal-oriented adaptive algorithms are robust with respect to the choice of $0.1 \le \theta \le 0.9$ and lead to optimal convergence behavior. Throughout, our observation was that the new algorithm (Algorithm~\ref{algorithm:mod}) leads to the best results with respect to the cumulative sum of elements~\eqref{eq:Ncum} which seems to be an appropriate measure for the overall computational performance to reach a prescribed accuracy. Although we did not observe that Algorithm~\ref{algorithm:bet} leads to suboptimal convergence rates for large $\theta$, where Algorithm~\ref{algorithm} and~\ref{algorithm:mod} still are optimal, we note that this has been observed in~\cite{pointabem} for the point evaluation in boundary element computations which is a linear and continuous functional (and hence an advantage) of boundary integral formulations.

\subsection{Extensions \& open questions} 
First, following the work of Mommer \& Stevenson~\cite[Section~7.1]{ms}, it is possible to use the extraction framework to apply our convergence and quasi-optimality results to compute point values. 
Second, arguing along the lines of~\cite{ffp}, we think that it is possible to include (at least certain) nonlinear goal functional and nonlinear PDEs based on strongly monotone operators. As in~\cite{ffp}, we note that the proof of stability~\eqref{ass:stable} and reduction~\eqref{ass:reduction} might be challenging for higher-order elements $p\ge2$, since even optimality results for standard AFEM for nonlinear problems are restricted to the lowest-order case $p=1$; see, e.g.,~\cite{MR2911397,axioms,csw,ffp,MR2911871}. Finally and for the ease of presentation, we focussed on (homogeneous) Dirichlet conditions throughout our experiments. We note that the extension to mixed Dirichlet-Neumann-Robin boundary conditions is easily possible; see~\cite{dirichlet3d,axioms,dirichlet2d} in the frame of standard AFEM. However, we note that our analysis currently requires that the Dirichlet data belongs to the coarsest trace space $\SS^1(\TT_0|_\Gamma)$. The main reason is that our analysis uses that the difference of solution and FEM approximation, i.e., $u-U_\ell$ for the primal problem resp.\ $z-Z_\ell$ for the dual problem, is an admissible test function. The latter fails for general inhomogeneous Dirichlet conditions. We believe that the rigorous  analysis of this problem is beyond the current work and requires further ideas beyond those of standard AFEM~\cite{dirichlet3d,axioms,dirichlet2d}. 


\color{black}
\bibliographystyle{siam}
\bibliography{literature}

\end{document}